\documentclass[namedate,webpdf,imanum]{ima-authoring-template}

\graphicspath{{Fig/}}

\theoremstyle{thmstyletwo}%
\newtheorem{theorem}{Theorem}[section]
\newtheorem{proposition}[theorem]{Proposition}%
\newtheorem{lemma}[theorem]{Lemma}
\newtheorem{cor}[theorem]{Corollary}
\newtheorem{assumption}[theorem]{Assumption}
\theoremstyle{remark}
\newtheorem{remark}{Remark}%
\theoremstyle{definition}
\newtheorem{definition}{Definition}
\newtheorem{example}{Example}%

\numberwithin{equation}{section}

\usepackage{upgreek} 
\usepackage{mathrsfs} 
\usepackage{pifont}
\usepackage{tikz}
\usepackage{gnuplot-lua-tikz}
\usepackage{url}
\newcommand{\domD}{\mathscr{D}}

\renewcommand{\pi}{\uppi}
\renewcommand{\Re}{\operatorname{Re}}
\renewcommand{\Im}{\operatorname{Im}}
\DeclareMathOperator{\erf}{erf}
\DeclareMathOperator{\Si}{Si}

\DeclareMathOperator{\sinc}{sinc}
\DeclareMathOperator{\OO}{O}
\DeclareMathOperator{\E}{e}
\DeclareMathOperator{\I}{i}
\DeclareMathOperator{\diag}{diag}
\newcommand{\SE}{\text{\scriptsize{\rm SE}}}
\newcommand{\DE}{\text{\scriptsize{\rm DE}}}
\newcommand{\SEt}{\psi^{\SE}}
\newcommand{\DEt}{\psi^{\DE}}
\newcommand{\SEi}{\phi^{\SE}}
\newcommand{\DEi}{\phi^{\DE}}
\newcommand{\JmSE}{\mathcal{J}_m^{\SE}}
\newcommand{\tJmSE}{\mathcal{\tilde{J}}_m^{\SE}}
\newcommand{\JmDE}{\mathcal{J}_m^{\DE}}
\newcommand{\tJmDE}{\mathcal{\tilde{J}}_m^{\DE}}
\newcommand{\oSE}{\omega^{\SE}}
\newcommand{\oDE}{\omega^{\DE}}
\newcommand{\wSE}{W^{\SE}}
\newcommand{\wDE}{W^{\DE}}
\newcommand{\vSEm}{\mathcal{V}^{\SE}_m}
\newcommand{\vDEm}{\mathcal{V}^{\DE}_m}
\newcommand{\dSEm}{D^{\SE}_m}
\newcommand{\dDEm}{D^{\DE}_m}
\newcommand{\aSEm}{A^{\SE}_m}
\newcommand{\aDEm}{A^{\DE}_m}
\newcommand{\vvSEm}{\boldsymbol{v}^{\SE}_m}
\newcommand{\vvDEm}{\boldsymbol{v}^{\DE}_m}
\newcommand{\D}{\,\mathrm{d}}
\newcommand{\Hinf}{\mathbf{H}^{\infty}}
\newcommand{\MC}{\mathbf{M}}
\DeclareMathSymbol{\round}{\mathord}{letters}{"40} 
\renewcommand{\partial}{\textrm{\Pisymbol{psy}{"B6}}} 

\begin{document}

\DOI{DOI HERE}
\copyrightyear{2021}
\vol{00}
\pubyear{2021}
\access{Advance Access Publication Date: Day Month Year}
\appnotes{Paper}
\copyrightstatement{Published by Oxford University Press on behalf of the Institute of Mathematics and its Applications. All rights reserved.}
\firstpage{1}


\title[Refinement of the theory and convergence of the Sinc convolution]{Refinement of the theory and convergence of the Sinc convolution---beyond Stenger's conjecture}

\author{Tomoaki Okayama*\ORCID{0000-0001-9942-1670}
\address{\orgdiv{Graduate School of Information Sciences}, \orgname{Hiroshima City University}, \orgaddress{\street{3-4-1, Ozuka-higashi, Asaminami-ku}, \postcode{731-3194}, \state{Hiroshima}, \country{Japan}}}}

\authormark{Tomoaki Okayama}

\corresp[*]{Corresponding author: \href{okayama@hiroshima-cu.ac.jp}{okayama@hiroshima-cu.ac.jp}}

\received{Date}{0}{Year}
\revised{Date}{0}{Year}
\accepted{Date}{0}{Year}


\abstract{The Sinc convolution is an
approximate formula for indefinite convolutions proposed by Stenger.
The formula was derived based on the Sinc indefinite integration formula
combined with the single-exponential transformation.
Although its efficiency has been confirmed in various fields,
several theoretical issues remain unresolved.
The first contribution of this study is to resolve those issues
by refining the underlying theory of the Sinc convolution.
This contribution includes an essential resolution of Stenger's conjecture.
The second contribution of this study is to improve the convergence rate
by replacing the single-exponential transformation with
the double-exponential transformation.
Theoretical analysis and numerical experiments confirm that
the modified formula achieves superior convergence
compared to Stenger's original formula.
}
\keywords{Sinc approximatioin; Sinc indefinite integration; Laplace transform; Dunford integral.}


\maketitle

\section{Introduction}
\label{sec:intro}

This study is concerned with indefinite convolutions of the form
\begin{equation}
 p(x) = \int_a^x f(x - t)g(t) \D{t},\quad a\leq x\leq b,
\label{eq:p-x}
\end{equation}
and their collocation-type approximations given by
\[
 p(x)\approx \sum_{j=1}^m p_j \omega_j(x),
\]
where the collocation points $x_i$ satisfy $\omega_j(x_i)=\delta_{ij}$
($\delta_{ij}$ denotes the Kronecker delta).
Here, $p_j$ are fixed with respect to $x$,
and $\omega_j$ are composed of elementary functions.
Such approximations are particularly well-suited for
solving ordinary and partial differential equations whose solution can be
written as convolution-type integral, or
that can be transformed into
convolution-type integral equations~\citep{stenger00:_summar_sinc}.

\citet{stenger95:_colloc}
derived a beautiful approximation formula of such a type
utilizing the method known as the ``Sinc indefinite integration.''
The derivation of the formula comprises two parts.
The first part is to derive the expression of $p(x)$ as
\begin{equation}
 p(x) = \left(F(\mathcal{J}) g\right)(x),
\label{eq:p-F-J-g}
\end{equation}
where $F(s)=\hat{f}(1/s)$, $\hat{f}$ denotes the Laplace transform
of $f$, and $\mathcal{J}$ denotes an integral operator defined by
\begin{equation}
 \left(\mathcal{J}g\right)(x) = \int_a^x g(t)\D{t}.
\label{eq:def-J}
\end{equation}
For instance, when $f(x)=x$, its transform yields $F(s)=\hat{f}(1/s)=s^2$,
and consequently, $p(x)$ can be represented as
\begin{equation}
 \int_a^x (x - t) g(t)\D{t}
= \left(\mathcal{J}^2 g\right)(x).
\label{eq:example-f-x}
\end{equation}
This expression $\mathcal{J}^2 g$ aligns with the fact
that the left-hand side of~\eqref{eq:example-f-x} can be written as
a repeated integral of $g$ as
\[
 \int_a^x (x - t) g(t)\D{t}
=\int_a^x\left(\int_a^t g(u)\D{u}\right)\D{t}
=\left(\mathcal{J}\left(\mathcal{J} g\right)\right)(x),
\]
and we can interpret~\eqref{eq:p-F-J-g} as a general formulation
of an indefinite convolution.
The second part is to approximate $\mathcal{J}$ by $\JmSE$
in the expression~\eqref{eq:p-F-J-g} as
\[
 p(x) = \left(F(\mathcal{J}) g\right)(x) \approx
 \left(F(\JmSE) g\right)(x),
\]
where $\JmSE$ denotes an approximation of $\mathcal{J}$
via the Sinc indefinite integration
combined with the single-exponential (SE) transformation
(the definition will be given in~\eqref{eq:JmSE}).
Furthermore, he theoretically showed that
the formula may achieve root-exponential convergence:
$\OO(\exp(-c\sqrt{m}))$.
Owing to its efficiency, the method has found applications
in diverse fields, such as
the Hilbert transform~\citep{yamamoto06:_approx_hilbert},
fractional calculus~\citep{baumann11:_fract_calc,baumann15:_sinc_approx},
Laplace transform inversion~\citep{Stenger,stenger98:_comput_sinc},
integral equations~\citep{stenger13:_visco},
partial differential equations~\citep{morlet95:_sinc,stenger07:_separ_pde,stenger02:_sampl_pde,stenger04:_sinc}, and
Navier--Stokes equations~\citep{Navier-Stokes}.

In the theory of the Sinc convolution, however,
there still remain two points to be resolved.
The first point is the assumption on the function $F$.
In Stenger's derivation of the expression~\eqref{eq:p-F-J-g},
$F$ is assumed to be analytic on the open right half of
the complex plane $\Omega^{+}=\{z\in\mathbb{R}:\Re z > 0\}$.
This assumption is not satisfied in the case of several classical examples,
such as $f(x)=\E^{x}$ $(F(s)=s/(1-s))$.
Furthermore, the restriction of $F$ implies that
the spectrum of the operator $\JmSE$, denoted by $\sigma(\JmSE)$,
must lie on $\Omega^{+}$ to ensure the well-definedness of $F(\JmSE)$.
This spectral inclusion, $\sigma(\JmSE)\subset \Omega^{+}$, is known as
Stenger's conjecture~\citep{stenger97:_matrices_sinc},
which has remained open since the inception of the Sinc convolution.
\citet{han14:_proof_steng_i_sincy}
made substantial progress on the conjecture,
but as remarked at the end of their paper,
the conjecture was not fully proved.
Even if proved,
the case $F(s)=s/(1-s)$ remains problematic due to
the singular point on $\Omega^{+}$.

The second point is the assumption on the function
\begin{equation}
\label{eq:P-v-x}
 P(v, x) = \int_a^x f(v + x - t) g(t)\D{t},\quad a\leq x\leq b,
\end{equation}
where $v\in [0, b-a]$.
In Stenger's error analysis,
$P(v, x)$ is assumed to belong to a function space
that is characterized by parameters $\alpha$, $\beta$ and $d$
(details will be explained in Assumption~\ref{assump:SE2}).
The values of these parameters are required for implementation;
however, their explicit determination is challenging
because $P(v, x)$ is not a given function.

The first objective of this study is to resolve the two points outlined above.
Regarding the first point, instead of the earlier assumption on $F$,
this study assumes that
there exists a sufficiently large circle
centered at the origin
encompassing all singularities of $\hat{f}$.
This assumption permits singularities on $\Omega^{+}$.
Under the assumption, the expression~\eqref{eq:p-F-J-g} is derived.
Furthermore, this study shows that for all sufficiently large $m$,
$\sigma(\JmSE)$ is located in some domain on which $F$ is analytic,
i.e., $F(\JmSE)$ is well-defined for all sufficiently large $m$.
This finding provides an alternative resolution to Stenger's conjecture:
it is not necessary to establish
$\sigma(\JmSE)\subset \Omega^{+}$
to ensure the well-definedness of $F(\JmSE)$.
Regarding the second point, this study provides an error analysis
yielding the convergence rate $\OO(\exp(-c\sqrt{m}))$
without making an assumption on the function $P(v, x)$.
This finding is of great practical importance because
it eliminates the need for prior knowledge of $P(v, x)$,
which is unavailable or difficult to ascertain in applied contexts.

The second objective of this study is to enhance the convergence rate
of the Sinc convolution.
Recently,
by replacing the SE transformation
with the double-exponential (DE) transformation,
an improved Sinc indefinite integration has been
developed~\citep{okayama22:_yet_de_sinc}.
Based on the result,
if we appropriately define $\JmDE$ as an approximation of $\mathcal{J}$,
we naturally derive the improved approximation formula as
\[
 p(x) = \left(F(\mathcal{J}) g\right)(x) \approx
 \left(F(\JmDE) g\right)(x).
\]
This study also theoretically shows that
the enhanced formula may attain substantially accelerated convergence:
$\OO(\exp(-\tilde{c} m /\log m))$.

It should be emphasized that the expression of
the indefinite convolution~\eqref{eq:p-F-J-g} is valid
regardless of the approximation method for the integral
operator $\mathcal{J}$.
If we apply another method instead of the Sinc
indefinite integration, then we obtain another approximation
formula for~\eqref{eq:p-F-J-g}.
In fact, for some approximation formulas,
Stenger's conjecture
was formulated~\citep{stenger15:_comput_schr,stenger21:_indef},
which motivates researchers for further interest
investigations~\citep{gautschi19:_steng,ait-haddou23:_steng}.
The results provided in this study are also useful
for such formulas.

The remainder of this paper is organized as follows.
Section~\ref{sec:preliminary} introduces
essential approximation formulas needed in the subsequent sections.
Section~\ref{sec:derive-expression}
explains the derivation of the expression~\eqref{eq:p-F-J-g}.
Section~\ref{sec:main-result}
summarizes the main results of the present paper
with indicating the locations of corresponding proofs.
Section~\ref{sec:resolvent-approx-op}
establishes the existence of resolvent of approximated operators
$\JmSE$ and $\JmDE$.
Section~\ref{sec:proof-F-welldef}
analyzes the properties of the matrices
arising in the Sinc convolution,
thereby addressing the first point (well-definedness of $F(\JmSE)$).
Section~\ref{sec:proof-convergence} provides the proofs
of the presented convergence theorems,
thereby addressing the second point (assumption on $P(v,x)$).
Section~\ref{sec:numerical} provides numerical illustrations.
Finally,
Section~\ref{sec:conclusion} concludes this paper.

\section{Sinc approximation and Sinc indefinite integration}
\label{sec:preliminary}

The Sinc approximation is a function approximation formula
expressed as
\begin{equation}
 f(u) \approx \sum_{j=-M}^N f(j h) \sinc\left(\frac{u - j h}{h}\right),
\quad u\in\mathbb{R},
\label{eq:sinc-approx}
\end{equation}
where $\sinc(x)$ denotes the normalized sinc function.
As explained later, $h$, $M$ and $N$ are suitably selected
depending on a positive integer $n$.
Furthermore, throughout this paper, $m$ is set as $m=M+N+1$.

Integrating both sides of~\eqref{eq:sinc-approx},
we derive an approximate formula for the indefinite integral
over the real axis as
\begin{align}
\label{eq:sinc-indef}
\int_{-\infty}^{\tau}f(u)\D{u}
\approx \sum_{j=-M}^N f(j h)
\int_{-\infty}^{\tau}\sinc\left(\frac{u - j h}{h}\right)\D{u}
=\sum_{j=-M}^N f(j h) J(j,h)(\tau),\quad \tau\in\mathbb{R},
\end{align}
where $J(j,h)(x)$ is defined by
\begin{equation}
\label{eq:basis-J}
J(j,h)(x) = h\left\{\frac{1}{2}
 + \frac{1}{\pi}\Si\left(\frac{\pi(x - j h)}{h}\right)\right\},
\end{equation}
and $\Si(x)$ denotes the sine integral
defined by $\Si(x)=\int_0^x \{(\sin t)/t\}\D{t}$.
The approximation formula~\eqref{eq:sinc-indef}
is referred to as the ``Sinc indefinite integration.''

\subsection{Original and mat-vec friendly SE-Sinc indefinite integration}

For the indefinite integral
over the finite interval $(a, b)$,
\citet{haber93:_two} combined the formula~\eqref{eq:sinc-indef}
with the tanh transformation
\[
 t = \SEt(u) = \frac{b-a}{2}\tanh\left(\frac{u}{2}\right) + \frac{b+a}{2},
\]
which is also referred to as the single-exponential (SE) transformation.
More precisely, the formula was derived as
\begin{align}
\int_a^x g(t)\D{t}
&=\int_{-\infty}^{\phi^{\text{\tiny{\rm SE}}}(x)}g(\SEt(u))(\SEt)'(u)\D{u}
\nonumber\\
&\approx\sum_{j=-M}^N g(\SEt(j h))(\SEt)'(j h) J(j,h)(\SEi(x)),
\quad x\in [a, b],
\label{eq:SE-Sinc-indef-Haber}
\end{align}
where $\SEi$ denotes the inverse function of $\SEt$.
The formula~\eqref{eq:SE-Sinc-indef-Haber}
is referred to as the original SE-Sinc indefinite integration
in the present paper.
If we define an approximate operator $\tJmSE$ as
\begin{equation}
 \left(\tJmSE g\right)(x)
=\sum_{j=-M}^N g(\SEt(j h))(\SEt)'(j h) J(j,h)(\SEi(x)),
\label{eq:tJmSE}
\end{equation}
the approximation~\eqref{eq:SE-Sinc-indef-Haber} can be compactly expressed as
$\mathcal{J} g \approx \tJmSE g$ (recall~\eqref{eq:def-J}).
The approximation is highly efficient if
$g$ is analytic on an eye-shaped domain
\[
 \SEt(\domD_d) = \left\{z=\SEt(\zeta): \zeta\in\domD_d\right\},
\]
where $\domD_d$ is a strip complex domain defined by
\[
 \domD_d = \{\zeta\in\mathbb{C}: \left|\Im\zeta\right| < d\},
\]
for a positive constant $d$.
In fact, the following convergence theorem
was given.

\begin{theorem}[{\citet[Theorem 2.9]{Okayama-et-al}}]
\label{thm:Okayama-tJmSE}
Assume that $g$ is analytic on $\SEt(\domD_d)$ for $d\in(0,\pi)$,
and that there exist positive constants
$K$, $\alpha$ and $\beta$
such that
\begin{equation}
 |g(z)|\leq K|z-a|^{\alpha-1}|b-z|^{\beta-1}
\label{eq:indef-condition}
\end{equation}
holds for all $z\in\SEt(\domD_d)$.
Let $\mu = \min\{\alpha, \beta\}$,
$n$ be a positive integer, and $h$ be
selected by the formula
\begin{equation}
 h = \sqrt{ \frac{\pi d}{\mu n} }. \label{eq:h-SE}
\end{equation}
Moreover, let $M$ and $N$ be positive integers defined by
\begin{align}
M = \left\lceil\frac{\mu}{\alpha}n\right\rceil,
\quad
N = \left\lceil\frac{\mu}{\beta}n\right\rceil,
\label{eq:MN-SE}
\end{align}
respectively.
Then, there exists a constant $C$ independent of $n$ such that
\begin{align*}
 \max_{x\in [a,b]}\left|
\left(\mathcal{J} g\right)(x) - \left(\tJmSE g\right)(x)
\right|
 \leq C  \exp\left(-\sqrt{\pi d \mu n}\right).
\end{align*}
\end{theorem}

However,
the approximation formula~\eqref{eq:SE-Sinc-indef-Haber}
entails a computational limitation
due to the involvement of the special function $\Si(x)$
in the basis function $J(j,h)$.
because the basis function $J(j,h)$ includes a special function $\Si(x)$.
To circumvent this issue,
\citet{stenger93:_numer,stenger95:_colloc}
derived another approximation formula
whose the basis functions comprise
the normalized sinc function augmented with
two auxiliary functions:
\[
 \eta(x) = \frac{x-a}{b-a},\quad
\tilde{\eta}(x) = \frac{b-x}{b-a}.
\]
The explicit representation of the basis functions is given by
\begin{align*}
 \oSE_{-M}(x)&=\frac{1}{\tilde{\eta}(\SEt(-Mh))}
\left\{
\tilde{\eta}(x)
-\sum_{k=-M+1}^N\tilde{\eta}(\SEt(kh))\sinc\left(\frac{\SEi(x) - k h}{h}\right)
\right\},\\
\oSE_{j}(x) &= \sinc\left(\frac{\SEi(x) - j h}{h}\right)
\quad (j=-M+1,\,-M+2,\,\ldots,\,N-2,\,N-1),\\
 \oSE_{N}(x)&=\frac{1}{\eta(\SEt(Nh))}
\left\{
\eta(x)-\sum_{k=-M}^{N-1}\eta(\SEt(kh))\sinc\left(\frac{\SEi(x) - k h}{h}\right)
\right\}.
\end{align*}
Note that $\oSE_j(\SEt(ih))=\delta_{ij}$ ($i,j=-M,\,\ldots,\,N$) holds,
where $\delta_{ij}$ denotes the Kronecker delta.
Consequently, Stenger's formula takes the form
\begin{equation}
 \int_a^x g(t)\D t
\approx \sum_{i=-M}^N\left(
h\sum_{j=-M}^N \delta_{i-j}^{(-1)} g(\SEt(jh))(\SEt)'(jh)
\right)\oSE_i(x),\quad x\in [a,b],
\label{eq:SE-Sinc-indef-orig}
\end{equation}
where $\delta_{k}^{(-1)} = (1/2) + \sigma_{k}$, and
$\sigma_{k} = \int_0^{k} \sinc t\D{t}$.
For reference, tabulated values of $\sigma_k$ are
provided~\citep[Table~1.10.1]{stenger93:_numer}.

The approximation formula~\eqref{eq:SE-Sinc-indef-orig}
admits a representation in matrix-vector form,
requiring the following notations.
Let $\boldsymbol{\omega}^{\SE}_m$
be a row vector of order $m$ (recall $m=M+N+1$) defined by
\begin{equation}
\label{eq:oSEm}
 \boldsymbol{\omega}^{\SE}_m(x)
=\left[
\oSE_{-M}(x),\,\oSE_{-M+1}(x),\,\ldots,\,\oSE_{N-1}(x),\,\oSE_{N}(x)
\right],
\end{equation}
and let $\vSEm$ be a sampling operator
that maps a function $f$ to a column vector of order $m$
as
\begin{equation}
\label{eq:vSEm}
 \vSEm f =
 \left[
f(\SEt(-Mh)),\,f(\SEt(-(M-1)h)),\,\ldots,\,f(\SEt(Nh))
\right]^{\mathrm{T}}.
\end{equation}
Note that
$\vSEm \boldsymbol{\omega}^{\SE}_m = I_m$ holds,
where $I_m$ denotes the identity matrix of order $m$.
Let $I^{(-1)}_m$ be a square matrix of order $m$ having
$\delta^{(-1)}_{i-j}$ as its $(i, j)$th element, and
\begin{align}
\nonumber
 \dSEm &= \diag
 \left[
(\SEt)'(-Mh),\,(\SEt)'(-(M-1)h),\,\ldots,\,(\SEt)'(Nh)
\right],\\
\aSEm &= h I_m^{(-1)} \dSEm.
\label{eq:aSEm}
\end{align}
Then, the approximation~\eqref{eq:SE-Sinc-indef-orig} can be expressed as
\begin{equation}
 \int_{a}^x g(t)\D{t}
\approx \boldsymbol{\omega}^{\SE}_m(x) \aSEm \vSEm g ,\quad x\in [a, b].
\label{eq:SE-Sinc-indef}
\end{equation}
This mat-vec friendly approximation formula
is referred to as the SE-Sinc indefinite integration
in the present paper
(note that the formula~\eqref{eq:SE-Sinc-indef-Haber}
is referred to as the \textit{original} SE-Sinc indefinite integration).
If we define an approximate operator $\JmSE$ as
\begin{equation}
 \left(\JmSE g\right)(x) = \boldsymbol{\omega}^{\SE}_m(x) \aSEm \vSEm g ,
\label{eq:JmSE}
\end{equation}
the approximation~\eqref{eq:SE-Sinc-indef} can be compactly expressed as
$\mathcal{J} g \approx \JmSE g$.
Its convergence was analyzed as follows.

\begin{theorem}[{\citet[Theorem 4.9]{stenger95:_colloc}}]
\label{thm:Stenger-JmSE}
Assume that $g$ is analytic on $\SEt(\domD_d)$ for $d\in(0,\pi)$,
and that there exist positive constants
$K$, $\alpha$ with $\alpha\leq 1$ and $\beta$ with $\beta\leq 1$
such that~\eqref{eq:indef-condition}
holds for all $z\in\SEt(\domD_d)$.
Let $\mu = \min\{\alpha, \beta\}$,
$n$ be a positive integer, and $h$ be
selected by the formula~\eqref{eq:h-SE}.
Moreover, let $M$ and $N$ be positive integers defined by~\eqref{eq:MN-SE}.
Then, there exists a constant $C$ independent of $n$ such that
\[
 \max_{x\in [a,b]}\left|
\left(\mathcal{J} g\right)(x) - \left(\JmSE g\right)(x)
\right| \leq C  \sqrt{n} \exp\left(-\sqrt{\pi d \mu n}\right).
\]
\end{theorem}

\subsection{Original and mat-vec friendly DE-Sinc indefinite integration}

In the original SE-Sinc indefinite integration~\eqref{eq:SE-Sinc-indef-Haber},
the SE transformation $x=\SEt(u)$ is employed to
map $\mathbb{R}$ onto $(a, b)$.
Not only the SE transformation but also
the double-exponential (DE) transformation
\[
 x = \DEt(u)
 = \frac{b-a}{2}\tanh\left(\frac{\pi}{2}\sinh u\right) + \frac{b+a}{2}
\]
maps $\mathbb{R}$ onto $(a, b)$.
Notably, replacing the SE with the DE transformation
often enhances numerical methods based on the Sinc
approximation~\citep{mori05:_discov,murota25:_doubl,sugihara04:_recen_sinc}.
This observation motivated
\citet{Muhammad-Mori} to derive
another approximation formula as
\begin{align}
\int_a^x g(t)\D{t}
&=\int_{-\infty}^{\phi^{\text{\tiny{\rm DE}}}(x)}g(\DEt(u))(\DEt)'(u)\D{u}
\nonumber\\
&\approx\sum_{j=-M}^N g(\DEt(j h))(\DEt)'(j h) J(j,h)(\DEi(x)),
\quad x\in [a, b],
\label{eq:DE-Sinc-indef-MM}
\end{align}
where $\DEi$ denotes the inverse function of $\DEt$.
The formula~\eqref{eq:DE-Sinc-indef-MM}
is referred to as the original DE-Sinc indefinite integration
in the present paper.
If we define an approximate operator $\tJmDE$ as
\begin{equation}
 \left(\tJmDE g\right)(x)
=\sum_{j=-M}^N g(\DEt(j h))(\DEt)'(j h) J(j,h)(\DEi(x)),
\label{eq:tJmDE}
\end{equation}
the approximation~\eqref{eq:DE-Sinc-indef-MM} can be compactly expressed as
$\mathcal{J} g \approx \tJmDE g$.
The following convergence theorem shows that
replacing $\SEt$ in~\eqref{eq:SE-Sinc-indef-Haber}
with $\DEt$ may improve the convergence rate.
Here, $g$ is assumed to be analytic on
\[
 \DEt(\domD_d) = \left\{z=\DEt(\zeta): \zeta\in\domD_d\right\},
\]
which is a Riemann surface.

\begin{theorem}[{\citet[Theorem 2.16]{Okayama-et-al}}]
\label{thm:Okayama-tJmDE}
Assume that $g$ is analytic on $\DEt(\domD_d)$ for $d\in(0,\pi/2)$,
and that there exist positive constants
$K$, $\alpha$ and $\beta$
such that~\eqref{eq:indef-condition}
holds for all $z\in\DEt(\domD_d)$.
Let $\mu = \min\{\alpha, \beta\}$,
$n$ be a positive integer, and $h$ be
selected by the formula
\begin{equation}
 h = \frac{\log(2dn/\mu)}{n}. \label{eq:h-DE}
\end{equation}
Moreover, let $M$ and $N$ be positive integers defined by
\begin{align}
M= n - \left\lfloor\frac{1}{h}\log\left(\dfrac{\alpha}{\mu}\right)\right\rfloor,
\quad
N= n - \left\lfloor\frac{1}{h}\log\left(\dfrac{\beta}{\mu}\right)\right\rfloor,
\label{eq:MN-DE}
\end{align}
respectively.
Then, there exists a constant $C$ independent of $n$ such that
\begin{align*}
 \max_{x\in [a,b]}\left|
\left(\mathcal{J} g\right)(x) - \left(\tJmDE g\right)(x)
\right|
 \leq C \frac{\log(2 d n/\mu)}{n}
 \exp\left(\frac{- \pi d n}{\log(2 d n/\mu)}\right).
\end{align*}
\end{theorem}

In the basis function $J(j,h)$, however,
a special function $\Si(x)$ is included
in common with~\eqref{eq:SE-Sinc-indef-Haber}.
The remedy was proposed by replacing $\SEt$ in~\eqref{eq:SE-Sinc-indef}
with $\DEt$~\citep{okayama22:_yet_de_sinc}.
To express the resulting approximation precisely,
let us introduce notational conventions
following the framework of SE-Sinc indefinite integration.
Let $\boldsymbol{\omega}^{\DE}_m$
be a row vector of order $m$ defined by
\begin{equation}
 \boldsymbol{\omega}^{\DE}_m(x)
=\left[
\oDE_{-M}(x),\,\oDE_{-M+1}(x),\,\ldots,\,\oDE_{N-1}(x),\,\oDE_{N}(x)
\right],
\label{eq:oDEm}
\end{equation}
where
\begin{align*}
 \oDE_{-M}(x)&=\frac{1}{\tilde{\eta}(\DEt(-Mh))}
\left\{
\tilde{\eta}(x)
-\sum_{k=-M+1}^N\tilde{\eta}(\DEt(kh))\sinc\left(\frac{\DEi(x) - k h}{h}\right)
\right\},\\
\oDE_{j}(x) &= \sinc\left(\frac{\DEi(x) - j h}{h}\right)
\quad (j=-M+1,\,-M+2,\,\ldots,\,N-2,\,N-1),\\
 \oDE_{N}(x)&=\frac{1}{\eta(\DEt(Nh))}
\left\{
\eta(x)-\sum_{k=-M}^{N-1}\eta(\DEt(kh))\sinc\left(\frac{\DEi(x) - k h}{h}\right)
\right\}.
\end{align*}
Furthermore, let $\vDEm$ be a sampling operator
that maps a function $f$ to a column vector of order $m$
as
\begin{equation}
 \vDEm f =
 \left[
f(\DEt(-Mh)),\,f(\DEt(-(M-1)h)),\,\ldots,\,f(\DEt(Nh))
\right]^{\mathrm{T}},
\label{eq:vDEm}
\end{equation}
and let $\dDEm$ and $\aDEm$ are defined by
\begin{align}
\nonumber
 \dDEm &= \diag
 \left[
(\DEt)'(-Mh),\,(\DEt)'(-(M-1)h),\,\ldots,\,(\DEt)'(Nh)
\right],\\
\aDEm &= h I_m^{(-1)} \dDEm. \label{eq:aDEm}
\end{align}
Then, the precise approximation form is written as
\begin{equation}
 \int_a^x g(t)\D{t}
\approx \boldsymbol{\omega}^{\DE}_m(x) \aDEm \vDEm g ,\quad x\in [a, b].
\label{eq:DE-Sinc-indef}
\end{equation}
This approximation is referred to as the DE-Sinc indefinite integration
in the present paper.
If we define an approximate operator $\JmDE$ as
\begin{equation}
 \left(\JmDE g\right)(x) = \boldsymbol{\omega}^{\DE}_m(x) \aDEm \vDEm g ,
\label{eq:JmDE}
\end{equation}
the approximation~\eqref{eq:DE-Sinc-indef} can be compactly expressed as
$\mathcal{J} g \approx \JmDE g$.
Its convergence was analyzed as follows.

\begin{theorem}[{\citet[Theorem 2.6]{okayama22:_yet_de_sinc}}]
\label{thm:Okayama-JmDE}
Assume that $g$ is analytic on $\DEt(\domD_d)$ for $d\in(0,\pi/2)$,
and that there exist positive constants
$K$, $\alpha$ with $\alpha\leq 1$ and $\beta$ with $\beta\leq 1$
such that~\eqref{eq:indef-condition}
holds for all $z\in\DEt(\domD_d)$.
Let $\mu = \min\{\alpha, \beta\}$,
$n$ be a positive integer, and $h$ be
selected by the formula~\eqref{eq:h-DE}.
Moreover, let $M$ and $N$ be positive integers defined by~\eqref{eq:MN-DE}.
Then, there exists a constant $C$ independent of $n$ such that
\[
 \max_{x\in [a,b]}\left|
\left(\mathcal{J} g\right)(x) - \left(\JmDE g\right)(x)
\right| \leq C  \exp\left(\frac{-\pi d n}{\log(2 d n/\mu)}\right).
\]
\end{theorem}

\section{Rewriting the indefinite convolution}
\label{sec:derive-expression}

In this section,
for completeness,
the derivation of the expression~\eqref{eq:p-F-J-g}
is described.
Although this result was previously established
by Stenger~\cite{stenger95:_colloc},
the derivation herein differs and reflects the formulation
adopted in the present work.

\subsection{Existence of the resolvent of \texorpdfstring{$\mathcal{J}$}{J}}

Let $\mathbf{Z}=L^1(a,b)$, and let
$\mathcal{J}g$ be defined for $g\in\mathbf{Z}$ by~\eqref{eq:def-J}.
Then, $\mathcal{J}$ denotes a linear operator
from $\mathbf{Z}$ onto itself.
By the Cauchy formula for repeated integration, it follows
for all positive integers $k$ that
\begin{equation}
 \left(\mathcal{J}^k g\right)
=\int_{a}^x\frac{(x - t)^{k-1}}{(k-1)!}g(t)\D{t}.
\label{eq:Cauchy-formula}
\end{equation}
With the standard norm $\|g\|_{\mathbf{Z}}=\int_a^b |g(t)|\D{t}$,
$\mathbf{Z}$ is a Banach space.
For the space $\mathbf{Z}$,
the resolvent of $\mathcal{J}$
exists for all $z\in\mathbb{C}\setminus \{0\}$,
as stated below.

\begin{proposition}\label{prop:resolvent-Z}
For all $z\in\mathbb{C}\setminus \{0\}$,
$(z - \mathcal{J})^{-1}:\mathbf{Z}\to\mathbf{Z}$ exists.
\end{proposition}
\begin{proof}
We show that the Neumann series
\[
 \frac{1}{z}\sum_{k=0}^{\infty}\frac{\mathcal{J}^k}{z^k}
=(z - \mathcal{J})^{-1}
\]
converges for all $z\in\mathbb{C}$ excluding the point $z=0$.
Applying~\eqref{eq:Cauchy-formula}, we have
\begin{align*}
  \left\|\mathcal{J}^k g\right\|_{\mathbf{Z}}
&=\int_a^b \left| \int_a^x \frac{(x - t)^{k-1}}{(k-1)!}g(t)\D{t}\right|\D{x}\\
&\leq\int_a^b \left\{
\int_a^x \frac{(x - t)^{k-1}}{(k-1)!}|g(t)|\D{t}\right\}\D{x}\\
&=\int_a^b \left\{
\int_t^b \frac{(x - t)^{k-1}}{(k-1)!}\D{x}\right\}|g(t)|\D{t}\\
&=\int_a^b \frac{(b-t)^k}{k!}|g(t)|\D{t}\\
&\leq \frac{(b-a)^k}{k!}\int_a^b|g(t)|\D{t}\\
&= \frac{(b-a)^k}{k!} \|g\|_{\mathbf{Z}},
\end{align*}
from which
$\|\mathcal{J}^k\|_{\mathcal{L}(\mathbf{Z},\mathbf{Z})}\leq (b-a)^k/k!$
follows.
This inequality implies that for $z\neq 0$
\begin{equation}
 \left\|
\frac{1}{z}\sum_{k=0}^{\infty}\frac{\mathcal{J}^k}{z^k}
 \right\|_{\mathcal{L}(\mathbf{Z},\mathbf{Z})}
\leq \frac{1}{|z|}
\sum_{k=0}^{\infty}\left\|\frac{\mathcal{J}^k}{z^k}\right\|_{\mathcal{L}(\mathbf{Z},\mathbf{Z})}
\leq \frac{1}{|z|}
\sum_{k=0}^{\infty}\frac{(b - a)^k}{|z|^k k!}
=\frac{1}{|z|}\E^{(b-a)/|z|} < \infty,
\label{eq:bound-resolvent}
\end{equation}
which shows the claim.
\end{proof}

The argument here can be reproduced even if
$\mathbf{Z}=L^1(a, b)$ is replaced with
another Banach space $\mathbf{W}=C([a, b])$,
and the norm $\|\cdot \|_{\mathbf{Z}}$
is replaced with $\|\cdot \|_{\mathbf{W}}$.
This result is important for the error analysis,
which is given by the uniform norm on $[a, b]$, i.e.,
$\|\cdot \|_{\mathbf{W}}$.

\begin{proposition}\label{prop:resolvent-W}
For all $z\in\mathbb{C}\setminus \{0\}$,
$(z - \mathcal{J})^{-1}:\mathbf{W}\to\mathbf{W}$ exists.
\end{proposition}
\begin{proof}
Note that $\mathcal{J}$ maps $\mathbf{W}$ onto $\mathbf{W}$.
Applying~\eqref{eq:Cauchy-formula}, we have
\begin{align*}
 \left\|\mathcal{J}^k g\right\|_{\mathbf{W}}
&=\max_{a\leq x\leq b}
\left|\int_a^x \frac{(x-t)^{k-1}}{(k-1)!}g(t)\D{t}\right|\\
&\leq \max_{a\leq x\leq b}
\left\{
\int_a^x \frac{(x-t)^{k-1}}{(k-1)!}|g(t)|\D{t}
\right\}\\
&\leq \|g\|_{\mathbf{W}}
\max_{a\leq x\leq b}\left\{
\int_a^x \frac{(x-t)^{k-1}}{(k-1)!}\D{t}
\right\}\\
&=\|g\|_{\mathbf{W}}
\max_{a\leq x\leq b} \left\{\frac{(x - a)^k}{k!}\right\}\\
&=\|g\|_{\mathbf{W}}\frac{(b-a)^k}{k!},
\end{align*}
from which we have
$\|\mathcal{J}^k\|_{\mathcal{L}(\mathbf{W},\mathbf{W})}\leq (b - a)^k/k!$.
Then, the claim follows
by the same argument as~\eqref{eq:bound-resolvent}.
\end{proof}

The same result holds for the following function space,
which is crucial for the error analysis in the present paper.

\begin{definition}
Let $\domD$ be a bounded and simply-connected domain
(or Riemann surface).
Then, $\Hinf(\domD)$ denotes
the family of all functions $f$
that are analytic and bounded on $\domD$.
This function space is a Banach space with the norm
\[
 \|f\|_{\Hinf(\domD)} = \sup_{z\in\domD}|f(z)|.
\]
\end{definition}

The following result is instrumental
in establishing properties of the function space $\mathbf{Y}=\Hinf(\domD)$.

\begin{lemma}[Okayama et al.~{\cite[Lemma~5.1]{okayama15:_theor_sinc_nystr_volter}}]
Let $\domD=\psi(\domD_d)$ or $\domD=\phi(\domD_d)$.
Assume that $f\in\Hinf(\domD)$.
Then, it holds for all positive integers $k$ that
\[
 \|\mathcal{J}^k f\|_{\Hinf(\domD)}
\leq \frac{\{(b - a) c_d\}^k}{k!}\|f\|_{\Hinf(\domD)},
\]
where $c_d$ is a constant depending only on $d$.
\end{lemma}

From the inequality and noting that
$\mathcal{J}$ maps $\mathbf{Y}$ onto $\mathbf{Y}$,
we have the following result in the same way as
Propositions~\ref{prop:resolvent-Z} and~\ref{prop:resolvent-W}.

\begin{proposition}\label{prop:resolvent-Y}
For all $z\in\mathbb{C}\setminus \{0\}$,
$(z - \mathcal{J})^{-1}:\mathbf{Y}\to\mathbf{Y}$ exists.
\end{proposition}

\subsection{Derivation of the expression of \texorpdfstring{$p(x)$}{p(x)} using the ``Laplace transform''}

Throughout Stenger's paper~\cite{stenger95:_colloc},
the Laplace transform denotes
the function $\hat{f}(s)=\int_0^{c} \E^{-st}f(t)\D{t}$,
while the term ``Laplace transform''
(referred with quotations) denotes the function
\begin{equation}
 F(s) = \hat{f}(1/s) = \int_0^{c} \E^{-t/s}f(t)\D{t},
\label{eq:def-F}
\end{equation}
where $c\in [b - a, \infty]$.
The typical choice of $c$ is $c=\infty$,
which is appropriate if the integral~\eqref{eq:def-F} exists.
If this is not the case, e.g., $f(t)=t^{-1/3}\exp(t^2)$ and
$(a,b)=(0,1)$,
$c$ can be restricted to a finite value within $b - a \leq c < \infty$,
to ensure $F(s)$ exists.
%
Using the Bromwich integral
for the inverse Laplace transform, we have
\begin{equation}
 f(t) = \frac{1}{2\pi\I}
\int_{\gamma - \I\infty}^{\gamma + \I\infty}
\E^{st} \hat{f}(s)\D{s},
\label{eq:Bromwich}
\end{equation}
where $\gamma$ is chosen so that
all singularities of $\hat{f}(s)$ lie to the left of the line
$s=\gamma$.
If $\gamma<0$,
we can rechoose $\gamma$ as $\gamma=0$.
Therefore, without loss of generality,
let $\gamma\geq 0$ in the following derivation.

\begin{remark}
In view of~\eqref{eq:p-x},
the variable of $f$ is restricted to the interval $[0, b-a]$.
Accordingly, instead of a given function $f$,
we can consider another function
\[
 f_c(x) =
\begin{cases}
 f(x) & (x\in [0, c])\\
 0    & (\text{otherwise})
\end{cases}
\]
where $c\in [b - a, \infty]$.
Even when $c$ is finite,
the standard Laplace transform can be written as
\[
 \int_0^{\infty} \E^{-st} f_c(t)\D{t} =
\int_0^c \E^{-st} f(t)\D{t} = \hat{f}(s),
\]
and by the inverse Laplace transform via the Bromwich integral,
it follows for $0\leq t\leq c$ that
\[
 f(t) = f_c(t) = \frac{1}{2\pi\I}
\int_{\gamma - \I\infty}^{\gamma + \I\infty} \E^{st}
\left(
\int_0^{\infty} \E^{-st} f_c(t)\D{t}
\right)
\D{s}
=\frac{1}{2\pi\I}
\int_{\gamma - \I\infty}^{\gamma + \I\infty} \E^{st}
\hat{f}(s)
\D{s},
\]
which conforms with the framework presented by Stenger.
\end{remark}

Hereafter,
let $\varGamma_{R}$ denote
the positively oriented circle of radius $R$ centered at the origin.
We rewrite the integral~\eqref{eq:Bromwich}
as a contour integral along $\varGamma_{R}$ as follows.
\begin{lemma}
\label{lem:Bromwich-contour}
Let $\hat{f}(s)=\int_0^c\E^{-st}f(t)\D{t}$ for some $c\in[b - a,\infty]$.
Assume that $\hat{f}$ is an analytic function
satisfying the condition that
there exists a sufficiently large real positive number $R$
such that
all singularities of $\hat{f}$ lie inside the contour $\varGamma_{R}$.
Then, the integral path of the Bromwich integral~\eqref{eq:Bromwich}
can be replaced with the contour
$\varGamma_{R}$. That is,
\begin{equation}
 \frac{1}{2\pi\I}\int_{\gamma - \I\infty}^{\gamma + \I\infty}
\E^{st} \hat{f}(s)\D{s}
=\frac{1}{2\pi\I}\oint_{\varGamma_{R}} \E^{st} \hat{f}(s)\D{s},
\quad t>0,
\label{eq:Bromwich-contour}
\end{equation}
where $\gamma$ is a nonnegative constant chosen so that
all singularities of $\hat{f}(s)$ lie to the left of the vertical line
$s=\gamma$.
\end{lemma}

To prove the lemma, the following estimate
is useful.

\begin{lemma}
\label{lem:Jordan}
Let $R$ be a positive constant,
and let $C_{R}$ be a semicircular contour given by
\[
 \left\{z = R\E^{\I\theta}\,
  \middle|\, \frac{\pi}{2}\leq \theta\leq \frac{3}{2}\pi\right\}
\]
with counterclockwise direction (see Fig.~\ref{fig:cntr_C_R_gamma}).
Furthermore,
let $\gamma$ be a nonnegative constant,
let $C_{R,\gamma}^{+}$ be a straight-line path from
$\gamma+R\I$ to $R\I$,
let $C_{R,\gamma}^{-}$ be a straight-line path from
$-R\I$ to $\gamma - R\I$,
and let $C_{R,\gamma}=C_{R,\gamma}^{+}\cup C_{R}\cup C_{R,\gamma}^{-}$.
Assume that $\hat{f}$ is continuous on $C_{R,\gamma}$.
Then, for $t>0$, we have
\[
 \left|
  \int_{C_{R,\gamma}} \E^{st}\hat{f}(s)\D{s}
\right|
\leq \left\{
 \frac{\pi(1 - \E^{-R t})}{t} + \frac{2(\E^{\gamma t} - 1)}{t}
\right\}\varLambda,
\]
where
\[
\varLambda
 =\max_{s\in C_{R,\gamma}}\left|\hat{f}(s)\right|.
\]
\end{lemma}
\begin{proof}
First, it holds that
\[
 \left|
  \int_{C_{R,\gamma}} \E^{st}\hat{f}(s)\D{s}
\right|
\leq
\left|
  \int_{C_{R,\gamma}^{+}} \E^{st}\hat{f}(s)\D{s}
\right|
+
\left|
  \int_{C_{R}} \E^{st}\hat{f}(s)\D{s}
\right|
+
\left|
  \int_{C_{R,\gamma}^{-}} \E^{st}\hat{f}(s)\D{s}
\right|.
\]
On the first and third terms,
changing the variable as $s=u\pm\I R$, we have
\[
 \left|
  \int_{C_{R,\gamma}^{\pm}} \E^{st}\hat{f}(s)\D{s}
\right|
\leq \varLambda\int_{C_{R,\gamma}^{\pm}}|\E^{st}| |\D{s}|
=\varLambda\int_0^{\gamma} \E^{ut}\D{u}
=\varLambda\frac{\E^{\gamma t}-1}{t}.
\]
On the second term,
changing the variable as $s=R\E^{\I\theta}$, we have
\begin{align*}
 \left|
  \int_{C_{R}} \E^{st}\hat{f}(s)\D{s}
\right|
\leq \varLambda \int_{C_{R}} |\E^{st}| |\D{s}|
=\varLambda \int_{\frac{\pi}{2}}^{\frac{3}{2}\pi}
 \E^{tR\cos\theta} R\D{\theta}
=2R  \varLambda \int_{0}^{\frac{\pi}{2}}
\E^{-tR\sin\theta} \D{\theta}.
\end{align*}
Furthermore, applying Jordan's inequality, we have
\begin{align*}
2R  \varLambda \int_{0}^{\frac{\pi}{2}}
\E^{-tR\sin\theta} \D{\theta}
\leq
2R \varLambda \int_{0}^{\frac{\pi}{2}}
 \E^{-2tR\theta/\pi} \D{\theta}
=\pi\varLambda\frac{1 - \E^{-R t}}{t},
\end{align*}
from which the claim follows.
\end{proof}

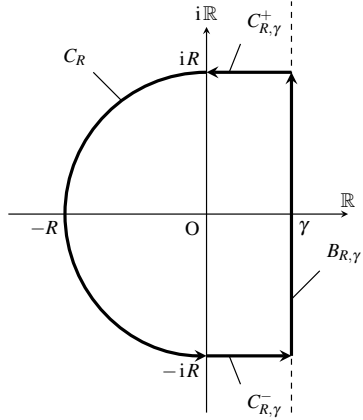
\begin{figure}
\centering
\scalebox{0.75}{
\begin{tikzpicture}
\draw[dashed,semithick] (1.5,-3.5)--(1.5,3.75);
\draw[->,>=stealth,semithick] (-3.5,0)--(2.5,0)node[above]{$\mathbb{R}$};
\draw[->,>=stealth,semithick] (0,-3.5)--(0,3.3)node[above]{$\I \mathbb{R}$};
\draw (0,0)node[below left]{$\OO$};
\draw (-2.5,0)node[below left]{$- R$};
\draw (1.5,0)node[below right]{$\gamma$};
\draw [->,>=stealth,ultra thick] (0,2.5) arc (90:270:2.5);
\draw[->,>=stealth,ultra thick] (1.5, 2.5)--(0, 2.5);
\draw[<-,>=stealth,ultra thick] (1.5,-2.5)--(0,-2.5);
\draw[->,>=stealth,ultra thick] (1.5,-2.5)--(1.5,2.5);
\draw (0, 2.5)node[above left]{$\I R$};
\draw (0,-2.5)node[below left]{$-\I R$};
\draw[semithick] (0.4,2.5)--(0.6,3)node[above right]{$C_{R,\gamma}^{+}$};
\draw[semithick] (0.4,-2.5)--(0.6,-3)node[below right]{$C_{R,\gamma}^{-}$};
\draw[semithick] (-1.5,2)--(-2,2.5)node[above left]{$C_{R}$};
\draw[semithick] (1.5,-1.5)--(2,-1)node[above right]{$B_{R,\gamma}$};
\end{tikzpicture}
}
\caption{Contours $C_{R}$, $C_{R,\gamma}^{+}$, $C_{R,\gamma}^{-}$, and $B_{R,\gamma}$.}
\label{fig:cntr_C_R_gamma}
\end{figure}

Using Lemma~\ref{lem:Jordan},
we can prove Lemma~\ref{lem:Bromwich-contour} as follows.

\begin{proof}[Proof of Lemma~\ref{lem:Bromwich-contour}]
From the assumption, $\hat{f}$
has no singularity at $s=\infty$,
because all singularities of $\hat{f}$ lie inside $\varGamma_R$.
Furthermore, using
\begin{equation}
\lim_{s\to\infty}\hat{f}(s)
=\lim_{s\to\infty}\int_0^c \E^{-st}f(t)\D{t}
=0
\label{eq:lim-hat-f-zero}
\end{equation}
and Lemma~\ref{lem:Jordan},
we can rewrite the integral~\eqref{eq:Bromwich} as
\[
 \int_{\gamma - \I\infty}^{\gamma + \I\infty}
\E^{st} \hat{f}(s)\D{s}
=\lim_{R\to\infty}
\oint_{B_{R,\gamma}\cup C_{R,\gamma}}
\E^{st} \hat{f}(s)\D{s},
\]
where $C_{R,\gamma}$ is a contour defined in Lemma~\ref{lem:Jordan},
and $B_{R,\gamma}$ is a straight-line path from $\gamma-\I R$
to $\gamma + \I R$ (see Fig.~\ref{fig:cntr_C_R_gamma}).
The desired integral is obtained as
\[
 \lim_{R\to\infty}
\oint_{B_{R,\gamma}\cup C_{R,\gamma}}
\E^{st} \hat{f}(s)\D{s}
=\oint_{B_{R,\gamma}\cup C_{R,\gamma}}
\E^{st} \hat{f}(s)\D{s}
=\oint_{\varGamma_{R}}\E^{st} \hat{f}(s)\D{s},
\]
because $\hat{f}$ is analytic outside $\varGamma_{R}$
and to the right of a vertical line $s=\gamma$.
\end{proof}

Putting $r=1/R$ and $s = (1/r)\E^{\I\theta}$, we can further rewrite
the integral~\eqref{eq:Bromwich-contour} as
\begin{align*}
 \frac{1}{2\pi\I}\oint_{\varGamma_{1/r}} \E^{st} \hat{f}(s)\D{s}
&=\frac{1}{2\pi\I}\int_{-\pi}^{\pi} \E^{t \E^{\I\theta}/r}
\hat{f}\left(\frac{1}{r}\E^{\I\theta}\right)
\frac{1}{r}\E^{\I\theta} \I \D{\theta}\\
&=\frac{1}{2\pi\I}\int_{\pi}^{-\pi} \E^{t \E^{-\I\theta}/r}
\hat{f}\left(\frac{1}{r}\E^{-\I\theta}\right)
\frac{1}{r}\E^{-\I\theta} \I (-\D{\theta})\\
&=\frac{1}{2\pi\I}\int_{-\pi}^{\pi} \E^{t /(r\E^{\I\theta})}
\hat{f}\left(\frac{1}{r\E^{\I\theta}}\right)
\frac{1}{r \E^{\I\theta}} \I \D{\theta}\\
&=\frac{1}{2\pi\I}\oint_{\varGamma_{r}}
\E^{t/z} \hat{f}\left(\frac{1}{z}\right)\frac{1}{z^2}\D{z},
\end{align*}
where $z$ is put as $z=r\E^{\I\theta}$ in the last equality.
Using the equality and $F(s)=\hat{f}(1/s)$, we have
\begin{equation}
f(t) = \frac{1}{2\pi \I}
 \oint_{\varGamma_{r}}
 \frac{1}{z^2}\E^{t/z}F(z)\D{z}.
\label{eq:Bromwich2}
\end{equation}
By using~\eqref{eq:Bromwich2}, $p(x)$ in~\eqref{eq:p-x}
can be rewritten as
\begin{equation}
 p(x) = \int_a^x\left(
\frac{1}{2\pi\I}
 \oint_{\varGamma_{r}}
 \frac{1}{z^2}\E^{(x-t)/z}F(z)\D{z}
\right) g(t)\D{t}
 = \frac{1}{2\pi\I} \oint_{\varGamma_{r}}
w(x,z)\D{z},
\label{eq:p-x-Laplace}
\end{equation}
where
\[
 w(x,z) = \int_a^x \frac{1}{z^2}\E^{(x-t)/z}F(z) g(t)\D{t},
\]
which exists for $g\in\mathbf{Z}$.
Using the series expansion of $\E^{x}$ and~\eqref{eq:Cauchy-formula},
we can further rewrite $w(x,z)$ as
\begin{align}
\label{eq:def-w-x-s}
w(x,z)
&=\int_a^x \frac{1}{z^2}
\sum_{n=0}^{\infty}\frac{(x - t)^n}{n! z^n} F(z) g(t)\D{t}
=\left(
\frac{\mathcal{J}}{z^2}\sum_{n=0}^{\infty}\frac{\mathcal{J}^{n}}{z^{n}} F(z) g
\right)(x)
=\left(\frac{\mathcal{J}}{z} (z - \mathcal{J})^{-1} F(z) g\right)(x).
\end{align}
Substituting~\eqref{eq:def-w-x-s} into~\eqref{eq:p-x-Laplace},
we have
\begin{align*}
p(x) = \left(
\frac{1}{2\pi\I}\oint_{\varGamma_{r}}
\frac{\mathcal{J}}{z} (z - \mathcal{J})^{-1} F(z) \D{z} g
\right)(x)
=\left(
\frac{1}{2\pi\I}\oint_{\varGamma_{r}}
\left\{(z - \mathcal{J})^{-1} - \frac{1}{z}\right\}  F(z) \D{z} g
\right)(x).
\end{align*}
This can be regarded as the Dunford integral,
and we may use Cauchy's integral formula to evaluate the integral as
\[
\left(
\frac{1}{2\pi\I}\oint_{\varGamma_{r}}
\left\{(z - \mathcal{J})^{-1} - \frac{1}{z}\right\}  F(z) \D{z} g
\right)(x)
=\left(\{F(\mathcal{J}) - F(0)\}g \right)(x),
\]
where the residues at $z=\mathcal{J}$ and $z=0$ are calculated
(note that $F$ is analytic inside $\varGamma_r$).
Furthermore, from~\eqref{eq:lim-hat-f-zero},
we have $F(0)=\lim_{s\to\infty}\hat{f}(s)=0$.
Thus, the desired expression~\eqref{eq:p-F-J-g} is finally obtained.
To summarize the result, let us set the following assumption.

\begin{assumption}
\label{assump:F}
Let $F$ be defined by~\eqref{eq:def-F} for some $c\in[b - a,\infty]$.
Assume that $F$ is an analytic function
satisfying the condition that
there exists a sufficiently small real positive number $r$
such that
all singularities of $F$ lie outside the contour $\varGamma_{r}$.
\end{assumption}

Under this assumption, if $g$ belongs to $L^1(a, b)$,
we have the following theorem.

\begin{theorem}
Assume that Assumption~\ref{assump:F} is fulfilled.
Furthermore, assume that $g\in L^1(a,b)$.
Then, it holds that
\[
 \int_a^x f(x - t)g(t) \D{t} = \left(F(\mathcal{J}) g\right)(x),
\quad a\leq x\leq b,
\]
where $\mathcal{J}$
denotes the integral operator defined by~\eqref{eq:def-J}.
\end{theorem}

In the same way,
the expression~\eqref{eq:p-F-J-g} can be obtained for
both $C([a,b])$ and $\Hinf(\domD)$ as follows.

\begin{theorem}
Assume that Assumption~\ref{assump:F} is fulfilled.
Furthermore, assume that $g\in C([a, b])$.
Then, it holds that
\[
 \int_a^x f(x - t)g(t) \D{t} = \left(F(\mathcal{J}) g\right)(x),
\quad a\leq x\leq b,
\]
where $\mathcal{J}$
denotes the integral operator defined by~\eqref{eq:def-J}.
\end{theorem}

\begin{theorem}
\label{thm:convolution-express}
Assume that Assumption~\ref{assump:F} is fulfilled.
Furthermore, assume that $g\in \Hinf(\domD)$.
Then, it holds that
\[
 \int_a^x f(x - t)g(t) \D{t} = \left(F(\mathcal{J}) g\right)(x),
\quad a\leq x\leq b,
\]
where $\mathcal{J}$
denotes the integral operator defined by~\eqref{eq:def-J}.
\end{theorem}

\section{Approximate formulas for the indefinite convolution and their convergence theorems (main results)}
\label{sec:main-result}

This section presents the principal findings of this study
concerning the sinc convolution,
originally derived by \citet{stenger95:_colloc}.
As noted in Section~\ref{sec:intro},
he assumed that $F$ is analytic on the open right half of
the complex plane $\Omega^{+}=\{z\in\mathbb{R}:\Re z > 0\}$,
and the spectrum of the operator $\JmSE$ lies on $\Omega^{+}$.
These assumptions are made to ensure that
$F(\JmSE)$ is well-defined for all $m$.
As a first contribution,
without the assumptions,
this paper shows that $F(\JmSE)$ is well-defined for
all sufficiently large $m$.
As a second contribution,
this paper improves Stenger's convergence theorem
by eliminating the assumption on $P(v,x)$
defined by~\eqref{eq:P-v-x}.
Then, as a third contribution,
this paper extends these results to the DE-Sinc framework;
the SE transformation is replaced with the DE transformation.

\subsection{Sinc convolution combined with the SE transformation}
\label{sec:Stenger-Sinc-conv}

First, we discuss on Stenger's approximate formula.
His idea to derive the formula is
to approximate $\mathcal{J}$ by $\JmSE$ as
\begin{equation}
\label{eq:approx-p-SE}
 p(x) = \left(F(\mathcal{J}) g\right)(x)
\approx \left(F(\JmSE) g\right)(x)
:= \boldsymbol{\omega}^{\SE}_m(x) F(\aSEm) \vSEm g,
\end{equation}
where $\boldsymbol{\omega}^{\SE}_m(x)$,
$\aSEm$, and $\vSEm$
are defined by~\eqref{eq:oSEm},~\eqref{eq:aSEm}
and~\eqref{eq:vSEm}, respectively.
If a non-singular matrix $X_m$ and complex numbers $s_{m,j}$
are determined such that
\[
 \aSEm = X_m \diag[s_{m,-M},\,\ldots,\,s_{m,N}] X_m^{-1},
\]
then the square matrix $F(\aSEm)$ may be defined via the equation
\[
 F(\aSEm) = X_m \diag[F(s_{m,-M}),\,\ldots,\,F(s_{m,N})] X_m^{-1}.
\]
Therefore, to ensure that $F(\aSEm)$ is well-defined,
the spectrum of $\aSEm$ must lie entirely
within the domain on which $F$ is non-singular.
He pointed out numerical evidence demonstrating that
$I_m^{(-1)}$ is diagonalizable,
and all eigenvalues lie on the open right half of
the complex plane $\Omega^{+}=\{z\in\mathbb{R}:\Re z > 0\}$
for all orders $m$ up to $513$
(later, up to $1024$~\citep{Stenger}).
He also noted that the real parts of
the eigenvalues of $\aSEm$ are positive
if and only if those of $I_m^{(-1)}$
are positive.
On the basis of these results,
he adopted the following assumption
to ensure that $F(\aSEm)$ is well-defined.

\begin{assumption}
\label{assump:SE1}
Let $\hat{f}(s)=\int_0^c\E^{-st}f(t)\D{t}$ for some $c\in[2(b - a),\infty]$.
Assume that $\hat{f}$ is analytic on $\Omega^{+}$.
Furthermore, assume that $\aSEm$ defined by~\eqref{eq:aSEm}
is diagonalizable,
and all eigenvalues of $\aSEm$ lie on $\Omega^{+}$.
\end{assumption}

Note that if $\hat{f}$ is analytic on $\Omega^{+}$,
then $F$ defined by~\eqref{eq:def-F} is also analytic on $\Omega^{+}$.
The condition $c\geq 2(b - a)$
(not $c\geq b - a$) is not for the well-definedness of $F(\aSEm)$,
but due to technical reasons for the error analysis.

In this study,
to eliminate Assumption~\ref{assump:SE1},
the following lemma is established.
The proof is given in Section~\ref{sec:proof-F-welldef-SE}.

\begin{lemma}
\label{lem:SE-F-welldef}
Assume that Assumption~\ref{assump:F} is fulfilled.
Let $d$ be a positive constant with $d<\pi$,
let $n$ be a positive integer, and let $h$ be
selected by the formula
\begin{equation}
h = \sqrt{\frac{\pi d}{n}}.
 \label{eq:h-SE-new}
\end{equation}
Moreover, let $M=N=n$,
and let $\aSEm$ be defined by~\eqref{eq:aSEm}.
Then, there exists a positive integer $n_r$ such that
for all $n\geq n_r$, $F(\aSEm)$ is well-defined.
\end{lemma}

Additionally, for the error analysis,
Stenger introduced the following function space.

\begin{definition}
\label{def:MC}
Let $\alpha$ and $\beta$ be positive constants
with $\alpha\leq 1$ and $\beta\leq 1$,
and let $\domD$ be a bounded and simply-connected domain
(or Riemann surface) that contains the interval $(a, b)$.
Then, $\MC_{\alpha,\beta}(\domD)$ denotes the family of all functions
$f\in\Hinf(\domD)$ that satisfy the
following two inequalities with a constant $H$:
\begin{align}
 |f(z) - f(a)| &\leq H |z - a|^{\alpha}, \label{leq:H-1}\\
 |f(b) - f(z)| &\leq H |b - z|^{\beta}, \label{leq:H-2}
\end{align}
for all $z\in\domD$.
\end{definition}

On the function $P(v,x)$ in~\eqref{eq:P-v-x},
he adopted the following assumption.

\begin{assumption}
\label{assump:SE2}
Let $\alpha$, $\beta$, $\alpha_f$, $\beta_f$,
$d$, and $\epsilon$ be positive constants
with $\alpha\leq 1$ and $\beta\leq 1$.
Let $P(v,x)$ be a function defined by~\eqref{eq:P-v-x}.
Assume that $P(v,\cdot)\in\MC_{\alpha,\beta}(\SEt(\domD_{d+\epsilon}))$
uniformly for $v\in[0, b-a]$.
Furthermore, assume that
there exists a constant $c_1$ independent of $v$ and $\tau$
such that
\[
 \left|
\frac{\partial}{\partial v}P(v,\tau)
\right|
\leq c_1 v^{\alpha_f - 1}\left((b-a) - v\right)^{\beta_f - 1}
\]
holds for all $v\in[0, b-a]$ and $\tau\in\SEt(\domD_d)$.
\end{assumption}

Under Assumptions~\ref{assump:SE1} and~\ref{assump:SE2},
the following result was presented.

\begin{theorem}[{\citet[Theorem~4.13]{stenger95:_colloc}}]
\label{thm:SE-old}
Assume that $g$ is analytic on $\SEt(\domD_d)$ for $d\in(0,\pi)$,
and absolutely integrable on the boundary of $\SEt(\domD_d)$.
Furthermore, assume that
Assumptions~\ref{assump:SE1} and~\ref{assump:SE2} are fulfilled.
Let $F$ be defined by~\eqref{eq:def-F}.
Let $\mu = \min\{\alpha, \beta\}$,
$n$ be a positive integer, and $h$ be
selected by the formula~\eqref{eq:h-SE}.
Moreover, let $M$ and $N$ be positive integers defined by~\eqref{eq:MN-SE}.
Then, there exists a constant $C$ independent of $n$ such that
\[
 \max_{x\in[a, b]}
\left|
p(x) - \left(F(\JmSE) g\right)(x)
\right|
\leq C \sqrt{n} \exp\left(-\sqrt{\pi d \mu n}\right),
\]
where $F(\JmSE) g$ is defined by~\eqref{eq:approx-p-SE}.
\end{theorem}

As an improvement of this theorem,
this study provides the following theorem.
The proof is given in Section~\ref{sec:proof-SE}.

\begin{theorem}
\label{thm:SE-new}
Assume that all the assumptions of
Theorem~\ref{thm:convolution-express} are fulfilled with $\domD=\SEt(\domD_d)$
for $d\in(0,\pi)$.
Let $n$ be a positive integer, and $h$ be
selected by the formula~\eqref{eq:h-SE-new}.
Moreover, let $M=N=n$.
Then, there exists a constant $C$ independent of $n$ such that
for all sufficiently large $n$
\[
 \max_{x\in[a, b]}
\left|
p(x) - \left(F(\JmSE) g\right)(x)
\right|
\leq C \log(n+1) \sqrt{n} \exp\left(-\sqrt{\pi d n}\right),
\]
where $F(\JmSE) g$ is defined by~\eqref{eq:approx-p-SE}.
\end{theorem}

The principal distinction
between Theorems~\ref{thm:SE-old} and~\ref{thm:SE-new}
lies in their assumptions.
In Assumption~\ref{assump:SE1},
all eigenvalues of $\aSEm$ must lie on $\Omega^{+}$,
and $\aSEm$ must be diagonalizable.
In Theorem~\ref{thm:SE-new},
such an assumption on $\aSEm$ is eliminated.
Furthermore, Assumption~\ref{assump:SE1},
$\hat{f}$ must not have any singular point on $\Omega^{+}$.
In Theorem~\ref{thm:SE-new},
such an assumption on $\hat{f}$ is eliminated.
In Assumption~\ref{assump:SE2},
some conditions are assumed on $P(v,x)$ defined by~\eqref{eq:P-v-x}.
These conditions are not easy to verify
because $P(v,x)$ is not a given function.
In Theorem~\ref{thm:SE-new},
such an assumption on $P(v,x)$ is eliminated.
Only the condition on $g$ in Theorem~\ref{thm:SE-old} is weaker
than that in Theorem~\ref{thm:SE-new}.
This is because it is not easy
to show the existence of the resolvent of $\mathcal{J}$
for the function space of $g$ supposed in Theorem~\ref{thm:SE-old},
whereas it is shown for $\Hinf(\SEt(\domD_d))$
by Proposition~\ref{prop:resolvent-Y}.

Disappearance of assumptions on $P(v,x)$ has an important aspect in
implementation. In Theorem~\ref{thm:SE-old},
the parameters $\alpha$, $\beta$ and $d$ of $P(v,x)$
are used in the formulas of $h$, $M$ and $N$.
These parameters are, however, not easy to determine in application
because $P(v,x)$ is not a given function.
In Theorem~\ref{thm:SE-new}, only one parameter $d$
of a \emph{given} function $g$ is needed for the computation of $h$.

\subsection{Sinc convolution combined with the DE transformation}
\label{sec:DE-Sinc-conv}

In Stenger's approximation formula~\eqref{eq:approx-p-SE},
the SE transformation is employed.
Replacing the transformation with the DE transformation,
this study derives another approximation formula for $p(x)$ as
\begin{equation}
\label{eq:approx-p-DE}
 p(x) = \left(F(\mathcal{J}) g\right)(x)
\approx \left(F(\JmDE) g\right)(x)
:= \boldsymbol{\omega}^{\DE}_m(x) F(\aDEm) \vDEm g,
\end{equation}
where $\boldsymbol{\omega}^{\DE}_m(x)$,
$\aDEm$, and $\vDEm$
are defined by~\eqref{eq:oDEm},~\eqref{eq:aDEm}
and~\eqref{eq:vDEm}, respectively.
This study shows the well-definedness of $F(\aDEm)$ as follows.
The proof is given in Section~\ref{sec:proof-F-welldef-DE}.

\begin{lemma}
\label{lem:DE-F-welldef}
Assume that Assumption~\ref{assump:F} is fulfilled.
Let $d$ be a positive constant with $d<\pi/2$,
let $n$ be a positive integer, and let $h$ be
selected by the formula
\begin{equation}
h = \frac{\log(2 d n)}{n}.
 \label{eq:h-DE-new}
\end{equation}
Moreover, let $M=N=n$,
and let $\aDEm$ be defined by~\eqref{eq:aDEm}.
Then, there exists a positive integer $n_r$ such that
for all $n\geq n_r$, $F(\aDEm)$ is well-defined.
\end{lemma}

Furthermore,
this study provides the error analysis of the formula~\eqref{eq:approx-p-DE}
as follows.
The proof is given in Section~\ref{sec:proof-DE}.

\begin{theorem}
\label{thm:DE-new}
Assume that all the assumptions of
Theorem~\ref{thm:convolution-express} are fulfilled with $\domD=\DEt(\domD_d)$.
Let $n$ be a positive integer, and $h$ be
selected by the formula~\eqref{eq:h-DE-new}.
Moreover, let $M=N=n$.
Then, there exists a constant $C$ independent of $n$ such that
for all sufficiently large $n$
\[
 \max_{x\in[a, b]}
\left|
p(x) - \left(F(\JmDE) g\right)(x)
\right|
\leq C \log(n+1)  \exp\left(\frac{-\pi d n}{\log(2 d n)}\right),
\]
where $F(\JmDE) g$ is defined by~\eqref{eq:approx-p-DE}.
\end{theorem}

Compared to Theorem~\ref{thm:SE-new},
we see that the convergence rate is considerably improved.

\section{Existence of the resolvent of approximated operators}
\label{sec:resolvent-approx-op}

The existence of the resolvent of $\mathcal{J}$
for the function space $\mathbf{W}=C([a,b])$
is shown by Proposition~\ref{prop:resolvent-W}.
The objective of this section is to show the
existence of the resolvent of
$\JmSE$ and $\JmDE$,
which are approximate operators of $\mathcal{J}$.

\subsection{Existence of the resolvent of \texorpdfstring{$\JmSE$}{JmSE}}

The goal of this subsection is to show the following lemma.

\begin{lemma}
\label{lem:resolvent-JmSE-exist}
Let $\JmSE:\mathbf{W}\to\mathbf{W}$ be the linear operator
defined by~\eqref{eq:JmSE}.
Let $d$ be a positive constant with $d<\pi$,
$n$ be a positive integer, and $h$ be
selected by the formula~\eqref{eq:h-SE-new}.
Moreover, let $M=N=n$.
Let $r$ be a positive constant.
Then, there exists a sufficiently large positive integer $n_r$
such that for all $n\geq n_r$,
$(z - \JmSE)^{-1}$ exists for all $|z|\geq r$ and is bounded as
\begin{equation}
\label{eq:Bound-Resolvent-JmSE}
 \left\|
(z - \JmSE)^{-1}
\right\|_{\mathcal{L}(\mathbf{W},\mathbf{W})}
\leq C_r\log(n+1),
\end{equation}
where $C_r$ is a positive constant independent of $n$ and $z$.
\end{lemma}

The following perturbation theorem is instrumental for this purpose.

\begin{theorem}[{\citet[Theorem~4.1.1]{atkinson97:_numer_solut}}]
\label{thm:Atkinson}
Assume the following conditions:
\begin{enumerate}
 \item[1.] Operators $\mathcal{X}$ and $\mathcal{X}_m$
are bounded operators on $\mathbf{W}$ to $\mathbf{W}$.
 \item[2.] Operator $(z - \mathcal{X}):\mathbf{W}\to\mathbf{W}$
has a bounded inverse
$(z - \mathcal{X})^{-1}:\mathbf{W}\to\mathbf{W}$.
 \item[3.] Operator $\mathcal{X}_m$ is compact on $\mathbf{W}$.
 \item[4.] The following inequality holds:
\begin{equation}
\label{eq:operator-converge-condition}
\|(\mathcal{X}-\mathcal{X}_m)\mathcal{X}_m\|_{\mathcal{L}(\mathbf{W},\mathbf{W})}
<\frac{|z|}{\|(z - \mathcal{X})^{-1}\|_{\mathcal{L}(\mathbf{W},\mathbf{W})}}.
\end{equation}
\end{enumerate}
Then, $(z - \mathcal{X}_m)^{-1}$ exists as a bounded operator
on $\mathbf{W}$ to $\mathbf{W}$, with
\begin{equation}
\|(z - \mathcal{X}_m)^{-1}\|_{\mathcal{L}(\mathbf{W},\mathbf{W})}
\leq \frac{1 + \|(z - \mathcal{X})^{-1}\|_{\mathcal{L}(\mathbf{W},\mathbf{W})}
\|\mathcal{X}_m\|_{\mathcal{L}(\mathbf{W},\mathbf{W})}}
{1 - \|(z - \mathcal{X})^{-1}\|_{\mathcal{L}(\mathbf{W},\mathbf{W})}\|(\mathcal{X}-\mathcal{X}_m)\mathcal{X}_m\|_{\mathcal{L}(\mathbf{W},\mathbf{W})}}.
\label{InEq:Bound-Inverse-Op}
\end{equation}
\end{theorem}

To apply this theorem, we verify that the four conditions
of this theorem are fulfilled with $\mathcal{X}=\mathcal{J}$
and $\mathcal{X}_m=\JmSE$
under the assumptions of Lemma~\ref{lem:resolvent-JmSE-exist}.
Condition~1 is trivially satisfied.
Condition~2 is already shown by Proposition~\ref{prop:resolvent-W}.
Condition~3 follows directly from the Arzel\`{a}--Ascoli theorem.
Condition~4 is the main task to prove Lemma~\ref{lem:resolvent-JmSE-exist}.
If we show that
\begin{equation}
 \|(\mathcal{J} - \JmSE)\JmSE\|_{\mathcal{L}(\mathbf{W},\mathbf{W})}\to 0
\label{eq:operator-converge-JmSE}
\end{equation}
as $n\to\infty$ (recall $m=2n+1$),
then we can conclude that for sufficiently large $n$,
the inequality~\eqref{eq:operator-converge-condition} holds
with $\mathcal{X}=\mathcal{J}$
and $\mathcal{X}_m=\JmSE$.

To show~\eqref{eq:operator-converge-JmSE},
let us first examine
\begin{align}
\nonumber
&\left\{(\mathcal{J} - \JmSE)\JmSE f\right\}(x)\\
\nonumber
&=\int_a^x\sum_{i=-M}^N
\left(
\sum_{j=-M}^N h\delta^{(-1)}_{i-j}f(\SEt(jh))(\SEt)'(jh)
\right)\oSE_i(t)\D{t}\\
\nonumber
&\quad - \sum_{i=-M}^N\left(
\sum_{j=-M}^N
\sum_{k=-M}^N h \delta^{(-1)}_{i-k}(\SEt)'(kh) h\delta^{(-1)}_{k-j}
f(\SEt(jh))(\SEt)'(jh)
\right)\oSE_{i}(x)\\
\nonumber
&=h\sum_{j=-M}^N f(\SEt(jh))(\SEt)'(jh)
\left\{(\mathcal{J} - \JmSE)\wSE_j\right\}(x),
\end{align}
where $\wSE_j(t)=\sum_{i=-M}^N \delta_{i-j}^{(-1)} \oSE_i(t)$.
In the final equality,
$\wSE_j(\SEt(kh))=\delta_{k-j}^{(-1)}$ is used
(recall $\oSE_i(\SEt(kh))=\delta_{ki}$).
According to Theorem~\ref{thm:Stenger-JmSE},
$(\mathcal{J} - \JmSE)\wSE_j$ seems to
converge with $\OO(\exp(-c\sqrt{n}))$.
However, we must note that the constant $C$
in Theorem~\ref{thm:Stenger-JmSE} depends on
the infinity norm of the integrand $g$,
as clarified in the following theorem.
This theorem is derived by examining
the structure of the constant $C$ of Theorem~\ref{thm:Stenger-JmSE}
with $\alpha=\beta=1$ in the proof.

\begin{theorem}
\label{thm:Stenger-JmSE-with-const}
Assume that $g$ is analytic on $\SEt(\domD_d)$ for $d\in (0,\pi)$,
and that there exist positive constants $K_d$
and $K_{0}$ such that
$\|g\|_{\Hinf(\SEt(\domD_d))}\leq K_d$
and
$\|g\|_{\mathbf{W}}\leq K_0$.
Let $n$ be a positive integer, and $h$ be
selected by the formula~\eqref{eq:h-SE-new}.
Moreover, let $M=N=n$.
Then, there exists a constant $C_{d}$ depending only on $d$,
and also exists a constant $C_{a,b}$ depending only on $a$ and $b$
such that
\begin{align*}
\max_{x\in[a, b]}\left|
\left(\mathcal{J}g\right)(x) - \left(\JmSE g\right)(x)
\right|
\leq
(b-a) C_{d} K_d \exp\left(-\sqrt{\pi d n}\right)
+ C_{a,b} K_0 \sqrt{n} \exp\left(-\sqrt{\pi d n}\right).
\end{align*}
\end{theorem}

In the case where $g = \wSE_j$,
$K_d$ diverges with
$\OO(\exp(\pi d/h)\log n)$
($K_0$ also diverges with $\OO(\log n)$,
which is comparatively negligible).
This behavior is established via the subsequent three lemmas.

\begin{lemma}[{\citet[Lemma~3.6.5]{stenger93:_numer}}]
\label{lem:bound-J}
Let $h>0$. Then, it holds that
\[
 \sup_{x\in \mathbb{R}}|J(j,h)(x)|\leq 1.1 h,
\]
where $J(j,h)(x)$ is defined by~\eqref{eq:basis-J}.
\end{lemma}

\begin{lemma}[{\citet[p.~142]{stenger93:_numer}}]
\label{lem:bound-sum-sinc-real}
Let $h>0$. Then, it holds that
\[
 \sup_{t\in \mathbb{R}}\sum_{j=-n}^n
 \left|\sinc\left(\frac{t-jh}{h}\right)\right|
\leq \frac{2}{\pi}(3 +\log n).
\]
\end{lemma}

\begin{lemma}
\label{lem:bound-sum-sinc-complex}
Let $h>0$. Then, it holds for all $y\in\mathbb{R}$ that
\[
 \sup_{t\in \mathbb{R}}\sum_{j=-n}^n
 \left|\sinc\left(\frac{(t+\I y)-jh}{h}\right)\right|
\leq\frac{2\cosh(\pi y/h)}{\pi}
\left(\pi + 1 +\log n\right).
\]
\end{lemma}
\begin{proof}
First, it holds that
\begin{align}
\label{eq:sinc-sum-complex-1st}
\sum_{j=-n}^n
 \left|
\sinc\left(\frac{(t+\I y)-jh}{h}\right)
\right|
&=\sum_{j=-n}^n
\sqrt{
\frac{\sinh^2(\pi y/h) + \sin^2(\pi (t - jh)/h)}
{(\pi(t - jh)/h)^2 + (\pi y/h)^2}
},
\end{align}
which is even in $t$. Therefore,
without loss of generality,
we restrict attention to $t\geq 0$.
For any fixed $y\in\mathbb{R}$,
we establish an upper bound for the summation in the case $0\leq t\leq h$.
This is because the sum in the case $kh\leq t\leq (k+1)h$ ($k\geq 1$)
is less than that in the case $0\leq t\leq h$.
When $j=0$, we have
\begin{align*}
\frac{\sinh^2(\pi y/h) + \sin^2(\pi t /h)}{(\pi t /h)^2 + (\pi y/h)^2}
&\leq
\frac{\sinh^2(\pi y/h) + \sin^2(\pi t /h)}
{\sin^2(\pi t /h) + (\pi y/h)^2},
\end{align*}
which has its maximum at $t = 0$ and $t=h$.
In the same way,
when $j=1$, we have
\begin{align*}
\frac{\sinh^2(\pi y/h) + \sin^2(\pi (t - h) /h)}
{(\pi (t - h) /h)^2 + (\pi y/h)^2}
&\leq
\frac{\sinh^2(\pi y/h) + \sin^2(\pi (t - h) /h)}
{\sin^2(\pi (t - h) /h) + (\pi y/h)^2},
\end{align*}
which has its maximum at $t = 0$ and $t=h$.
Consequently,
the right-hand side of~\eqref{eq:sinc-sum-complex-1st} is bounded as
\begin{align*}
\sum_{j=-n}^n
\sqrt{
\frac{\sinh^2(\pi y/h) + \sin^2(\pi (t - jh)/h)}
{(\pi(t - jh)/h)^2 + (\pi y/h)^2}
}
&\leq \frac{\sinh(\pi y/h)}{\pi y/h}
+ \frac{\sinh(\pi y/h)}{\pi y/h}\\
&\quad +
\sum_{j=1}^{n}
\sqrt{\frac{\sinh^2(\pi y/h) + \sin^2(\pi (t + jh) /h)}
{(\pi (t + jh) /h)^2 + (\pi y/h)^2}}\\
&\quad +\sum_{j=1}^{n-1}
\sqrt{\frac{\sinh^2(\pi y/h) + \sin^2(\pi (t - (j+1)h) /h)}
{(\pi (t - (j+1)h) /h)^2 + (\pi y/h)^2}}.
\end{align*}
On the first $\sum$, it holds that
\begin{align*}
\sum_{j=1}^{n}
\sqrt{\frac{\sinh^2(\pi y/h) + \sin^2(\pi (t + jh) /h)}
{(\pi (t + jh) /h)^2 + (\pi y/h)^2}}
&\leq
\sum_{j=1}^{n}
\sqrt{\frac{\sinh^2(\pi y/h) + 1}
{(\pi (0 + jh) /h)^2 + 0}}\\
&=\sum_{j=1}^{n}\frac{\cosh(\pi y/h)}{\pi j}.
\end{align*}
On the second $\sum$, it holds that
\begin{align*}
\sum_{j=1}^{n-1}
\sqrt{\frac{\sinh^2(\pi y/h) + \sin^2(\pi (t - (j+1)h) /h)}
{(\pi (t - (j+1)h) /h)^2 + (\pi y/h)^2}}
&\leq \sum_{j=1}^{n-1}
\sqrt{\frac{\sinh^2(\pi y/h) + 1}
{(\pi (h - (j+1)h) /h)^2 + 0}}\\
&=\sum_{j=1}^{n-1}\frac{\cosh(\pi y/h)}{\pi j}\\
&\leq\sum_{j=1}^{n}\frac{\cosh(\pi y/h)}{\pi j}.
\end{align*}
Furthermore, using
$\sinh(\pi y/h)/(\pi y/h)\leq \cosh(\pi y/h)$ along
with the following estimate
\begin{align*}
 \sum_{j=1}^{n}\frac{1}{j}
=1 + \sum_{j=2}^{n}\frac{1}{j}
\leq 1 + \int_1^n\frac{1}{x}\D{x}
= 1 + \log n,
\end{align*}
we arrive at the desired inequality.
\end{proof}

From Lemma~\ref{lem:bound-J}, we have $|\delta_{i-j}^{(-1)}|\leq 1.1$
because $J(j,h)(ih)=h\delta_{i-j}^{(-1)}$.
From Lemma~\ref{lem:bound-sum-sinc-complex},
$\sum_{i=-M}^N |\oSE_i(t)|$ may diverge with
$\OO(\exp(\pi d/h)\log n)$.
Consequently, $K_d=\|\wSE_j\|_{\Hinf(\SEt(\domD_d))}$
may diverge with $\OO(\exp(\pi d/h)\log n)$,
and setting $h$ as~\eqref{eq:h-SE-new} we have
$K_d\exp(-\sqrt{\pi d n})=\OO(\log n)\to \infty$
as $n\to\infty$.
Accordingly,
we cannot show $(\mathcal{J} - \JmSE)\wSE_j\to 0$
by Theorem~\ref{thm:Stenger-JmSE-with-const} in the current form.

To address this issue,
let us rewrite $(\mathcal{J} - \JmSE)\JmSE f$ as
\begin{align}
\nonumber
&\left\{(\mathcal{J} - \JmSE)\JmSE f\right\}(x)\\
\nonumber
&=h\sum_{j=-M}^N f(\SEt(jh))(\SEt)'(jh)
\biggl[
\left\{(\mathcal{J} - \tJmSE)\wSE_j\right\}(x)\biggr.\\
&\quad\quad\biggl.
-\frac{1}{h}\left\{(\mathcal{J} - \tJmSE)J(j,h)\circ\SEi\right\}(x)
+\frac{1}{h}\left\{(\mathcal{J} - \JmSE)J(j,h)\circ\SEi\right\}(x)
\biggr],
\label{eq:rewrite-target-SE}
\end{align}
where $\tJmSE$ is defined by~\eqref{eq:tJmSE}.
We then estimate the three terms in the square brackets
of~\eqref{eq:rewrite-target-SE} individually.
For the first and second terms,
we use the following theorem.
This theorem is derived by examining
the structure of the constant $C$ of Theorem~\ref{thm:Okayama-tJmSE}
with $\alpha=\beta=1$ in the proof.

\begin{theorem}
\label{thm:Okayama-tJmSE-with-const}
Assume that $g$ is analytic on $\SEt(\domD_d)$ for $d\in(0,\pi)$,
and that there exist positive constants
$K_d$ and $K_0$ such that
$\|g\|_{\Hinf(\SEt(\domD_d))}\leq K_d$
and
$\|g\|_{\mathbf{W}}\leq K_0$.
Let $n$ be a positive integer, and $h$ be
selected by the formula~\eqref{eq:h-SE-new}.
Moreover, let $M=N=n$.
Then, there exists a constant $\tilde{C}_d$
depending only on $d$ such that
\begin{align*}
 \max_{x\in [a,b]}\left|
\left(\mathcal{J} g\right)(x) - \left(\tJmSE g\right)(x)
\right| \leq 2(b - a)\left[\tilde{C}_d K_d \frac{1}{\sqrt{n}}
\exp\left(-\sqrt{\pi d n}\right)
+ 1.1 K_0 \exp\left(-\sqrt{\pi d n}\right)
\right].
\end{align*}
\end{theorem}

In this theorem, the convergence rate is a bit higher than
that of Theorem~\ref{thm:Stenger-JmSE-with-const},
which allows us to obtain
\[
 \max_{a\leq x\leq b}
\left|\left\{(\mathcal{J} - \tJmSE)\wSE_j\right\}(x)\right|
\leq \tilde{C}_1 \frac{\log(n+1)}{\sqrt{n}}
\]
with a certain constant $\tilde{C}_1$
($\log n$ is replaced with $\log(n+1)$ to ensure that the right-hand side
remains nonzero when $n=1$).
Thus, we can show $(\mathcal{J} - \tJmSE)\wSE_j\to 0$
as $n\to\infty$.

Next, we show the convergence of the second term
in the square brackets of~\eqref{eq:rewrite-target-SE}.
For this purpose, we require the bounds of
$\|J(j,h)\circ\SEt\|_{\Hinf(\SEt(\domD_d))}$
and
$\|J(j,h)\circ\SEt\|_{\mathbf{W}}$,
which are obtained by the following lemma
and Lemma~\ref{lem:bound-J}, respectively.

\begin{lemma}[{\citet[Lemma~6.4]{okayama15:_theor_sinc_nystr_volter}}]
Let $h>0$. Then, it holds for all $y\in\mathbb{R}$ that
\[
 \sup_{x\in\mathbb{R}}|J(j,h)(x+\I y)|\leq
\frac{5h}{\pi}\cdot\frac{\sinh(\pi y/h)}{\pi y/h}.
\]
\end{lemma}

From these lemmas, we have
$K_d = \|J(j,h)\circ\SEt\|_{\Hinf(\SEt(\domD_d))} = \OO(h^2\exp(\pi d/h))$
and
$K_0 = \|J(j,h)\circ\SEt\|_{\mathbf{W}} = \OO(h)$.
Therefore, by Theorem~\ref{thm:Okayama-tJmSE-with-const},
setting $h$ as~\eqref{eq:h-SE-new} we obtain
\[
\max_{a\leq x\leq b} \left|
\frac{1}{h}\left\{(\mathcal{J} - \tJmSE)J(j,h)\circ\SEi\right\}(x)
\right|
\leq \tilde{C}_2 \frac{1}{n}
\]
with a certain constant $\tilde{C}_2$.
This inequality shows the convergence of the second term.

The only thing remaining is to bound the third term.
Using Theorem~\ref{thm:Stenger-JmSE-with-const}
with $K_d = \OO(h^2\exp(\pi d/h))$ and $K_0 = \OO(h)$,
setting $h$ as~\eqref{eq:h-SE-new} we obtain
\[
\max_{a\leq x\leq b} \left|
\frac{1}{h}\left\{(\mathcal{J} - \JmSE)J(j,h)\circ\SEi\right\}(x)
\right|
\leq \tilde{C}_3 \frac{1}{\sqrt{n}}
\]
with a certain constant $\tilde{C}_3$.

Finally, using the following estimate
\begin{align}
\nonumber
 h\sum_{j=-M}^N (\SEt)'(jh)
&\leq h \cdot (\SEt)'(0) + 2h\sum_{j=1}^n(\SEt)'(jh)\\
\nonumber
&\leq \sqrt{\frac{\pi d}{1}}\cdot\frac{b-a}{4}+2\int_0^{\infty}(\SEt)'(x)\D{x}\\
&= \sqrt{\pi d}\cdot\frac{b-a}{4} + (b - a),
\label{eq:bound-sum-SEtDiv}
\end{align}
we establish the following lemma.

\begin{lemma}
Assume that all the assumptions
of Lemma~\ref{lem:resolvent-JmSE-exist} are fulfilled.
Let $f\in\mathbf{W}$.
Then, there exists a constant $C$ independent of $n$ and $f$
such that
\[
 \|(\mathcal{J} - \JmSE)\JmSE f\|_{\mathbf{W}}
\leq C \|f\|_{\mathbf{W}} \frac{\log(n+1)}{\sqrt{n}}.
\]
\end{lemma}

From this lemma, we have the bound of the operator norm as
\[
 \|(\mathcal{J} - \JmSE)\JmSE\|_{\mathcal{L}(\mathbf{W},\mathbf{W})}
\leq C\frac{\log(n+1)}{\sqrt{n}},
\]
which shows~\eqref{eq:operator-converge-JmSE}.
Thus, there exists a positive integer $n(z)$ such that
for all $n\geq n(z)$
the inequality~\eqref{eq:operator-converge-condition} holds
with $\mathcal{X}=\mathcal{J}$
and $\mathcal{X}_m=\JmSE$.
Here, $n(z)$ depends on $z$.
In reality, we can choose $n_r$ independent of $z$
so that~\eqref{eq:operator-converge-condition} holds
for all $z$ with $|z|\geq r$.
This is shown as follows.
Pick $n_r$ such that for all $n\geq n_r$
\[
 \|(\mathcal{J} - \JmSE)\JmSE\|_{\mathcal{L}(\mathbf{W},\mathbf{W})}
< \frac{r^2}{\E^{(b-a)/r}}.
\]
In the same way as~\eqref{eq:bound-resolvent},
we have
\[
\frac{1}{|z|} \|(z - \mathcal{J})^{-1}\|_{\mathcal{L}(\mathbf{W},\mathbf{W})}
\leq \frac{1}{|z|^2}\E^{(b-a)/|z|}.
\]
Because the right-hand side is monotonically decreasing
with respect to $|z|$, it holds for $|z|\geq r$ that
\begin{equation}
\label{eq:operator-less-than-1-SE}
 \|(\mathcal{J} - \JmSE)\JmSE\|_{\mathcal{L}(\mathbf{W},\mathbf{W})}
\cdot \frac{1}{|z|}
\|(z - \mathcal{J})^{-1}\|_{\mathcal{L}(\mathbf{W},\mathbf{W})}
< \frac{r^2}{\E^{(b-a)/r}}
\cdot \frac{1}{|z|^2}\E^{(b-a)/|z|}
\leq 1,
\end{equation}
which implies that~\eqref{eq:operator-converge-condition} holds
for all $z$ with $|z|\geq r$.

The final task to prove Lemma~\ref{lem:resolvent-JmSE-exist}
is to estimate the right-hand side of~\eqref{InEq:Bound-Inverse-Op}.
By the same argument as above,
there exists a constant $\tilde{C}_r$
independent of $n$ and $z$ such that
\[
 \frac{1}
{1 - \|(z - \mathcal{J})^{-1}\|_{\mathcal{L}(\mathbf{W},\mathbf{W})}
\|(\mathcal{J} - \JmSE)\JmSE\|_{\mathcal{L}(\mathbf{W},\mathbf{W})}
}
\leq \tilde{C}_r.
\]
In addition, we have
\begin{equation}
\label{eq:bound-J-resolvent}
 \|(z - \mathcal{J})^{-1}\|_{\mathcal{L}(\mathbf{W},\mathbf{W})}
\leq \frac{1}{|z|}\E^{(b-a)/|z|}
\leq \frac{1}{r}\E^{(b-a)/r}
\end{equation}
for all $z$ with $|z|\geq r$.
As for the remaining term $\|\JmSE\|_{\mathbf{W}}$,
using Lemma~\ref{lem:bound-sum-sinc-real},
we have
\begin{align*}
 |(\JmSE g)(x)|
&=\left|
\sum_{i=-M}^N\left(
h\sum_{j=-M}^N \delta_{i-j}^{(-1)} g(\SEt(jh))(\SEt)'(jh)
\right)\oSE_i(x)
\right|\\
&\leq \|g\|_{\mathbf{W}}h\sum_{j=-M}^N(\SEt)'(jh) |\wSE_j(x)|\\
&\leq \|g\|_{\mathbf{W}}
\left(h\sum_{j=-M}^N(\SEt)'(jh)\right)\cdot \hat{C}\log(n+1),
\end{align*}
with a certain constant $\hat{C}$.
Furthermore, using~\eqref{eq:bound-sum-SEtDiv},
we obtain
$\|\JmSE\|_{\mathcal{L}(\mathbf{W},\mathbf{W})}\leq \{1 + \sqrt{\pi d}/4\}(b-a)\hat{C}\log(n+1)$,
and thus~\eqref{eq:Bound-Resolvent-JmSE} holds
with a certain constant $C_r$.
This completes the proof of Lemma~\ref{lem:resolvent-JmSE-exist}.

\subsection{Existence of the resolvent of \texorpdfstring{$\JmDE$}{JmDE}}

The goal of this subsection is to show the following lemma.

\begin{lemma}
\label{lem:resolvent-JmDE-exist}
Let $\JmDE:\mathbf{W}\to\mathbf{W}$ be the linear operator
defined by~\eqref{eq:JmDE}.
Let $d$ be a positive constant with $d<\pi/2$,
$n$ be a positive integer, and $h$ be
selected by the formula~\eqref{eq:h-DE-new}.
Moreover, let $M=N=n$.
Let $r$ be a positive constant.
Then, there exists a sufficiently large positive integer $n_r$
such that for all $n\geq n_r$,
$(z - \JmDE)^{-1}$ exists for all $|z|\geq r$ and is bounded as
\begin{equation}
 \label{eq:Bound-Resolvent-JmDE}
 \left\|
(z - \JmDE)^{-1}
\right\|_{\mathcal{L}(\mathbf{W},\mathbf{W})}
\leq C_r\log(n+1),
\end{equation}
where $C_r$ is a positive constant independent of $n$ and $z$.
\end{lemma}

We prove this lemma in the same way as Lemma~\ref{lem:resolvent-JmSE-exist}.
First, we show
\begin{equation}
\label{eq:operator-converge-JmDE}
 \|(\mathcal{J} - \JmDE)\JmDE\|_{\mathcal{L}(\mathbf{W},\mathbf{W})}\to 0
\end{equation}
as $n\to\infty$.
For the purpose,
let us rewrite $(\mathcal{J} - \JmDE)\JmDE f$ as
\begin{align}
\nonumber
&\left\{(\mathcal{J} - \JmDE)\JmDE f\right\}(x)\\
\nonumber
&=h\sum_{j=-M}^N f(\DEt(jh))(\DEt)'(jh)
\left[
\left\{(\mathcal{J} - \tJmDE)\wDE_j\right\}(x)\right .\\
&\quad\quad\left.
-\frac{1}{h}\left\{(\mathcal{J} - \tJmDE)J(j,h)\circ\DEi\right\}(x)
+\frac{1}{h}\left\{(\mathcal{J} - \JmDE)J(j,h)\circ\DEi\right\}(x)
\right],
\label{eq:rewrite-target-DE}
\end{align}
where $\tJmDE$ is defined by~\eqref{eq:tJmDE},
and $\wDE_j(t)=\sum_{i=-M}^N \delta_{i-j}^{(-1)} \oDE_i(t)$.
We then estimate three terms in the square brackets
of~\eqref{eq:rewrite-target-DE} individually.
For the first and second terms,
we use the following theorem.
This theorem is derived by examining
the structure of the constant $C$ of Theorem~\ref{thm:Okayama-tJmDE}
with $\alpha=\beta=1$ in the proof.

\begin{theorem}
\label{thm:Okayama-tJmDE-with-const}
Assume that $g$ is analytic on $\DEt(\domD_d)$ for $d\in(0,\pi/2)$,
and that there exist positive constants
$K_d$ and $K_0$ such that
$\|g\|_{\Hinf(\DEt(\domD_d))}\leq K_d$
and
$\|g(z)\|_{\mathbf{W}}\leq K_0$.
Let $n$ be a positive integer, and $h$ be
selected by the formula~\eqref{eq:h-DE-new}.
Moreover, let $M=N=n$.
Then, there exists a constant $\tilde{C}_d$
depending only on $d$ such that
\begin{align*}
 \max_{x\in [a,b]}\left|
\left(\mathcal{J} g\right)(x) - \left(\tJmDE g\right)(x)
\right|
 \leq 2(b - a)\left[\tilde{C}_d K_d \frac{\log(2 d n)}{n}
\exp\left(\frac{-\pi d n}{\log(2 d n)}\right)
+ 1.1 K_0 \exp\left(-\pi d n\right)
\right].
\end{align*}
\end{theorem}

In the case where $g = \wDE_j$,
$K_d$ diverges with
$\OO(\exp(\pi d/h)\log n)$,
and $K_0$ also diverges with $\OO(\log n)$.
That is shown in the same way as $g = \wSE_j$.
Therefore, by Theorem~\ref{thm:Okayama-tJmDE-with-const},
setting $h$ as~\eqref{eq:h-DE-new}
we obtain
\[
 \max_{a\leq x\leq b}
\left|
\left\{(\mathcal{J} - \tJmDE)\wDE_j\right\}(x)
\right|
\leq \tilde{C}_1 \log(n+1) \cdot \frac{\log(2dn)}{n}
\]
with a certain constant $\tilde{C}_1$.
In the case where $g = J(j,h)\circ\DEi$,
we have $K_d=\OO(h^2\exp(\pi d/h))$
and $K_0=\OO(h)$ in the same way as $g = J(j,h)\circ\SEi$.
Therefore, by Theorem~\ref{thm:Okayama-tJmDE-with-const},
setting $h$ as~\eqref{eq:h-DE-new}
we obtain
\[
  \max_{a\leq x\leq b}
\left|\frac{1}{h}
\left\{(\mathcal{J} - \tJmDE)J(j,h)\circ\DEi\right\}(x)
\right|
\leq \tilde{C}_2 \left\{\frac{\log(2dn)}{n}\right\}^2
\]
with a certain constant $\tilde{C}_2$.

For the third term,
we use the following theorem.
This theorem is derived by examining
the structure of the constant $C$ of Theorem~\ref{thm:Okayama-JmDE}
with $\alpha=\beta=1$ in the proof.

\begin{theorem}
\label{thm:Okayama-JmDE-with-const}
Assume that $g$ is analytic on $\DEt(\domD_d)$ for $d\in (0,\pi/2)$,
and that there exist positive constants $K_d$
and $K_{0}$ such that
$\|g\|_{\Hinf(\DEt(\domD_d))}\leq K_d$
and
$\|g(z)\|_{\mathbf{W}}\leq K_0$.
Let $n$ be a positive integer, and $h$ be
selected by the formula~\eqref{eq:h-DE-new}.
Moreover, let $M=N=n$.
Then, there exists a constant $C_{d}$ depending only on $d$,
and also exists a constant $C_{a,b}$ depending only on $a$ and $b$
such that
\begin{align*}
\max_{x\in[a, b]}\left|
\left(\mathcal{J}g\right)(x) - \left(\JmDE g\right)(x)
\right|
\leq (b-a)C_d K_d \exp\left(\frac{-\pi d n}{\log(2 d n)}\right)
+ C_{a,b} K_0 \exp\left(-\sqrt{\pi d n}\right).
\end{align*}
\end{theorem}

Using Theorem~\ref{thm:Okayama-JmDE-with-const}
with $K_d = \OO(h^2\exp(\pi d/h))$ and $K_0 = \OO(h)$,
setting $h$ as~\eqref{eq:h-DE-new} we obtain
\[
\max_{a\leq x\leq b} \left|
\frac{1}{h}\left\{(\mathcal{J} - \JmDE)J(j,h)\circ\DEi\right\}(x)
\right|
\leq \tilde{C}_3 \frac{\log(2 d n)}{n}
\]
with a certain constant $\tilde{C}_3$.

Finally, using the following estimate
\begin{align}
\nonumber
 h\sum_{j=-M}^N (\DEt)'(jh)
&\leq 2 d\cdot
\frac{\log(2 d n)}{2 d n}\cdot(\DEt)'(0) + 2h\sum_{j=1}^n(\DEt)'(jh)\\
\nonumber
&\leq 2 d\cdot \frac{\log \E}{\E}\cdot\frac{\pi}{4}(b-a)
+2\int_0^{\infty}(\DEt)'(x)\D{x}\\
&= \frac{\pi d}{2\E}(b-a) + (b - a),
\label{eq:bound-sum-DEtDiv}
\end{align}
we establish the following lemma.

\begin{lemma}
Assume that all the assumptions
of Lemma~\ref{lem:resolvent-JmDE-exist} are fulfilled.
Let $f\in\mathbf{W}$.
Then, there exists a constant $C$ independent of $n$ and $f$
such that
\[
 \|(\mathcal{J} - \JmDE)\JmDE f\|_{\mathbf{W}}
\leq C \|f\|_{\mathbf{W}} \log(n+1)\cdot \frac{\log(2 d n)}{n}.
\]
\end{lemma}

From this lemma, we have the bound of the operator norm as
\[
 \|(\mathcal{J} - \JmDE)\JmDE\|_{\mathcal{L}(\mathbf{W},\mathbf{W})}
\leq C\log(n+1)\cdot \frac{\log(2 d n)}{n},
\]
which shows~\eqref{eq:operator-converge-JmDE}.
Thus, for all sufficiently large $n$,
we can use Theorem~\ref{thm:Atkinson}
with $\mathcal{X}=\mathcal{J}$
and $\mathcal{X}_m=\JmDE$.

Pick $n_r$ such that for all $n\geq n_r$
\[
 \|(\mathcal{J} - \JmDE)\JmDE\|_{\mathcal{L}(\mathbf{W},\mathbf{W})}
< \frac{r^2}{\E^{(b-a)/r}}.
\]
Then, in the same way as~\eqref{eq:operator-less-than-1-SE},
it holds for $|z|\geq r$ that
\[
 \|(\mathcal{J} - \JmDE)\JmDE\|_{\mathcal{L}(\mathbf{W},\mathbf{W})}
\cdot \frac{1}{|z|}
\|(z - \mathcal{J})^{-1}\|_{\mathcal{L}(\mathbf{W},\mathbf{W})}
< \frac{r^2}{\E^{(b-a)/r}}
\cdot \frac{1}{|z|^2}\E^{(b-a)/|z|}
\leq 1,
\]
which implies that~\eqref{eq:operator-converge-condition} holds
for all $z$ with $|z|\geq r$.

The final task to prove Lemma~\ref{lem:resolvent-JmDE-exist}
is to estimate the right-hand side of~\eqref{InEq:Bound-Inverse-Op}.
By the same argument as above,
there exists a constant $\tilde{C}_r$
independent of $n$ and $z$ such that
\[
 \frac{1}
{1 - \|(z - \mathcal{J})^{-1}\|_{\mathcal{L}(\mathbf{W},\mathbf{W})}
\|(\mathcal{J} - \JmDE)\JmDE\|_{\mathcal{L}(\mathbf{W},\mathbf{W})}
}
\leq \tilde{C}_r.
\]
In addition, we already know that
$\|(z - \mathcal{J})^{-1}\|_{\mathcal{L}(\mathbf{W},\mathbf{W})}$
is bounded as~\eqref{eq:bound-J-resolvent}
for all $z$ with $|z|\geq r$.
As for the remaining term $\|\JmDE\|_{\mathbf{W}}$,
using Lemma~\ref{lem:bound-sum-sinc-real},
we have
\begin{align*}
 |(\JmDE g)(x)|
&=\left|
\sum_{i=-M}^N\left(
h\sum_{j=-M}^N \delta_{i-j}^{(-1)} g(\DEt(jh))(\DEt)'(jh)
\right)\oDE_i(x)
\right|\\
&\leq \|g\|_{\mathbf{W}}h\sum_{j=-M}^N(\DEt)'(jh) |\wDE_j(x)|\\
&\leq \|g\|_{\mathbf{W}}
\left(h\sum_{j=-M}^N(\DEt)'(jh)\right)\cdot \hat{C}\log(n+1),
\end{align*}
with a certain constant $\hat{C}$.
Furthermore, using~\eqref{eq:bound-sum-DEtDiv},
we obtain
$\|\JmDE\|_{\mathcal{L}(\mathbf{W},\mathbf{W})}\leq\{1+(\pi d)/(2\E)\}(b-a)\hat{C}\log(n+1)$,
and thus~\eqref{eq:Bound-Resolvent-JmDE} holds
with a certain constant $C_r$.
This completes the proof of Lemma~\ref{lem:resolvent-JmDE-exist}.

\section{Analysis of matrices appearing in the Sinc convolution}
\label{sec:proof-F-welldef}

This section provides the proof of Lemmas~\ref{lem:SE-F-welldef}
and~\ref{lem:DE-F-welldef}.

\subsection{Analysis in the case of the SE transformation}
\label{sec:proof-F-welldef-SE}

As for the proof of Lemma~\ref{lem:SE-F-welldef},
the following lemma is essential,
which shows that the resolvent sets $\rho(\JmSE)$ and
$\rho(\aSEm)$ are equivalent.

\begin{lemma}
\label{lem:existence-unique-SE}
Let $\JmSE$ be the linear operator defined by~\eqref{eq:JmSE},
and let $\aSEm$ be an $m\times m$ matrix defined by~\eqref{eq:aSEm}.
Let $v\in C([a,b])$, and consider the following two equations:
\begin{align}
 (z - \JmSE)u(x)             &= v(x)\quad (a\leq x\leq b),
 \label{eq:continuous-SE} \\
 (z - \aSEm)\boldsymbol{c}_m &= \vvSEm, \label{eq:discrete-SE}
\end{align}
where
$\vvSEm = \vSEm v = [v(\SEt(-Mh)),\,\ldots,\,v(\SEt(Nh))]^{\mathrm{T}}$.
Then, the following statements are equivalent:
\begin{enumerate}
 \item[(A)] Equation~\eqref{eq:continuous-SE} has a unique solution
$u\in C([a,b])$.
 \item[(B)] Equation~\eqref{eq:discrete-SE} has a unique solution
$\boldsymbol{c}_m\in\mathbb{R}^m$.
\end{enumerate}
\end{lemma}
\begin{proof}
We begin by proving that (A) $\Rightarrow$ (B).
If we define $\boldsymbol{c}_m$ as
$\boldsymbol{c}_m = \vSEm u$,
then the equation~\eqref{eq:discrete-SE} holds.
That is, $\boldsymbol{c}_m$ is a solution of~\eqref{eq:discrete-SE}.
To prove the uniqueness, suppose that
$\tilde{\boldsymbol{c}}_m=[\tilde{c}_{-M},\,\ldots,\,\tilde{c}_N]^{\mathrm{T}}$
is another solution of~\eqref{eq:discrete-SE}.
Using the vector $\tilde{\boldsymbol{c}}_m$,
let us define $\tilde{u}\in C([a,b])$ as
\begin{equation}
\label{eq:tilde-u-define}
 \tilde{u}(x) = \frac{1}{z}
\left\{
v(x) + \sum_{k=-M}^N
\left(h\sum_{j=-M}^N\delta^{(-1)}_{k-j}\tilde{c}_j(\SEt)'(jh)\right)
 \oSE_k(x)
\right\}.
\end{equation}
Because
$\oSE_k(\SEt(ih)) = \delta_{ki}$ $(k,i = -M,\,\ldots,\,N)$ holds,
it holds on the points $x=\SEt(ih)$ that
\[
\tilde{u}(\SEt(ih)) =
\frac{1}{z}\left\{
v(\SEt(ih)) + h\sum_{j=-M}^N\delta^{(-1)}_{i-j}\tilde{c}_j(\SEt)'(jh)
\right\}
\quad (i=-M,\,\ldots,\,N).
\]
On the other hand, because $\tilde{\boldsymbol{c}}_m$
is a solution of~\eqref{eq:discrete-SE}, it holds that
\[
\tilde{c}_i =
\frac{1}{z}\left\{
v(\SEt(ih)) + h\sum_{j=-M}^N\delta^{(-1)}_{i-j}\tilde{c}_j(\SEt)'(jh)
\right\}
\quad (i=-M,\,\ldots,\,N).
\]
Consequently, $\tilde{u}(\SEt(ih))=\tilde{c}_i$ holds.
Therefore, the equation~\eqref{eq:tilde-u-define}
can be rewritten as $(z - \JmSE)\tilde{u} = w$,
which shows that $\tilde{u}$ is a solution of~\eqref{eq:continuous-SE}.
Because the solution of~\eqref{eq:continuous-SE} is unique,
$u = \tilde{u}$ holds, which implies
$\boldsymbol{c}_m=\tilde{\boldsymbol{c}}_m$.
This shows the desired uniqueness.

Next, we prove that (B) $\Rightarrow$ (A).
Using the vector $\boldsymbol{c}_m$ (the solution of~\eqref{eq:discrete-SE}),
let us define $u\in C([a,b])$ as
\begin{equation}
\label{eq:u-define}
 u(x) = \frac{1}{z}
\left\{
v(x) + \sum_{k=-M}^N
\left(h\sum_{j=-M}^N\delta^{(-1)}_{k-j}{c}_j(\SEt)'(jh)\right)
 \oSE_k(x)
\right\}.
\end{equation}
By the same argument as above,
$u(\SEt(ih))=c_i$ holds, from which we have $(z - \JmSE)u = v$.
Therefore, $u$ is a solution of~\eqref{eq:continuous-SE}.
To prove the uniqueness, suppose that
$\tilde{u}$ is another solution of~\eqref{eq:continuous-SE}.
If we define $\tilde{\boldsymbol{c}}_m$ as
$\tilde{\boldsymbol{c}}_m = \vSEm \tilde{u}$,
then $\tilde{\boldsymbol{c}}_m$ is a solution of~\eqref{eq:discrete-SE}.
Because the solution of~\eqref{eq:discrete-SE} is unique,
$\boldsymbol{c}_m=\tilde{\boldsymbol{c}}_m$ holds.
Consequently, $\tilde{u}(\SEt(jh))=c_j$ holds.
Therefore, the equation $(z - \JmSE)\tilde{u} = v$ can be rewritten as
\begin{equation}
\label{eq:tilde-u-SE}
\tilde{u}(x)
= \frac{1}{z}
\left\{
v(x) + \sum_{k=-M}^N
\left(h\sum_{j=-M}^N\delta^{(-1)}_{k-j}{c}_j(\SEt)'(jh)\right)
 \oSE_k(x)
\right\}.
\end{equation}
Comparing the right-hand sides of~\eqref{eq:u-define}
and~\eqref{eq:tilde-u-SE},
we conclude $u = \tilde{u}$, which shows the desired uniqueness.
\end{proof}

Using this lemma, we prove Lemma~\ref{lem:SE-F-welldef} as follows.

\begin{proof}[Proof of Lemma~\ref{lem:SE-F-welldef}]
Let $r$ be a positive constant such that $F$ has no singular point
inside $\varGamma_{r}$.
Lemma~\ref{lem:resolvent-JmSE-exist} shows that
there exists a sufficiently large positive integer $n_r$ such that
for all $n\geq n_r$,
$(z - \JmSE)^{-1}:C([a,b])\to C([a,b])$ exists
for all $|z|\geq r$.
According to Lemma~\ref{lem:existence-unique-SE},
$(z - \aSEm)^{-1}:\mathbb{R}^m\to\mathbb{R}^m$
exists if and only if
$(z - \JmSE)^{-1}:C([a,b])\to C([a,b])$ exists.
Therefore, for all $n\geq n_r$,
$(z - \aSEm)^{-1}$ exists for all $|z|\geq r$.
Thus, we can consider the Dunford integral along $\varGamma_r$ as
\[
 \frac{1}{2\pi\I}\oint_{\varGamma_r} (z - \aSEm)^{-1} F(z)\D{z}
= F(\aSEm),
\]
which shows the existence of $F(\aSEm)$.
This completes the proof.
\end{proof}

In the preceding proof, the Dunford integral is used
for showing the well-definedness of $F(\aSEm)$,
but we can show it more directly from the definition of matrix functions.
To this end, let us review the definition of matrix functions.

\begin{definition}[cf.~{\citet[Definition~1.2]{Higham}}]
\label{def:matrix-function}
Let $f$ be defined on the spectrum of $A\in \mathbb{C}^{m\times m}$
and let $A$ have the Jordan canonical form
\begin{align*}
 Z^{-1} A Z &= J = \diag[J_1,\,J_2,\,\ldots,\,J_q],\\
 J_k & =
\begin{bmatrix}
 \lambda_k & 1       &       & \\
           &\lambda_k&\ddots & \\
           &         &\ddots & 1 \\
           &         &       & \lambda_k
\end{bmatrix}\in\mathbb{C}^{m_k\times m_k},
\end{align*}
where $Z$ is nonsingular and $m_1 + m_2 + \cdots + m_q = m$.
Then,
\[
 f(A) := Z f(J) Z^{-1} = Z \diag(f(J_k)) Z^{-1},
\]
where
\[
 f(J_k) :=
\begin{bmatrix}
f(\lambda_k)&f'(\lambda_k)&\cdots&\dfrac{f^{(m_k-1)}(\lambda_k)}{(m_k - 1)!} \\
            & f(\lambda_k)&\ddots& \vdots \\
            &             &\ddots& f'(\lambda_k) \\
            &             &      & f(\lambda_k)
\end{bmatrix}.
\]
\end{definition}

Recall that $F$ is analytic, and accordingly infinitely differentiable,
inside $\varGamma_r$.
According to Definition~\ref{def:matrix-function},
if the spectrum $\sigma(\aSEm)$ is located inside $\varGamma_{r}$,
then $F(\aSEm)$ is well-defined,
because $F^{(m_k-1)}(\lambda_k)$ exists for any $k$.
In the same way of the proof of Lemma~\ref{lem:SE-F-welldef} provided above,
we can show that
there exists a sufficiently large positive integer $n_r$ such that
for all $n\geq n_r$,
$(z - \aSEm)^{-1}$ exists for all $|z|\geq r$.
This implies that for all $n\geq n_r$,
$\sigma(\aSEm)$ is located inside $\varGamma_{r}$.
Thus, $F(\aSEm)$ is well-defined for all $n\geq n_r$.
This completes another proof of the well-definedness of $F(\aSEm)$.

We find further information on the spectrum of $\aSEm$.
Recall that $\hat{f}(s)$ has no singular point outside $\varGamma_{R}$
(assumption of Lemma~\ref{lem:Bromwich-contour}),
and accordingly $F(s)=\hat{f}(1/s)$ is analytic inside $\varGamma_{r}$,
where $r=1/R$.
Here, note that we can choose $r>0$ as small as we want.
Lemma~\ref{lem:SE-F-welldef} shows that
even for arbitrary small $r$,
there exists a sufficiently large positive integer $n_r$ such that
for all $n\geq n_r$,
the spectrum $\sigma(\aSEm)$ lies inside $\varGamma_{r}$.
Thus, we arrive at the following result.

\begin{cor}
\label{cor:SE-new}
Let $d$ be a positive constant with $d<\pi$,
let $n$ be a positive integer, and let $h$ be
selected by the formula~\eqref{eq:h-SE-new}.
Moreover, let $M=N=n$,
and let $\aSEm$ be defined by~\eqref{eq:aSEm}.
Let an arbitrary small positive value $\epsilon$ be given.
Then, there exists a positive integer $n_{\epsilon}$ such that
for all $n\geq n_{\epsilon}$, $\sigma(\aSEm)$ lies inside
$\varGamma_{\epsilon}$.
\end{cor}

This result shows that $\sigma(\aSEm)$ contracts toward the origin
as $n\to\infty$.
Because $\sigma(\JmSE)=\sigma(\aSEm)$,
this is also true for $\sigma(\JmSE)$,
i.e., $\sigma(\JmSE)\to \{0\}$ as $n\to\infty$.
This behavior is consistent with Proposition~\ref{prop:resolvent-W},
which says $\sigma(\mathcal{J})=\{0\}$.

\subsection{Analysis in the case of the DE transformation}
\label{sec:proof-F-welldef-DE}

As for the proof of Lemma~\ref{lem:DE-F-welldef},
the following lemma is essential,
which shows that the resolvent sets $\rho(\JmDE)$ and
$\rho(\aDEm)$ are equivalent.

\begin{lemma}
\label{lem:existence-unique-DE}
Let $\JmDE$ be the linear operator defined by~\eqref{eq:JmDE},
and let $\aDEm$ be an $m\times m$ matrix defined by~\eqref{eq:aDEm}.
Let $v\in C([a,b])$, and consider the following two equations:
\begin{align}
 (z - \JmDE)u(x)             &= v(x)\quad (a\leq x\leq b),
 \label{eq:continuous-DE} \\
 (z - \aDEm)\boldsymbol{c}_m &= \vvDEm, \label{eq:discrete-DE}
\end{align}
where
$\vvDEm = \vDEm v = [v(\DEt(-Mh)),\,\ldots,\,v(\DEt(Nh))]^{\mathrm{T}}$.
Then, the following statements are equivalent:
\begin{enumerate}
 \item[(A)] Equation~\eqref{eq:continuous-DE} has a unique solution
$u\in C([a,b])$.
 \item[(B)] Equation~\eqref{eq:discrete-DE} has a unique solution
$\boldsymbol{c}_m\in\mathbb{R}^m$.
\end{enumerate}
\end{lemma}

The proof of Lemma~\ref{lem:existence-unique-DE} goes in exactly the same way
as that of Lemma~\ref{lem:existence-unique-DE}.
Using this lemma,
Lemma~\ref{lem:DE-F-welldef} is also proved
in exactly the same way
as that of Lemma~\ref{lem:SE-F-welldef}.
Therefore, the proofs for the two lemmas are omitted here.

In addition,
by the same argument as in the case of the SE transformation,
we derive the following result.

\begin{cor}
\label{cor:DE-new}
Let $d$ be a positive constant with $d<\pi/2$,
let $n$ be a positive integer, and let $h$ be
selected by the formula~\eqref{eq:h-DE-new}.
Moreover, let $M=N=n$,
and let $\aDEm$ be defined by~\eqref{eq:aDEm}.
Let an arbitrary small positive value $\epsilon$ be given.
Then, there exists a positive integer $n_{\epsilon}$ such that
for all $n\geq n_{\epsilon}$, $\sigma(\aDEm)$ lies inside
$\varGamma_{\epsilon}$.
\end{cor}

This result shows that $\sigma(\aDEm)$ contracts toward the origin
as $n\to\infty$.
Because $\sigma(\JmDE)=\sigma(\aDEm)$,
this is also true for $\sigma(\JmDE)$,
i.e., $\sigma(\JmDE)\to \{0\}$ as $n\to\infty$.

\section{Proof of convergence theorems}
\label{sec:proof-convergence}

We are now prepared to prove
Theorems~\ref{thm:SE-new} and~\ref{thm:DE-new}.

\subsection{Proof of convergence in the case of the SE transformation}
\label{sec:proof-SE}

\begin{proof}[Proof of Theorem~\ref{thm:SE-new}]
Let $r$ be a positive constant
such that $F$ has no singular point inside $\varGamma_r$.
According to
Proposition~\ref{prop:resolvent-W}
and
Lemma~\ref{lem:resolvent-JmSE-exist},
there exists a positive integer $n_r$ such that
for all $n\geq n_r$,
$\sigma(\mathcal{J})$
and $\sigma(\JmSE)$ locate inside $\varGamma_r$.
Then, using the Dunford integral, we have
\begin{align*}
 F(\mathcal{J})g - F(\JmSE)g
&=\frac{1}{2\pi\I}\oint_{\varGamma_r}F(z)
\left\{(z - \mathcal{J})^{-1} - (z - \JmSE)^{-1}\right\} \D{z}\, g\\
&=\frac{1}{2\pi\I}\oint_{\varGamma_r}F(z)
(z - \JmSE)^{-1}(\mathcal{J} - \JmSE)(z - \mathcal{J})^{-1} \D{z}\, g,
\end{align*}
where the second resolvent identity is used at the second equality.
Taking the uniform norm over $[a, b]$, we have
\begin{align*}
&\|F(\mathcal{J})g - F(\JmSE)g\|_{\mathbf{W}}\\
&\leq
\max_{z\in\varGamma_r}|F(z)|\cdot
\|(z - \JmSE)^{-1}\|_{\mathcal{L}(\mathbf{W},\mathbf{W})}
\|(\mathcal{J}-\JmSE)(z - \mathcal{J})^{-1} g\|_{\mathbf{W}}
\cdot\frac{1}{2\pi}
\oint_{\varGamma_r}|\D{z}|\\
&=\max_{z\in\varGamma_r}|F(z)|\cdot
\|(z - \JmSE)^{-1}\|_{\mathcal{L}(\mathbf{W},\mathbf{W})}
\|(\mathcal{J}-\JmSE)\tilde{g}\|_{\mathbf{W}} \cdot r,
\end{align*}
where $\tilde{g}=(z - \mathcal{J})^{-1}g$.
From Proposition~\ref{prop:resolvent-Y},
$\tilde{g}\in\Hinf(\SEt(\domD_d))$.
Therefore, by Theorem~\ref{thm:Stenger-JmSE} with $\alpha=\beta=1$
(alternatively, by Theorem~\ref{thm:Stenger-JmSE-with-const}),
we have
\[
 \|(\mathcal{J}-\JmSE)\tilde{g}\|_{\mathbf{W}}
\leq C\sqrt{n}\exp\left(-\sqrt{\pi d n}\right)
\]
with a certain constant $C$.
Finally, using~\eqref{eq:Bound-Resolvent-JmSE},
we obtain the conclusion.
\end{proof}

\subsection{Proof of convergence in the case of the DE transformation}
\label{sec:proof-DE}

\begin{proof}[Proof of Theorem~\ref{thm:DE-new}]
Let $r$ be a positive constant
such that $F$ has no singular point inside $\varGamma_r$.
According to
Proposition~\ref{prop:resolvent-W}
and
Lemma~\ref{lem:resolvent-JmDE-exist},
there exists a positive integer $n_r$ such that
for all $n\geq n_r$,
$\sigma(\mathcal{J})$
and $\sigma(\JmDE)$ locate inside $\varGamma_r$.
Then, using the Dunford integral, we have
\begin{align*}
 F(\mathcal{J})g - F(\JmDE)g
&=\frac{1}{2\pi\I}\oint_{\varGamma_r}F(z)
\left\{(z - \mathcal{J})^{-1} - (z - \JmDE)^{-1}\right\} \D{z}\, g\\
&=\frac{1}{2\pi\I}\oint_{\varGamma_r}F(z)
(z - \JmDE)^{-1}(\mathcal{J} - \JmDE)(z - \mathcal{J})^{-1} \D{z}\, g,
\end{align*}
where the second resolvent identity is used at the second equality.
Taking the uniform norm over $[a, b]$, we have
\begin{align*}
&\|F(\mathcal{J})g - F(\JmDE)g\|_{\mathbf{W}}\\
&\leq
\max_{z\in\varGamma_r}|F(z)|\cdot
\|(z - \JmDE)^{-1}\|_{\mathcal{L}(\mathbf{W},\mathbf{W})}
\|(\mathcal{J}-\JmDE)(z - \mathcal{J})^{-1} g\|_{\mathbf{W}}
\cdot\frac{1}{2\pi}
\oint_{\varGamma_r}|\D{z}|\\
&=\max_{z\in\varGamma_r}|F(z)|\cdot
\|(z - \JmDE)^{-1}\|_{\mathcal{L}(\mathbf{W},\mathbf{W})}
\|(\mathcal{J}-\JmDE)\tilde{g}\|_{\mathbf{W}} \cdot r,
\end{align*}
where $\tilde{g}=(z - \mathcal{J})^{-1}g$.
From Proposition~\ref{prop:resolvent-Y},
$\tilde{g}\in\Hinf(\DEt(\domD_d))$.
Therefore, by Theorem~\ref{thm:Okayama-JmDE} with $\alpha=\beta=1$
(alternatively, by Theorem~\ref{thm:Okayama-JmDE-with-const}),
we have
\[
 \|(\mathcal{J}-\JmDE)\tilde{g}\|_{\mathbf{W}}
\leq C\exp\left(\frac{-\pi d n}{\log(2 d n)}\right)
\]
with a certain constant $C$.
Finally, using~\eqref{eq:Bound-Resolvent-JmDE},
we obtain the conclusion.
\end{proof}

\section{Numerical examples}
\label{sec:numerical}

This section presents numerical results for approximation
of $p(x)$ in~\eqref{eq:p-x} by Stenger's and the improved
formulas. Hereafter, the former is referred to as
the ``SE-Sinc convolution,'' while the latter is
referred to as the ``DE-Sinc convolution.''
All the programs were implemented in C++ with double-precision
floating-point arithmetic.
The source code for all programs is publicly available at
\url{https://github.com/okayamat/sinc-conv}.
In all the examples given below,
the interval $[a, b]$ is set as $[0, 2]$,
and maximum error among 200 equally spaced points over the
interval is plotted on the graph.

\subsection{Comparison of \texorpdfstring{$g$}{g}: \texorpdfstring{$g$}{g} has a pole or not}

First, let us fix $F(s)$ (i.e., $f(x)$), and compare the difference
of the performance due to the presence of singularity in $g(t)$.

\begin{example}
\label{ex:1}
Consider the following indefinite convolution
\[
\int_0^x (x - t)\sqrt{t}\D{t} = \frac{4}{15}x^{5/2},
\quad 0\leq x\leq 2.
\]
In this case,
$\hat{f}(s) = 1/s^2$ and $F(s)=\hat{f}(1/s)=s^2$.
\end{example}

\begin{example}
\label{ex:2}
Consider the following indefinite convolution
\begin{align*}
\int_0^x (x - t)\frac{\sqrt{t}}{1+t^2}\D{t}
 &= \frac{x+1}{\sqrt{2}}
\left\{\arctan\left(\sqrt{2x}+1\right)
+\arctan\left(\sqrt{2x}-1\right)\right\}\\
&\quad +
 \frac{x-1}{2\sqrt{2}}\log\left(\frac{x-\sqrt{2x}+1}{x+\sqrt{2x}+1}\right)
-2\sqrt{x},
\quad 0\leq x\leq 2.
\end{align*}
In this case,
$\hat{f}(s) = 1/s^2$ and $F(s)=\hat{f}(1/s)=s^2$.
\end{example}

In the case of Example~\ref{ex:1},
$g(t)$ does not have any singularity excluding the endpoint $t=0$.
Therefore,
the assumptions in Theorem~\ref{thm:SE-new}
are fulfilled with $d=3.14$ (slightly less than $\pi$),
and those in Theorem~\ref{thm:DE-new}
are fulfilled with $d=1.57$ (slightly less than $\pi/2$).
The numerical results are shown in Fig.~\ref{fig:ex1}.
From the graph, we observe that the SE-Sinc convolution
attains $\OO(\exp(-c\sqrt{m}))$,
and the DE-Sinc convolution
attains $\OO(\exp(-\tilde{c} m/\log m))$.
In the case of Example~\ref{ex:2},
$g(t)$ has poles at $t=\pm\I$,
from which $g(\SEt(u))$ has poles at $u=-\log(\sqrt{2})\pm (3\pi/4)\I$.
Theorefore,
the assumptions in Theorem~\ref{thm:SE-new}
are fulfilled with $d=2.35$ (slightly less than $3\pi/4$).
In a similar manner,
the assumptions in Theorem~\ref{thm:DE-new}
are fulfilled with $d=0.833$.
The numerical results are shown in Fig.~\ref{fig:ex2}.
From the graph, we observe that
the SE-Sinc and DE-Sinc convolutions still attain
$\OO(\exp(-c\sqrt{m}))$ and $\OO(\exp(-\tilde{c} m/\log m))$,
respectively, although $c$ and $\tilde{c}$ get smaller than
those in Example~\ref{ex:1}.
This is attributed to the reduced values of $d$.

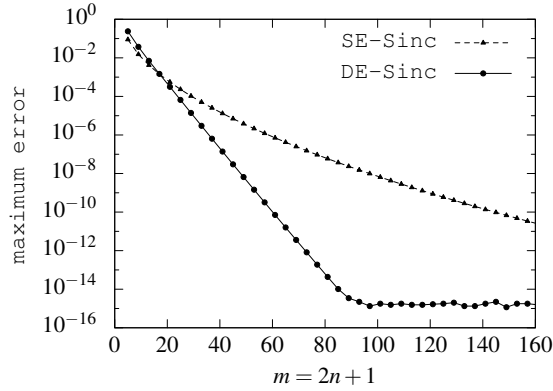
\begin{figure}
\centering
\scalebox{0.6}{
\begin{tikzpicture}[gnuplot]
\tikzset{every node/.append style={font={\ttfamily\fontsize{14.4pt}{17.28pt}\selectfont}}}
\gpmonochromelines
\path (0.000,0.000) rectangle (12.500,8.750);
\gpcolor{color=gp lt color border}
\gpsetlinetype{gp lt border}
\gpsetdashtype{gp dt solid}
\gpsetlinewidth{1.00}
\draw[gp path] (2.431,1.420)--(2.611,1.420);
\draw[gp path] (11.704,1.420)--(11.524,1.420);
\node[gp node right] at (2.166,1.420) {$10^{-16}$};
\draw[gp path] (2.431,1.845)--(2.521,1.845);
\draw[gp path] (11.704,1.845)--(11.614,1.845);
\draw[gp path] (2.431,2.270)--(2.611,2.270);
\draw[gp path] (11.704,2.270)--(11.524,2.270);
\node[gp node right] at (2.166,2.270) {$10^{-14}$};
\draw[gp path] (2.431,2.695)--(2.521,2.695);
\draw[gp path] (11.704,2.695)--(11.614,2.695);
\draw[gp path] (2.431,3.120)--(2.611,3.120);
\draw[gp path] (11.704,3.120)--(11.524,3.120);
\node[gp node right] at (2.166,3.120) {$10^{-12}$};
\draw[gp path] (2.431,3.545)--(2.521,3.545);
\draw[gp path] (11.704,3.545)--(11.614,3.545);
\draw[gp path] (2.431,3.970)--(2.611,3.970);
\draw[gp path] (11.704,3.970)--(11.524,3.970);
\node[gp node right] at (2.166,3.970) {$10^{-10}$};
\draw[gp path] (2.431,4.395)--(2.521,4.395);
\draw[gp path] (11.704,4.395)--(11.614,4.395);
\draw[gp path] (2.431,4.820)--(2.611,4.820);
\draw[gp path] (11.704,4.820)--(11.524,4.820);
\node[gp node right] at (2.166,4.820) {$10^{-8}$};
\draw[gp path] (2.431,5.244)--(2.521,5.244);
\draw[gp path] (11.704,5.244)--(11.614,5.244);
\draw[gp path] (2.431,5.669)--(2.611,5.669);
\draw[gp path] (11.704,5.669)--(11.524,5.669);
\node[gp node right] at (2.166,5.669) {$10^{-6}$};
\draw[gp path] (2.431,6.094)--(2.521,6.094);
\draw[gp path] (11.704,6.094)--(11.614,6.094);
\draw[gp path] (2.431,6.519)--(2.611,6.519);
\draw[gp path] (11.704,6.519)--(11.524,6.519);
\node[gp node right] at (2.166,6.519) {$10^{-4}$};
\draw[gp path] (2.431,6.944)--(2.521,6.944);
\draw[gp path] (11.704,6.944)--(11.614,6.944);
\draw[gp path] (2.431,7.369)--(2.611,7.369);
\draw[gp path] (11.704,7.369)--(11.524,7.369);
\node[gp node right] at (2.166,7.369) {$10^{-2}$};
\draw[gp path] (2.431,7.794)--(2.521,7.794);
\draw[gp path] (11.704,7.794)--(11.614,7.794);
\draw[gp path] (2.431,8.219)--(2.611,8.219);
\draw[gp path] (11.704,8.219)--(11.524,8.219);
\node[gp node right] at (2.166,8.219) {$10^{0}$};
\draw[gp path] (2.431,1.420)--(2.431,1.600);
\draw[gp path] (2.431,8.219)--(2.431,8.039);
\node[gp node center] at (2.431,0.976) {$0$};
\draw[gp path] (3.590,1.420)--(3.590,1.600);
\draw[gp path] (3.590,8.219)--(3.590,8.039);
\node[gp node center] at (3.590,0.976) {$20$};
\draw[gp path] (4.749,1.420)--(4.749,1.600);
\draw[gp path] (4.749,8.219)--(4.749,8.039);
\node[gp node center] at (4.749,0.976) {$40$};
\draw[gp path] (5.908,1.420)--(5.908,1.600);
\draw[gp path] (5.908,8.219)--(5.908,8.039);
\node[gp node center] at (5.908,0.976) {$60$};
\draw[gp path] (7.068,1.420)--(7.068,1.600);
\draw[gp path] (7.068,8.219)--(7.068,8.039);
\node[gp node center] at (7.068,0.976) {$80$};
\draw[gp path] (8.227,1.420)--(8.227,1.600);
\draw[gp path] (8.227,8.219)--(8.227,8.039);
\node[gp node center] at (8.227,0.976) {$100$};
\draw[gp path] (9.386,1.420)--(9.386,1.600);
\draw[gp path] (9.386,8.219)--(9.386,8.039);
\node[gp node center] at (9.386,0.976) {$120$};
\draw[gp path] (10.545,1.420)--(10.545,1.600);
\draw[gp path] (10.545,8.219)--(10.545,8.039);
\node[gp node center] at (10.545,0.976) {$140$};
\draw[gp path] (11.704,1.420)--(11.704,1.600);
\draw[gp path] (11.704,8.219)--(11.704,8.039);
\node[gp node center] at (11.704,0.976) {$160$};
\draw[gp path] (2.431,8.219)--(2.431,1.420)--(11.704,1.420)--(11.704,8.219)--cycle;
\node[gp node center,rotate=-270] at (0.354,4.819) {maximum error};
\node[gp node center] at (7.067,0.310) {$m=2n+1$};
\node[gp node right] at (9.669,7.701) {SE-Sinc};
\gpsetdashtype{gp dt 2}
\draw[gp path] (9.934,7.701)--(11.174,7.701);
\draw[gp path] (2.721,7.773)--(2.953,7.447)--(3.184,7.209)--(3.416,7.010)--(3.648,6.835)%
  --(3.880,6.675)--(4.112,6.528)--(4.344,6.393)--(4.575,6.266)--(4.807,6.144)--(5.039,6.028)%
  --(5.271,5.917)--(5.503,5.810)--(5.735,5.706)--(5.966,5.606)--(6.198,5.509)--(6.430,5.414)%
  --(6.662,5.323)--(6.894,5.233)--(7.125,5.146)--(7.357,5.061)--(7.589,4.978)--(7.821,4.897)%
  --(8.053,4.818)--(8.285,4.740)--(8.516,4.663)--(8.748,4.588)--(8.980,4.514)--(9.212,4.441)%
  --(9.444,4.370)--(9.676,4.299)--(9.907,4.231)--(10.139,4.162)--(10.371,4.095)--(10.603,4.028)%
  --(10.835,3.963)--(11.066,3.898)--(11.298,3.835)--(11.530,3.771)--(11.704,3.725);
\gpsetpointsize{4.00}
\gppoint{gp mark 9}{(2.721,7.773)}
\gppoint{gp mark 9}{(2.953,7.447)}
\gppoint{gp mark 9}{(3.184,7.209)}
\gppoint{gp mark 9}{(3.416,7.010)}
\gppoint{gp mark 9}{(3.648,6.835)}
\gppoint{gp mark 9}{(3.880,6.675)}
\gppoint{gp mark 9}{(4.112,6.528)}
\gppoint{gp mark 9}{(4.344,6.393)}
\gppoint{gp mark 9}{(4.575,6.266)}
\gppoint{gp mark 9}{(4.807,6.144)}
\gppoint{gp mark 9}{(5.039,6.028)}
\gppoint{gp mark 9}{(5.271,5.917)}
\gppoint{gp mark 9}{(5.503,5.810)}
\gppoint{gp mark 9}{(5.735,5.706)}
\gppoint{gp mark 9}{(5.966,5.606)}
\gppoint{gp mark 9}{(6.198,5.509)}
\gppoint{gp mark 9}{(6.430,5.414)}
\gppoint{gp mark 9}{(6.662,5.323)}
\gppoint{gp mark 9}{(6.894,5.233)}
\gppoint{gp mark 9}{(7.125,5.146)}
\gppoint{gp mark 9}{(7.357,5.061)}
\gppoint{gp mark 9}{(7.589,4.978)}
\gppoint{gp mark 9}{(7.821,4.897)}
\gppoint{gp mark 9}{(8.053,4.818)}
\gppoint{gp mark 9}{(8.285,4.740)}
\gppoint{gp mark 9}{(8.516,4.663)}
\gppoint{gp mark 9}{(8.748,4.588)}
\gppoint{gp mark 9}{(8.980,4.514)}
\gppoint{gp mark 9}{(9.212,4.441)}
\gppoint{gp mark 9}{(9.444,4.370)}
\gppoint{gp mark 9}{(9.676,4.299)}
\gppoint{gp mark 9}{(9.907,4.231)}
\gppoint{gp mark 9}{(10.139,4.162)}
\gppoint{gp mark 9}{(10.371,4.095)}
\gppoint{gp mark 9}{(10.603,4.028)}
\gppoint{gp mark 9}{(10.835,3.963)}
\gppoint{gp mark 9}{(11.066,3.898)}
\gppoint{gp mark 9}{(11.298,3.835)}
\gppoint{gp mark 9}{(11.530,3.771)}
\gppoint{gp mark 9}{(10.554,7.701)}
\node[gp node right] at (9.669,7.026) {DE-Sinc};
\gpsetdashtype{gp dt solid}
\draw[gp path] (9.934,7.026)--(11.174,7.026);
\draw[gp path] (2.721,7.954)--(2.953,7.612)--(3.184,7.304)--(3.416,7.015)--(3.648,6.730)%
  --(3.880,6.442)--(4.112,6.155)--(4.344,5.871)--(4.575,5.587)--(4.807,5.304)--(5.039,5.022)%
  --(5.271,4.742)--(5.503,4.463)--(5.735,4.184)--(5.966,3.907)--(6.198,3.631)--(6.430,3.356)%
  --(6.662,3.083)--(6.894,2.811)--(7.125,2.541)--(7.357,2.274)--(7.589,2.073)--(7.821,1.992)%
  --(8.053,1.898)--(8.285,1.951)--(8.516,1.926)--(8.748,1.951)--(8.980,1.926)--(9.212,1.926)%
  --(9.444,1.939)--(9.676,1.951)--(9.907,1.973)--(10.139,1.898)--(10.371,1.898)--(10.603,1.951)%
  --(10.835,1.992)--(11.066,1.873)--(11.298,1.951)--(11.530,1.951)--(11.704,1.932);
\gppoint{gp mark 7}{(2.721,7.954)}
\gppoint{gp mark 7}{(2.953,7.612)}
\gppoint{gp mark 7}{(3.184,7.304)}
\gppoint{gp mark 7}{(3.416,7.015)}
\gppoint{gp mark 7}{(3.648,6.730)}
\gppoint{gp mark 7}{(3.880,6.442)}
\gppoint{gp mark 7}{(4.112,6.155)}
\gppoint{gp mark 7}{(4.344,5.871)}
\gppoint{gp mark 7}{(4.575,5.587)}
\gppoint{gp mark 7}{(4.807,5.304)}
\gppoint{gp mark 7}{(5.039,5.022)}
\gppoint{gp mark 7}{(5.271,4.742)}
\gppoint{gp mark 7}{(5.503,4.463)}
\gppoint{gp mark 7}{(5.735,4.184)}
\gppoint{gp mark 7}{(5.966,3.907)}
\gppoint{gp mark 7}{(6.198,3.631)}
\gppoint{gp mark 7}{(6.430,3.356)}
\gppoint{gp mark 7}{(6.662,3.083)}
\gppoint{gp mark 7}{(6.894,2.811)}
\gppoint{gp mark 7}{(7.125,2.541)}
\gppoint{gp mark 7}{(7.357,2.274)}
\gppoint{gp mark 7}{(7.589,2.073)}
\gppoint{gp mark 7}{(7.821,1.992)}
\gppoint{gp mark 7}{(8.053,1.898)}
\gppoint{gp mark 7}{(8.285,1.951)}
\gppoint{gp mark 7}{(8.516,1.926)}
\gppoint{gp mark 7}{(8.748,1.951)}
\gppoint{gp mark 7}{(8.980,1.926)}
\gppoint{gp mark 7}{(9.212,1.926)}
\gppoint{gp mark 7}{(9.444,1.939)}
\gppoint{gp mark 7}{(9.676,1.951)}
\gppoint{gp mark 7}{(9.907,1.973)}
\gppoint{gp mark 7}{(10.139,1.898)}
\gppoint{gp mark 7}{(10.371,1.898)}
\gppoint{gp mark 7}{(10.603,1.951)}
\gppoint{gp mark 7}{(10.835,1.992)}
\gppoint{gp mark 7}{(11.066,1.873)}
\gppoint{gp mark 7}{(11.298,1.951)}
\gppoint{gp mark 7}{(11.530,1.951)}
\gppoint{gp mark 7}{(10.554,7.026)}
\draw[gp path] (2.431,8.219)--(2.431,1.420)--(11.704,1.420)--(11.704,8.219)--cycle;
\gpdefrectangularnode{gp plot 1}{\pgfpoint{2.431cm}{1.420cm}}{\pgfpoint{11.704cm}{8.219cm}}
\end{tikzpicture}
}
\caption{Numerical results for Example~\ref{ex:1}.}
\label{fig:ex1}
\end{figure}
\begin{figure}
\centering
\scalebox{0.6}{
\begin{tikzpicture}[gnuplot]
\tikzset{every node/.append style={font={\ttfamily\fontsize{14.4pt}{17.28pt}\selectfont}}}
\gpmonochromelines
\path (0.000,0.000) rectangle (12.500,8.750);
\gpcolor{color=gp lt color border}
\gpsetlinetype{gp lt border}
\gpsetdashtype{gp dt solid}
\gpsetlinewidth{1.00}
\draw[gp path] (2.431,1.420)--(2.611,1.420);
\draw[gp path] (11.704,1.420)--(11.524,1.420);
\node[gp node right] at (2.166,1.420) {$10^{-16}$};
\draw[gp path] (2.431,1.845)--(2.521,1.845);
\draw[gp path] (11.704,1.845)--(11.614,1.845);
\draw[gp path] (2.431,2.270)--(2.611,2.270);
\draw[gp path] (11.704,2.270)--(11.524,2.270);
\node[gp node right] at (2.166,2.270) {$10^{-14}$};
\draw[gp path] (2.431,2.695)--(2.521,2.695);
\draw[gp path] (11.704,2.695)--(11.614,2.695);
\draw[gp path] (2.431,3.120)--(2.611,3.120);
\draw[gp path] (11.704,3.120)--(11.524,3.120);
\node[gp node right] at (2.166,3.120) {$10^{-12}$};
\draw[gp path] (2.431,3.545)--(2.521,3.545);
\draw[gp path] (11.704,3.545)--(11.614,3.545);
\draw[gp path] (2.431,3.970)--(2.611,3.970);
\draw[gp path] (11.704,3.970)--(11.524,3.970);
\node[gp node right] at (2.166,3.970) {$10^{-10}$};
\draw[gp path] (2.431,4.395)--(2.521,4.395);
\draw[gp path] (11.704,4.395)--(11.614,4.395);
\draw[gp path] (2.431,4.820)--(2.611,4.820);
\draw[gp path] (11.704,4.820)--(11.524,4.820);
\node[gp node right] at (2.166,4.820) {$10^{-8}$};
\draw[gp path] (2.431,5.244)--(2.521,5.244);
\draw[gp path] (11.704,5.244)--(11.614,5.244);
\draw[gp path] (2.431,5.669)--(2.611,5.669);
\draw[gp path] (11.704,5.669)--(11.524,5.669);
\node[gp node right] at (2.166,5.669) {$10^{-6}$};
\draw[gp path] (2.431,6.094)--(2.521,6.094);
\draw[gp path] (11.704,6.094)--(11.614,6.094);
\draw[gp path] (2.431,6.519)--(2.611,6.519);
\draw[gp path] (11.704,6.519)--(11.524,6.519);
\node[gp node right] at (2.166,6.519) {$10^{-4}$};
\draw[gp path] (2.431,6.944)--(2.521,6.944);
\draw[gp path] (11.704,6.944)--(11.614,6.944);
\draw[gp path] (2.431,7.369)--(2.611,7.369);
\draw[gp path] (11.704,7.369)--(11.524,7.369);
\node[gp node right] at (2.166,7.369) {$10^{-2}$};
\draw[gp path] (2.431,7.794)--(2.521,7.794);
\draw[gp path] (11.704,7.794)--(11.614,7.794);
\draw[gp path] (2.431,8.219)--(2.611,8.219);
\draw[gp path] (11.704,8.219)--(11.524,8.219);
\node[gp node right] at (2.166,8.219) {$10^{0}$};
\draw[gp path] (2.431,1.420)--(2.431,1.600);
\draw[gp path] (2.431,8.219)--(2.431,8.039);
\node[gp node center] at (2.431,0.976) {$0$};
\draw[gp path] (3.590,1.420)--(3.590,1.600);
\draw[gp path] (3.590,8.219)--(3.590,8.039);
\node[gp node center] at (3.590,0.976) {$20$};
\draw[gp path] (4.749,1.420)--(4.749,1.600);
\draw[gp path] (4.749,8.219)--(4.749,8.039);
\node[gp node center] at (4.749,0.976) {$40$};
\draw[gp path] (5.908,1.420)--(5.908,1.600);
\draw[gp path] (5.908,8.219)--(5.908,8.039);
\node[gp node center] at (5.908,0.976) {$60$};
\draw[gp path] (7.068,1.420)--(7.068,1.600);
\draw[gp path] (7.068,8.219)--(7.068,8.039);
\node[gp node center] at (7.068,0.976) {$80$};
\draw[gp path] (8.227,1.420)--(8.227,1.600);
\draw[gp path] (8.227,8.219)--(8.227,8.039);
\node[gp node center] at (8.227,0.976) {$100$};
\draw[gp path] (9.386,1.420)--(9.386,1.600);
\draw[gp path] (9.386,8.219)--(9.386,8.039);
\node[gp node center] at (9.386,0.976) {$120$};
\draw[gp path] (10.545,1.420)--(10.545,1.600);
\draw[gp path] (10.545,8.219)--(10.545,8.039);
\node[gp node center] at (10.545,0.976) {$140$};
\draw[gp path] (11.704,1.420)--(11.704,1.600);
\draw[gp path] (11.704,8.219)--(11.704,8.039);
\node[gp node center] at (11.704,0.976) {$160$};
\draw[gp path] (2.431,8.219)--(2.431,1.420)--(11.704,1.420)--(11.704,8.219)--cycle;
\node[gp node center,rotate=-270] at (0.354,4.819) {maximum error};
\node[gp node center] at (7.067,0.310) {$m=2n+1$};
\node[gp node right] at (9.669,7.701) {SE-Sinc};
\gpsetdashtype{gp dt 2}
\draw[gp path] (9.934,7.701)--(11.174,7.701);
\draw[gp path] (2.721,7.539)--(2.953,7.144)--(3.184,6.891)--(3.416,6.669)--(3.648,6.462)%
  --(3.880,6.268)--(4.112,6.093)--(4.344,5.920)--(4.575,5.768)--(4.807,5.671)--(5.039,5.575)%
  --(5.271,5.478)--(5.503,5.379)--(5.735,5.286)--(5.966,5.190)--(6.198,5.099)--(6.430,5.005)%
  --(6.662,4.909)--(6.894,4.818)--(7.125,4.722)--(7.357,4.631)--(7.589,4.538)--(7.821,4.447)%
  --(8.053,4.358)--(8.285,4.268)--(8.516,4.192)--(8.748,4.155)--(8.980,4.116)--(9.212,4.072)%
  --(9.444,4.026)--(9.676,3.979)--(9.907,3.929)--(10.139,3.880)--(10.371,3.827)--(10.603,3.775)%
  --(10.835,3.722)--(11.066,3.668)--(11.298,3.612)--(11.530,3.558)--(11.704,3.516);
\gpsetpointsize{4.00}
\gppoint{gp mark 9}{(2.721,7.539)}
\gppoint{gp mark 9}{(2.953,7.144)}
\gppoint{gp mark 9}{(3.184,6.891)}
\gppoint{gp mark 9}{(3.416,6.669)}
\gppoint{gp mark 9}{(3.648,6.462)}
\gppoint{gp mark 9}{(3.880,6.268)}
\gppoint{gp mark 9}{(4.112,6.093)}
\gppoint{gp mark 9}{(4.344,5.920)}
\gppoint{gp mark 9}{(4.575,5.768)}
\gppoint{gp mark 9}{(4.807,5.671)}
\gppoint{gp mark 9}{(5.039,5.575)}
\gppoint{gp mark 9}{(5.271,5.478)}
\gppoint{gp mark 9}{(5.503,5.379)}
\gppoint{gp mark 9}{(5.735,5.286)}
\gppoint{gp mark 9}{(5.966,5.190)}
\gppoint{gp mark 9}{(6.198,5.099)}
\gppoint{gp mark 9}{(6.430,5.005)}
\gppoint{gp mark 9}{(6.662,4.909)}
\gppoint{gp mark 9}{(6.894,4.818)}
\gppoint{gp mark 9}{(7.125,4.722)}
\gppoint{gp mark 9}{(7.357,4.631)}
\gppoint{gp mark 9}{(7.589,4.538)}
\gppoint{gp mark 9}{(7.821,4.447)}
\gppoint{gp mark 9}{(8.053,4.358)}
\gppoint{gp mark 9}{(8.285,4.268)}
\gppoint{gp mark 9}{(8.516,4.192)}
\gppoint{gp mark 9}{(8.748,4.155)}
\gppoint{gp mark 9}{(8.980,4.116)}
\gppoint{gp mark 9}{(9.212,4.072)}
\gppoint{gp mark 9}{(9.444,4.026)}
\gppoint{gp mark 9}{(9.676,3.979)}
\gppoint{gp mark 9}{(9.907,3.929)}
\gppoint{gp mark 9}{(10.139,3.880)}
\gppoint{gp mark 9}{(10.371,3.827)}
\gppoint{gp mark 9}{(10.603,3.775)}
\gppoint{gp mark 9}{(10.835,3.722)}
\gppoint{gp mark 9}{(11.066,3.668)}
\gppoint{gp mark 9}{(11.298,3.612)}
\gppoint{gp mark 9}{(11.530,3.558)}
\gppoint{gp mark 9}{(10.554,7.701)}
\node[gp node right] at (9.669,7.026) {DE-Sinc};
\gpsetdashtype{gp dt solid}
\draw[gp path] (9.934,7.026)--(11.174,7.026);
\draw[gp path] (2.721,7.504)--(2.953,7.168)--(3.184,6.880)--(3.416,6.539)--(3.648,6.311)%
  --(3.880,6.134)--(4.112,5.952)--(4.344,5.760)--(4.575,5.546)--(4.807,5.323)--(5.039,5.080)%
  --(5.271,4.891)--(5.503,4.783)--(5.735,4.653)--(5.966,4.503)--(6.198,4.338)--(6.430,4.159)%
  --(6.662,3.965)--(6.894,3.756)--(7.125,3.559)--(7.357,3.478)--(7.589,3.370)--(7.821,3.244)%
  --(8.053,3.103)--(8.285,2.949)--(8.516,2.782)--(8.748,2.600)--(8.980,2.406)--(9.212,2.248)%
  --(9.444,2.168)--(9.676,2.095)--(9.907,2.005)--(10.139,1.962)--(10.371,1.933)--(10.603,1.951)%
  --(10.835,1.882)--(11.066,1.898)--(11.298,1.913)--(11.530,1.939)--(11.704,1.896);
\gppoint{gp mark 7}{(2.721,7.504)}
\gppoint{gp mark 7}{(2.953,7.168)}
\gppoint{gp mark 7}{(3.184,6.880)}
\gppoint{gp mark 7}{(3.416,6.539)}
\gppoint{gp mark 7}{(3.648,6.311)}
\gppoint{gp mark 7}{(3.880,6.134)}
\gppoint{gp mark 7}{(4.112,5.952)}
\gppoint{gp mark 7}{(4.344,5.760)}
\gppoint{gp mark 7}{(4.575,5.546)}
\gppoint{gp mark 7}{(4.807,5.323)}
\gppoint{gp mark 7}{(5.039,5.080)}
\gppoint{gp mark 7}{(5.271,4.891)}
\gppoint{gp mark 7}{(5.503,4.783)}
\gppoint{gp mark 7}{(5.735,4.653)}
\gppoint{gp mark 7}{(5.966,4.503)}
\gppoint{gp mark 7}{(6.198,4.338)}
\gppoint{gp mark 7}{(6.430,4.159)}
\gppoint{gp mark 7}{(6.662,3.965)}
\gppoint{gp mark 7}{(6.894,3.756)}
\gppoint{gp mark 7}{(7.125,3.559)}
\gppoint{gp mark 7}{(7.357,3.478)}
\gppoint{gp mark 7}{(7.589,3.370)}
\gppoint{gp mark 7}{(7.821,3.244)}
\gppoint{gp mark 7}{(8.053,3.103)}
\gppoint{gp mark 7}{(8.285,2.949)}
\gppoint{gp mark 7}{(8.516,2.782)}
\gppoint{gp mark 7}{(8.748,2.600)}
\gppoint{gp mark 7}{(8.980,2.406)}
\gppoint{gp mark 7}{(9.212,2.248)}
\gppoint{gp mark 7}{(9.444,2.168)}
\gppoint{gp mark 7}{(9.676,2.095)}
\gppoint{gp mark 7}{(9.907,2.005)}
\gppoint{gp mark 7}{(10.139,1.962)}
\gppoint{gp mark 7}{(10.371,1.933)}
\gppoint{gp mark 7}{(10.603,1.951)}
\gppoint{gp mark 7}{(10.835,1.882)}
\gppoint{gp mark 7}{(11.066,1.898)}
\gppoint{gp mark 7}{(11.298,1.913)}
\gppoint{gp mark 7}{(11.530,1.939)}
\gppoint{gp mark 7}{(10.554,7.026)}
\draw[gp path] (2.431,8.219)--(2.431,1.420)--(11.704,1.420)--(11.704,8.219)--cycle;
\gpdefrectangularnode{gp plot 1}{\pgfpoint{2.431cm}{1.420cm}}{\pgfpoint{11.704cm}{8.219cm}}
\end{tikzpicture}
}
\caption{Numerical results for Example~\ref{ex:2}.}
\label{fig:ex2}
\end{figure}
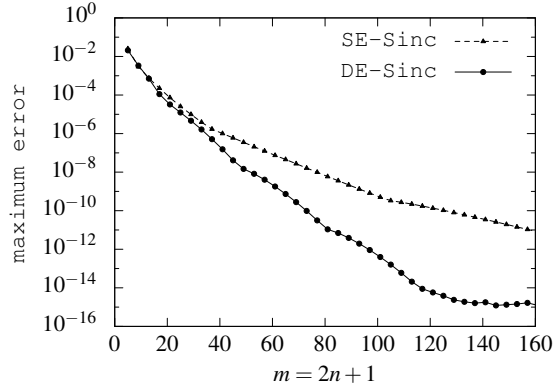

\subsection{Comparison of \texorpdfstring{$F$}{F}: \texorpdfstring{$F$}{F} has a pole/branch point or not}

Next, let us fix $g(t)$, and focus on the difference of $F(s)$.
In the subsequent four examples,
the assumptions in Theorem~\ref{thm:SE-new}
are fulfilled with $d=3.14$ (slightly less than $\pi$),
and those in Theorem~\ref{thm:DE-new}
are fulfilled with $d=1.57$ (slightly less than $\pi/2$).

\begin{example}
\label{ex:3}
Consider the following indefinite convolution
\[
\int_0^x J_0(2\sqrt{x - t})\sqrt{t}\D{t} = \frac{1}{4}
\left\{
\sin\left(2\sqrt{x}\right)
-2\sqrt{x}\cos\left(2\sqrt{x}\right)
\right\},
\quad 0\leq x\leq 2,
\]
where $J_0$ is the Bessel function of the first kind of order zero.
In this case,
$\hat{f}(s) = \E^{-1/s}/s$ and $F(s)=\hat{f}(1/s)=s\E^{-s}$.
\end{example}

\begin{example}
\label{ex:4}
Consider the following indefinite convolution
\[
\int_0^x \E^{x-t}\sqrt{t}\D{t}
 = \frac{\sqrt{\pi}}{2}\E^x \erf\left(\sqrt{x}\right) - \sqrt{x},
\quad 0\leq x\leq 2,
\]
where $\erf(x)$ is the error function.
In this case,
$\hat{f}(s) = 1/(s-1)$ and $F(s)=\hat{f}(1/s)=s/(1-s)$.
\end{example}

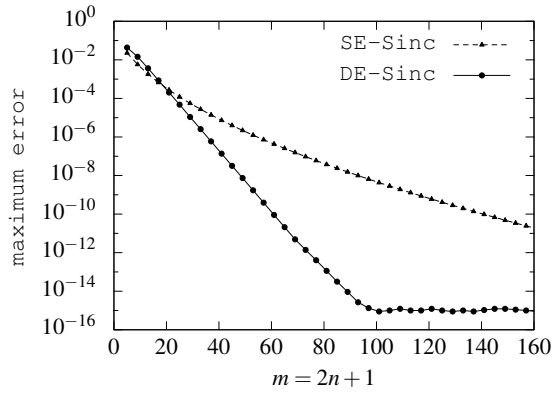
\begin{figure}
\centering
\scalebox{0.6}{
\begin{tikzpicture}[gnuplot]
\tikzset{every node/.append style={font={\ttfamily\fontsize{14.4pt}{17.28pt}\selectfont}}}
\gpmonochromelines
\path (0.000,0.000) rectangle (12.500,8.750);
\gpcolor{color=gp lt color border}
\gpsetlinetype{gp lt border}
\gpsetdashtype{gp dt solid}
\gpsetlinewidth{1.00}
\draw[gp path] (2.431,1.420)--(2.611,1.420);
\draw[gp path] (11.704,1.420)--(11.524,1.420);
\node[gp node right] at (2.166,1.420) {$10^{-16}$};
\draw[gp path] (2.431,1.845)--(2.521,1.845);
\draw[gp path] (11.704,1.845)--(11.614,1.845);
\draw[gp path] (2.431,2.270)--(2.611,2.270);
\draw[gp path] (11.704,2.270)--(11.524,2.270);
\node[gp node right] at (2.166,2.270) {$10^{-14}$};
\draw[gp path] (2.431,2.695)--(2.521,2.695);
\draw[gp path] (11.704,2.695)--(11.614,2.695);
\draw[gp path] (2.431,3.120)--(2.611,3.120);
\draw[gp path] (11.704,3.120)--(11.524,3.120);
\node[gp node right] at (2.166,3.120) {$10^{-12}$};
\draw[gp path] (2.431,3.545)--(2.521,3.545);
\draw[gp path] (11.704,3.545)--(11.614,3.545);
\draw[gp path] (2.431,3.970)--(2.611,3.970);
\draw[gp path] (11.704,3.970)--(11.524,3.970);
\node[gp node right] at (2.166,3.970) {$10^{-10}$};
\draw[gp path] (2.431,4.395)--(2.521,4.395);
\draw[gp path] (11.704,4.395)--(11.614,4.395);
\draw[gp path] (2.431,4.820)--(2.611,4.820);
\draw[gp path] (11.704,4.820)--(11.524,4.820);
\node[gp node right] at (2.166,4.820) {$10^{-8}$};
\draw[gp path] (2.431,5.244)--(2.521,5.244);
\draw[gp path] (11.704,5.244)--(11.614,5.244);
\draw[gp path] (2.431,5.669)--(2.611,5.669);
\draw[gp path] (11.704,5.669)--(11.524,5.669);
\node[gp node right] at (2.166,5.669) {$10^{-6}$};
\draw[gp path] (2.431,6.094)--(2.521,6.094);
\draw[gp path] (11.704,6.094)--(11.614,6.094);
\draw[gp path] (2.431,6.519)--(2.611,6.519);
\draw[gp path] (11.704,6.519)--(11.524,6.519);
\node[gp node right] at (2.166,6.519) {$10^{-4}$};
\draw[gp path] (2.431,6.944)--(2.521,6.944);
\draw[gp path] (11.704,6.944)--(11.614,6.944);
\draw[gp path] (2.431,7.369)--(2.611,7.369);
\draw[gp path] (11.704,7.369)--(11.524,7.369);
\node[gp node right] at (2.166,7.369) {$10^{-2}$};
\draw[gp path] (2.431,7.794)--(2.521,7.794);
\draw[gp path] (11.704,7.794)--(11.614,7.794);
\draw[gp path] (2.431,8.219)--(2.611,8.219);
\draw[gp path] (11.704,8.219)--(11.524,8.219);
\node[gp node right] at (2.166,8.219) {$10^{0}$};
\draw[gp path] (2.431,1.420)--(2.431,1.600);
\draw[gp path] (2.431,8.219)--(2.431,8.039);
\node[gp node center] at (2.431,0.976) {$0$};
\draw[gp path] (3.590,1.420)--(3.590,1.600);
\draw[gp path] (3.590,8.219)--(3.590,8.039);
\node[gp node center] at (3.590,0.976) {$20$};
\draw[gp path] (4.749,1.420)--(4.749,1.600);
\draw[gp path] (4.749,8.219)--(4.749,8.039);
\node[gp node center] at (4.749,0.976) {$40$};
\draw[gp path] (5.908,1.420)--(5.908,1.600);
\draw[gp path] (5.908,8.219)--(5.908,8.039);
\node[gp node center] at (5.908,0.976) {$60$};
\draw[gp path] (7.068,1.420)--(7.068,1.600);
\draw[gp path] (7.068,8.219)--(7.068,8.039);
\node[gp node center] at (7.068,0.976) {$80$};
\draw[gp path] (8.227,1.420)--(8.227,1.600);
\draw[gp path] (8.227,8.219)--(8.227,8.039);
\node[gp node center] at (8.227,0.976) {$100$};
\draw[gp path] (9.386,1.420)--(9.386,1.600);
\draw[gp path] (9.386,8.219)--(9.386,8.039);
\node[gp node center] at (9.386,0.976) {$120$};
\draw[gp path] (10.545,1.420)--(10.545,1.600);
\draw[gp path] (10.545,8.219)--(10.545,8.039);
\node[gp node center] at (10.545,0.976) {$140$};
\draw[gp path] (11.704,1.420)--(11.704,1.600);
\draw[gp path] (11.704,8.219)--(11.704,8.039);
\node[gp node center] at (11.704,0.976) {$160$};
\draw[gp path] (2.431,8.219)--(2.431,1.420)--(11.704,1.420)--(11.704,8.219)--cycle;
\node[gp node center,rotate=-270] at (0.354,4.819) {maximum error};
\node[gp node center] at (7.067,0.310) {$m=2n+1$};
\node[gp node right] at (9.669,7.701) {SE-Sinc};
\gpsetdashtype{gp dt 2}
\draw[gp path] (9.934,7.701)--(11.174,7.701);
\draw[gp path] (2.721,7.521)--(2.953,7.268)--(3.184,7.055)--(3.416,6.870)--(3.648,6.704)%
  --(3.880,6.552)--(4.112,6.411)--(4.344,6.280)--(4.575,6.155)--(4.807,6.036)--(5.039,5.923)%
  --(5.271,5.814)--(5.503,5.709)--(5.735,5.609)--(5.966,5.512)--(6.198,5.418)--(6.430,5.326)%
  --(6.662,5.236)--(6.894,5.149)--(7.125,5.064)--(7.357,4.980)--(7.589,4.899)--(7.821,4.819)%
  --(8.053,4.740)--(8.285,4.663)--(8.516,4.588)--(8.748,4.514)--(8.980,4.442)--(9.212,4.370)%
  --(9.444,4.300)--(9.676,4.231)--(9.907,4.163)--(10.139,4.095)--(10.371,4.029)--(10.603,3.964)%
  --(10.835,3.899)--(11.066,3.835)--(11.298,3.772)--(11.530,3.710)--(11.704,3.664);
\gpsetpointsize{4.00}
\gppoint{gp mark 9}{(2.721,7.521)}
\gppoint{gp mark 9}{(2.953,7.268)}
\gppoint{gp mark 9}{(3.184,7.055)}
\gppoint{gp mark 9}{(3.416,6.870)}
\gppoint{gp mark 9}{(3.648,6.704)}
\gppoint{gp mark 9}{(3.880,6.552)}
\gppoint{gp mark 9}{(4.112,6.411)}
\gppoint{gp mark 9}{(4.344,6.280)}
\gppoint{gp mark 9}{(4.575,6.155)}
\gppoint{gp mark 9}{(4.807,6.036)}
\gppoint{gp mark 9}{(5.039,5.923)}
\gppoint{gp mark 9}{(5.271,5.814)}
\gppoint{gp mark 9}{(5.503,5.709)}
\gppoint{gp mark 9}{(5.735,5.609)}
\gppoint{gp mark 9}{(5.966,5.512)}
\gppoint{gp mark 9}{(6.198,5.418)}
\gppoint{gp mark 9}{(6.430,5.326)}
\gppoint{gp mark 9}{(6.662,5.236)}
\gppoint{gp mark 9}{(6.894,5.149)}
\gppoint{gp mark 9}{(7.125,5.064)}
\gppoint{gp mark 9}{(7.357,4.980)}
\gppoint{gp mark 9}{(7.589,4.899)}
\gppoint{gp mark 9}{(7.821,4.819)}
\gppoint{gp mark 9}{(8.053,4.740)}
\gppoint{gp mark 9}{(8.285,4.663)}
\gppoint{gp mark 9}{(8.516,4.588)}
\gppoint{gp mark 9}{(8.748,4.514)}
\gppoint{gp mark 9}{(8.980,4.442)}
\gppoint{gp mark 9}{(9.212,4.370)}
\gppoint{gp mark 9}{(9.444,4.300)}
\gppoint{gp mark 9}{(9.676,4.231)}
\gppoint{gp mark 9}{(9.907,4.163)}
\gppoint{gp mark 9}{(10.139,4.095)}
\gppoint{gp mark 9}{(10.371,4.029)}
\gppoint{gp mark 9}{(10.603,3.964)}
\gppoint{gp mark 9}{(10.835,3.899)}
\gppoint{gp mark 9}{(11.066,3.835)}
\gppoint{gp mark 9}{(11.298,3.772)}
\gppoint{gp mark 9}{(11.530,3.710)}
\gppoint{gp mark 9}{(10.554,7.701)}
\node[gp node right] at (9.669,7.026) {DE-Sinc};
\gpsetdashtype{gp dt solid}
\draw[gp path] (9.934,7.026)--(11.174,7.026);
\draw[gp path] (2.721,7.640)--(2.953,7.437)--(3.184,7.183)--(3.416,6.915)--(3.648,6.649)%
  --(3.880,6.382)--(4.112,6.109)--(4.344,5.842)--(4.575,5.572)--(4.807,5.302)--(5.039,5.033)%
  --(5.271,4.764)--(5.503,4.494)--(5.735,4.222)--(5.966,3.951)--(6.198,3.681)--(6.430,3.413)%
  --(6.662,3.181)--(6.894,2.953)--(7.125,2.719)--(7.357,2.481)--(7.589,2.255)--(7.821,2.026)%
  --(8.053,1.898)--(8.285,1.823)--(8.516,1.845)--(8.748,1.882)--(8.980,1.845)--(9.212,1.845)%
  --(9.444,1.882)--(9.676,1.845)--(9.907,1.823)--(10.139,1.845)--(10.371,1.823)--(10.603,1.855)%
  --(10.835,1.882)--(11.066,1.882)--(11.298,1.864)--(11.530,1.845)--(11.704,1.828);
\gppoint{gp mark 7}{(2.721,7.640)}
\gppoint{gp mark 7}{(2.953,7.437)}
\gppoint{gp mark 7}{(3.184,7.183)}
\gppoint{gp mark 7}{(3.416,6.915)}
\gppoint{gp mark 7}{(3.648,6.649)}
\gppoint{gp mark 7}{(3.880,6.382)}
\gppoint{gp mark 7}{(4.112,6.109)}
\gppoint{gp mark 7}{(4.344,5.842)}
\gppoint{gp mark 7}{(4.575,5.572)}
\gppoint{gp mark 7}{(4.807,5.302)}
\gppoint{gp mark 7}{(5.039,5.033)}
\gppoint{gp mark 7}{(5.271,4.764)}
\gppoint{gp mark 7}{(5.503,4.494)}
\gppoint{gp mark 7}{(5.735,4.222)}
\gppoint{gp mark 7}{(5.966,3.951)}
\gppoint{gp mark 7}{(6.198,3.681)}
\gppoint{gp mark 7}{(6.430,3.413)}
\gppoint{gp mark 7}{(6.662,3.181)}
\gppoint{gp mark 7}{(6.894,2.953)}
\gppoint{gp mark 7}{(7.125,2.719)}
\gppoint{gp mark 7}{(7.357,2.481)}
\gppoint{gp mark 7}{(7.589,2.255)}
\gppoint{gp mark 7}{(7.821,2.026)}
\gppoint{gp mark 7}{(8.053,1.898)}
\gppoint{gp mark 7}{(8.285,1.823)}
\gppoint{gp mark 7}{(8.516,1.845)}
\gppoint{gp mark 7}{(8.748,1.882)}
\gppoint{gp mark 7}{(8.980,1.845)}
\gppoint{gp mark 7}{(9.212,1.845)}
\gppoint{gp mark 7}{(9.444,1.882)}
\gppoint{gp mark 7}{(9.676,1.845)}
\gppoint{gp mark 7}{(9.907,1.823)}
\gppoint{gp mark 7}{(10.139,1.845)}
\gppoint{gp mark 7}{(10.371,1.823)}
\gppoint{gp mark 7}{(10.603,1.855)}
\gppoint{gp mark 7}{(10.835,1.882)}
\gppoint{gp mark 7}{(11.066,1.882)}
\gppoint{gp mark 7}{(11.298,1.864)}
\gppoint{gp mark 7}{(11.530,1.845)}
\gppoint{gp mark 7}{(10.554,7.026)}
\draw[gp path] (2.431,8.219)--(2.431,1.420)--(11.704,1.420)--(11.704,8.219)--cycle;
\gpdefrectangularnode{gp plot 1}{\pgfpoint{2.431cm}{1.420cm}}{\pgfpoint{11.704cm}{8.219cm}}
\end{tikzpicture}
}
\caption{Numerical results for Example~\ref{ex:3}.}
\label{fig:ex3}
\end{figure}
\begin{figure}
\centering
\scalebox{0.6}{
\begin{tikzpicture}[gnuplot]
\tikzset{every node/.append style={font={\ttfamily\fontsize{14.4pt}{17.28pt}\selectfont}}}
\gpmonochromelines
\path (0.000,0.000) rectangle (12.500,8.750);
\gpcolor{color=gp lt color border}
\gpsetlinetype{gp lt border}
\gpsetdashtype{gp dt solid}
\gpsetlinewidth{1.00}
\draw[gp path] (2.431,1.420)--(2.611,1.420);
\draw[gp path] (11.704,1.420)--(11.524,1.420);
\node[gp node right] at (2.166,1.420) {$10^{-16}$};
\draw[gp path] (2.431,1.845)--(2.521,1.845);
\draw[gp path] (11.704,1.845)--(11.614,1.845);
\draw[gp path] (2.431,2.270)--(2.611,2.270);
\draw[gp path] (11.704,2.270)--(11.524,2.270);
\node[gp node right] at (2.166,2.270) {$10^{-14}$};
\draw[gp path] (2.431,2.695)--(2.521,2.695);
\draw[gp path] (11.704,2.695)--(11.614,2.695);
\draw[gp path] (2.431,3.120)--(2.611,3.120);
\draw[gp path] (11.704,3.120)--(11.524,3.120);
\node[gp node right] at (2.166,3.120) {$10^{-12}$};
\draw[gp path] (2.431,3.545)--(2.521,3.545);
\draw[gp path] (11.704,3.545)--(11.614,3.545);
\draw[gp path] (2.431,3.970)--(2.611,3.970);
\draw[gp path] (11.704,3.970)--(11.524,3.970);
\node[gp node right] at (2.166,3.970) {$10^{-10}$};
\draw[gp path] (2.431,4.395)--(2.521,4.395);
\draw[gp path] (11.704,4.395)--(11.614,4.395);
\draw[gp path] (2.431,4.820)--(2.611,4.820);
\draw[gp path] (11.704,4.820)--(11.524,4.820);
\node[gp node right] at (2.166,4.820) {$10^{-8}$};
\draw[gp path] (2.431,5.244)--(2.521,5.244);
\draw[gp path] (11.704,5.244)--(11.614,5.244);
\draw[gp path] (2.431,5.669)--(2.611,5.669);
\draw[gp path] (11.704,5.669)--(11.524,5.669);
\node[gp node right] at (2.166,5.669) {$10^{-6}$};
\draw[gp path] (2.431,6.094)--(2.521,6.094);
\draw[gp path] (11.704,6.094)--(11.614,6.094);
\draw[gp path] (2.431,6.519)--(2.611,6.519);
\draw[gp path] (11.704,6.519)--(11.524,6.519);
\node[gp node right] at (2.166,6.519) {$10^{-4}$};
\draw[gp path] (2.431,6.944)--(2.521,6.944);
\draw[gp path] (11.704,6.944)--(11.614,6.944);
\draw[gp path] (2.431,7.369)--(2.611,7.369);
\draw[gp path] (11.704,7.369)--(11.524,7.369);
\node[gp node right] at (2.166,7.369) {$10^{-2}$};
\draw[gp path] (2.431,7.794)--(2.521,7.794);
\draw[gp path] (11.704,7.794)--(11.614,7.794);
\draw[gp path] (2.431,8.219)--(2.611,8.219);
\draw[gp path] (11.704,8.219)--(11.524,8.219);
\node[gp node right] at (2.166,8.219) {$10^{0}$};
\draw[gp path] (2.431,1.420)--(2.431,1.600);
\draw[gp path] (2.431,8.219)--(2.431,8.039);
\node[gp node center] at (2.431,0.976) {$0$};
\draw[gp path] (3.590,1.420)--(3.590,1.600);
\draw[gp path] (3.590,8.219)--(3.590,8.039);
\node[gp node center] at (3.590,0.976) {$20$};
\draw[gp path] (4.749,1.420)--(4.749,1.600);
\draw[gp path] (4.749,8.219)--(4.749,8.039);
\node[gp node center] at (4.749,0.976) {$40$};
\draw[gp path] (5.908,1.420)--(5.908,1.600);
\draw[gp path] (5.908,8.219)--(5.908,8.039);
\node[gp node center] at (5.908,0.976) {$60$};
\draw[gp path] (7.068,1.420)--(7.068,1.600);
\draw[gp path] (7.068,8.219)--(7.068,8.039);
\node[gp node center] at (7.068,0.976) {$80$};
\draw[gp path] (8.227,1.420)--(8.227,1.600);
\draw[gp path] (8.227,8.219)--(8.227,8.039);
\node[gp node center] at (8.227,0.976) {$100$};
\draw[gp path] (9.386,1.420)--(9.386,1.600);
\draw[gp path] (9.386,8.219)--(9.386,8.039);
\node[gp node center] at (9.386,0.976) {$120$};
\draw[gp path] (10.545,1.420)--(10.545,1.600);
\draw[gp path] (10.545,8.219)--(10.545,8.039);
\node[gp node center] at (10.545,0.976) {$140$};
\draw[gp path] (11.704,1.420)--(11.704,1.600);
\draw[gp path] (11.704,8.219)--(11.704,8.039);
\node[gp node center] at (11.704,0.976) {$160$};
\draw[gp path] (2.431,8.219)--(2.431,1.420)--(11.704,1.420)--(11.704,8.219)--cycle;
\node[gp node center,rotate=-270] at (0.354,4.819) {maximum error};
\node[gp node center] at (7.067,0.310) {$m=2n+1$};
\node[gp node right] at (9.669,7.701) {SE-Sinc};
\gpsetdashtype{gp dt 2}
\draw[gp path] (9.934,7.701)--(11.174,7.701);
\draw[gp path] (2.721,7.997)--(2.953,7.611)--(3.184,7.338)--(3.416,7.106)--(3.648,6.918)%
  --(3.880,6.747)--(4.112,6.586)--(4.344,6.430)--(4.575,6.287)--(4.807,6.147)--(5.039,6.025)%
  --(5.271,5.928)--(5.503,5.834)--(5.735,5.741)--(5.966,5.652)--(6.198,5.564)--(6.430,5.478)%
  --(6.662,5.395)--(6.894,5.314)--(7.125,5.234)--(7.357,5.155)--(7.589,5.077)--(7.821,5.001)%
  --(8.053,4.926)--(8.285,4.852)--(8.516,4.780)--(8.748,4.709)--(8.980,4.639)--(9.212,4.569)%
  --(9.444,4.501)--(9.676,4.433)--(9.907,4.366)--(10.139,4.301)--(10.371,4.236)--(10.603,4.172)%
  --(10.835,4.109)--(11.066,4.046)--(11.298,3.984)--(11.530,3.922)--(11.704,3.877);
\gpsetpointsize{4.00}
\gppoint{gp mark 9}{(2.721,7.997)}
\gppoint{gp mark 9}{(2.953,7.611)}
\gppoint{gp mark 9}{(3.184,7.338)}
\gppoint{gp mark 9}{(3.416,7.106)}
\gppoint{gp mark 9}{(3.648,6.918)}
\gppoint{gp mark 9}{(3.880,6.747)}
\gppoint{gp mark 9}{(4.112,6.586)}
\gppoint{gp mark 9}{(4.344,6.430)}
\gppoint{gp mark 9}{(4.575,6.287)}
\gppoint{gp mark 9}{(4.807,6.147)}
\gppoint{gp mark 9}{(5.039,6.025)}
\gppoint{gp mark 9}{(5.271,5.928)}
\gppoint{gp mark 9}{(5.503,5.834)}
\gppoint{gp mark 9}{(5.735,5.741)}
\gppoint{gp mark 9}{(5.966,5.652)}
\gppoint{gp mark 9}{(6.198,5.564)}
\gppoint{gp mark 9}{(6.430,5.478)}
\gppoint{gp mark 9}{(6.662,5.395)}
\gppoint{gp mark 9}{(6.894,5.314)}
\gppoint{gp mark 9}{(7.125,5.234)}
\gppoint{gp mark 9}{(7.357,5.155)}
\gppoint{gp mark 9}{(7.589,5.077)}
\gppoint{gp mark 9}{(7.821,5.001)}
\gppoint{gp mark 9}{(8.053,4.926)}
\gppoint{gp mark 9}{(8.285,4.852)}
\gppoint{gp mark 9}{(8.516,4.780)}
\gppoint{gp mark 9}{(8.748,4.709)}
\gppoint{gp mark 9}{(8.980,4.639)}
\gppoint{gp mark 9}{(9.212,4.569)}
\gppoint{gp mark 9}{(9.444,4.501)}
\gppoint{gp mark 9}{(9.676,4.433)}
\gppoint{gp mark 9}{(9.907,4.366)}
\gppoint{gp mark 9}{(10.139,4.301)}
\gppoint{gp mark 9}{(10.371,4.236)}
\gppoint{gp mark 9}{(10.603,4.172)}
\gppoint{gp mark 9}{(10.835,4.109)}
\gppoint{gp mark 9}{(11.066,4.046)}
\gppoint{gp mark 9}{(11.298,3.984)}
\gppoint{gp mark 9}{(11.530,3.922)}
\gppoint{gp mark 9}{(10.554,7.701)}
\node[gp node right] at (9.669,7.026) {DE-Sinc};
\gpsetdashtype{gp dt solid}
\draw[gp path] (9.934,7.026)--(11.174,7.026);
\draw[gp path] (2.744,8.219)--(2.953,7.790)--(3.184,7.449)--(3.416,7.127)--(3.648,6.838)%
  --(3.880,6.620)--(4.112,6.390)--(4.344,6.146)--(4.575,5.897)--(4.807,5.645)--(5.039,5.388)%
  --(5.271,5.126)--(5.503,4.856)--(5.735,4.586)--(5.966,4.385)--(6.198,4.182)--(6.430,3.970)%
  --(6.662,3.753)--(6.894,3.531)--(7.125,3.304)--(7.357,3.073)--(7.589,2.839)--(7.821,2.604)%
  --(8.053,2.402)--(8.285,2.229)--(8.516,2.120)--(8.748,2.182)--(8.980,2.182)--(9.212,2.182)%
  --(9.444,2.168)--(9.676,2.154)--(9.907,2.195)--(10.139,2.229)--(10.371,2.182)--(10.603,2.207)%
  --(10.835,2.182)--(11.066,2.120)--(11.298,2.182)--(11.530,2.229)--(11.704,2.172);
\gppoint{gp mark 7}{(2.953,7.790)}
\gppoint{gp mark 7}{(3.184,7.449)}
\gppoint{gp mark 7}{(3.416,7.127)}
\gppoint{gp mark 7}{(3.648,6.838)}
\gppoint{gp mark 7}{(3.880,6.620)}
\gppoint{gp mark 7}{(4.112,6.390)}
\gppoint{gp mark 7}{(4.344,6.146)}
\gppoint{gp mark 7}{(4.575,5.897)}
\gppoint{gp mark 7}{(4.807,5.645)}
\gppoint{gp mark 7}{(5.039,5.388)}
\gppoint{gp mark 7}{(5.271,5.126)}
\gppoint{gp mark 7}{(5.503,4.856)}
\gppoint{gp mark 7}{(5.735,4.586)}
\gppoint{gp mark 7}{(5.966,4.385)}
\gppoint{gp mark 7}{(6.198,4.182)}
\gppoint{gp mark 7}{(6.430,3.970)}
\gppoint{gp mark 7}{(6.662,3.753)}
\gppoint{gp mark 7}{(6.894,3.531)}
\gppoint{gp mark 7}{(7.125,3.304)}
\gppoint{gp mark 7}{(7.357,3.073)}
\gppoint{gp mark 7}{(7.589,2.839)}
\gppoint{gp mark 7}{(7.821,2.604)}
\gppoint{gp mark 7}{(8.053,2.402)}
\gppoint{gp mark 7}{(8.285,2.229)}
\gppoint{gp mark 7}{(8.516,2.120)}
\gppoint{gp mark 7}{(8.748,2.182)}
\gppoint{gp mark 7}{(8.980,2.182)}
\gppoint{gp mark 7}{(9.212,2.182)}
\gppoint{gp mark 7}{(9.444,2.168)}
\gppoint{gp mark 7}{(9.676,2.154)}
\gppoint{gp mark 7}{(9.907,2.195)}
\gppoint{gp mark 7}{(10.139,2.229)}
\gppoint{gp mark 7}{(10.371,2.182)}
\gppoint{gp mark 7}{(10.603,2.207)}
\gppoint{gp mark 7}{(10.835,2.182)}
\gppoint{gp mark 7}{(11.066,2.120)}
\gppoint{gp mark 7}{(11.298,2.182)}
\gppoint{gp mark 7}{(11.530,2.229)}
\gppoint{gp mark 7}{(10.554,7.026)}
\draw[gp path] (2.431,8.219)--(2.431,1.420)--(11.704,1.420)--(11.704,8.219)--cycle;
\gpdefrectangularnode{gp plot 1}{\pgfpoint{2.431cm}{1.420cm}}{\pgfpoint{11.704cm}{8.219cm}}
\end{tikzpicture}
}
\caption{Numerical results for Example~\ref{ex:4}.}
\label{fig:ex4}
\end{figure}
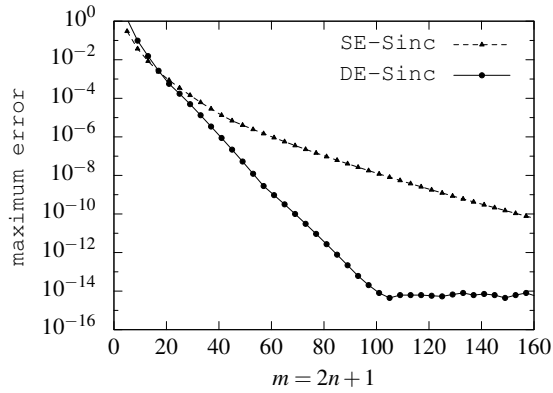

\begin{example}
\label{ex:5}
Consider the following indefinite convolution
\[
\int_0^x \cos(x - t)\sqrt{t}\D{t} = \sqrt{\frac{\pi}{2}}
\left\{
  C\left(\sqrt{\frac{2 x}{\pi}}\right)\sin x
- S\left(\sqrt{\frac{2 x}{\pi}}\right)\cos x
\right\},
\quad 0\leq x\leq 2,
\]
where $C(x)$ and $S(x)$ are the Fresnel integrals defined by
\begin{align}
 C(x) &= \int_0^x \cos\left(\frac{\pi}{2}t^2\right)\D{t},
\label{eq:FresnelC}
\\
 S(x) &= \int_0^x \sin\left(\frac{\pi}{2}t^2\right)\D{t}.
\label{eq:FresnelS}
\end{align}
In this case,
$\hat{f}(s) = s/(1+s^2)$ and $F(s)=\hat{f}(1/s)=s/(1 + s^2)$.
\end{example}

\begin{example}
\label{ex:6}
Consider the following indefinite convolution
\begin{align*}
\int_0^x \frac{\sin(x-t)}{x-t}\sqrt{t}\D{t}
& = \sqrt{2\pi}
\left\{
S\left(\sqrt{\frac{2 x}{\pi}}\right)\cos x
- C\left(\sqrt{\frac{2 x}{\pi}}\right)\sin x
\right\}\\
&\quad +
\pi\sqrt{x}
\left\{S^2\left(\sqrt{\frac{2 x}{\pi}}\right)
 + C^2\left(\sqrt{\frac{2 x}{\pi}}\right)\right\},
\quad 0\leq x\leq 2,\nonumber
\end{align*}
where $C(x)$ and $S(x)$ are the Fresnel integrals defined
by~\eqref{eq:FresnelC} and~\eqref{eq:FresnelS}, respectively.
In this case,
$\hat{f}(s) = \arctan(1/s)$ and $F(s)=\arctan s$.
\end{example}

\begin{figure}
\centering
\scalebox{0.6}{
\begin{tikzpicture}[gnuplot]
\tikzset{every node/.append style={font={\ttfamily\fontsize{14.4pt}{17.28pt}\selectfont}}}
\gpmonochromelines
\path (0.000,0.000) rectangle (12.500,8.750);
\gpcolor{color=gp lt color border}
\gpsetlinetype{gp lt border}
\gpsetdashtype{gp dt solid}
\gpsetlinewidth{1.00}
\draw[gp path] (2.431,1.420)--(2.611,1.420);
\draw[gp path] (11.704,1.420)--(11.524,1.420);
\node[gp node right] at (2.166,1.420) {$10^{-16}$};
\draw[gp path] (2.431,1.845)--(2.521,1.845);
\draw[gp path] (11.704,1.845)--(11.614,1.845);
\draw[gp path] (2.431,2.270)--(2.611,2.270);
\draw[gp path] (11.704,2.270)--(11.524,2.270);
\node[gp node right] at (2.166,2.270) {$10^{-14}$};
\draw[gp path] (2.431,2.695)--(2.521,2.695);
\draw[gp path] (11.704,2.695)--(11.614,2.695);
\draw[gp path] (2.431,3.120)--(2.611,3.120);
\draw[gp path] (11.704,3.120)--(11.524,3.120);
\node[gp node right] at (2.166,3.120) {$10^{-12}$};
\draw[gp path] (2.431,3.545)--(2.521,3.545);
\draw[gp path] (11.704,3.545)--(11.614,3.545);
\draw[gp path] (2.431,3.970)--(2.611,3.970);
\draw[gp path] (11.704,3.970)--(11.524,3.970);
\node[gp node right] at (2.166,3.970) {$10^{-10}$};
\draw[gp path] (2.431,4.395)--(2.521,4.395);
\draw[gp path] (11.704,4.395)--(11.614,4.395);
\draw[gp path] (2.431,4.820)--(2.611,4.820);
\draw[gp path] (11.704,4.820)--(11.524,4.820);
\node[gp node right] at (2.166,4.820) {$10^{-8}$};
\draw[gp path] (2.431,5.244)--(2.521,5.244);
\draw[gp path] (11.704,5.244)--(11.614,5.244);
\draw[gp path] (2.431,5.669)--(2.611,5.669);
\draw[gp path] (11.704,5.669)--(11.524,5.669);
\node[gp node right] at (2.166,5.669) {$10^{-6}$};
\draw[gp path] (2.431,6.094)--(2.521,6.094);
\draw[gp path] (11.704,6.094)--(11.614,6.094);
\draw[gp path] (2.431,6.519)--(2.611,6.519);
\draw[gp path] (11.704,6.519)--(11.524,6.519);
\node[gp node right] at (2.166,6.519) {$10^{-4}$};
\draw[gp path] (2.431,6.944)--(2.521,6.944);
\draw[gp path] (11.704,6.944)--(11.614,6.944);
\draw[gp path] (2.431,7.369)--(2.611,7.369);
\draw[gp path] (11.704,7.369)--(11.524,7.369);
\node[gp node right] at (2.166,7.369) {$10^{-2}$};
\draw[gp path] (2.431,7.794)--(2.521,7.794);
\draw[gp path] (11.704,7.794)--(11.614,7.794);
\draw[gp path] (2.431,8.219)--(2.611,8.219);
\draw[gp path] (11.704,8.219)--(11.524,8.219);
\node[gp node right] at (2.166,8.219) {$10^{0}$};
\draw[gp path] (2.431,1.420)--(2.431,1.600);
\draw[gp path] (2.431,8.219)--(2.431,8.039);
\node[gp node center] at (2.431,0.976) {$0$};
\draw[gp path] (3.590,1.420)--(3.590,1.600);
\draw[gp path] (3.590,8.219)--(3.590,8.039);
\node[gp node center] at (3.590,0.976) {$20$};
\draw[gp path] (4.749,1.420)--(4.749,1.600);
\draw[gp path] (4.749,8.219)--(4.749,8.039);
\node[gp node center] at (4.749,0.976) {$40$};
\draw[gp path] (5.908,1.420)--(5.908,1.600);
\draw[gp path] (5.908,8.219)--(5.908,8.039);
\node[gp node center] at (5.908,0.976) {$60$};
\draw[gp path] (7.068,1.420)--(7.068,1.600);
\draw[gp path] (7.068,8.219)--(7.068,8.039);
\node[gp node center] at (7.068,0.976) {$80$};
\draw[gp path] (8.227,1.420)--(8.227,1.600);
\draw[gp path] (8.227,8.219)--(8.227,8.039);
\node[gp node center] at (8.227,0.976) {$100$};
\draw[gp path] (9.386,1.420)--(9.386,1.600);
\draw[gp path] (9.386,8.219)--(9.386,8.039);
\node[gp node center] at (9.386,0.976) {$120$};
\draw[gp path] (10.545,1.420)--(10.545,1.600);
\draw[gp path] (10.545,8.219)--(10.545,8.039);
\node[gp node center] at (10.545,0.976) {$140$};
\draw[gp path] (11.704,1.420)--(11.704,1.600);
\draw[gp path] (11.704,8.219)--(11.704,8.039);
\node[gp node center] at (11.704,0.976) {$160$};
\draw[gp path] (2.431,8.219)--(2.431,1.420)--(11.704,1.420)--(11.704,8.219)--cycle;
\node[gp node center,rotate=-270] at (0.354,4.819) {maximum error};
\node[gp node center] at (7.067,0.310) {$m=2n+1$};
\node[gp node right] at (9.669,7.701) {SE-Sinc};
\gpsetdashtype{gp dt 2}
\draw[gp path] (9.934,7.701)--(11.174,7.701);
\draw[gp path] (2.721,7.601)--(2.953,7.299)--(3.184,7.093)--(3.416,6.916)--(3.648,6.758)%
  --(3.880,6.613)--(4.112,6.478)--(4.344,6.355)--(4.575,6.238)--(4.807,6.127)--(5.039,6.020)%
  --(5.271,5.918)--(5.503,5.818)--(5.735,5.723)--(5.966,5.634)--(6.198,5.545)--(6.430,5.458)%
  --(6.662,5.373)--(6.894,5.292)--(7.125,5.212)--(7.357,5.135)--(7.589,5.058)--(7.821,4.983)%
  --(8.053,4.908)--(8.285,4.839)--(8.516,4.770)--(8.748,4.700)--(8.980,4.630)--(9.212,4.565)%
  --(9.444,4.500)--(9.676,4.434)--(9.907,4.372)--(10.139,4.309)--(10.371,4.247)--(10.603,4.186)%
  --(10.835,4.125)--(11.066,4.067)--(11.298,4.006)--(11.530,3.951)--(11.704,3.905);
\gpsetpointsize{4.00}
\gppoint{gp mark 9}{(2.721,7.601)}
\gppoint{gp mark 9}{(2.953,7.299)}
\gppoint{gp mark 9}{(3.184,7.093)}
\gppoint{gp mark 9}{(3.416,6.916)}
\gppoint{gp mark 9}{(3.648,6.758)}
\gppoint{gp mark 9}{(3.880,6.613)}
\gppoint{gp mark 9}{(4.112,6.478)}
\gppoint{gp mark 9}{(4.344,6.355)}
\gppoint{gp mark 9}{(4.575,6.238)}
\gppoint{gp mark 9}{(4.807,6.127)}
\gppoint{gp mark 9}{(5.039,6.020)}
\gppoint{gp mark 9}{(5.271,5.918)}
\gppoint{gp mark 9}{(5.503,5.818)}
\gppoint{gp mark 9}{(5.735,5.723)}
\gppoint{gp mark 9}{(5.966,5.634)}
\gppoint{gp mark 9}{(6.198,5.545)}
\gppoint{gp mark 9}{(6.430,5.458)}
\gppoint{gp mark 9}{(6.662,5.373)}
\gppoint{gp mark 9}{(6.894,5.292)}
\gppoint{gp mark 9}{(7.125,5.212)}
\gppoint{gp mark 9}{(7.357,5.135)}
\gppoint{gp mark 9}{(7.589,5.058)}
\gppoint{gp mark 9}{(7.821,4.983)}
\gppoint{gp mark 9}{(8.053,4.908)}
\gppoint{gp mark 9}{(8.285,4.839)}
\gppoint{gp mark 9}{(8.516,4.770)}
\gppoint{gp mark 9}{(8.748,4.700)}
\gppoint{gp mark 9}{(8.980,4.630)}
\gppoint{gp mark 9}{(9.212,4.565)}
\gppoint{gp mark 9}{(9.444,4.500)}
\gppoint{gp mark 9}{(9.676,4.434)}
\gppoint{gp mark 9}{(9.907,4.372)}
\gppoint{gp mark 9}{(10.139,4.309)}
\gppoint{gp mark 9}{(10.371,4.247)}
\gppoint{gp mark 9}{(10.603,4.186)}
\gppoint{gp mark 9}{(10.835,4.125)}
\gppoint{gp mark 9}{(11.066,4.067)}
\gppoint{gp mark 9}{(11.298,4.006)}
\gppoint{gp mark 9}{(11.530,3.951)}
\gppoint{gp mark 9}{(10.554,7.701)}
\node[gp node right] at (9.669,7.026) {DE-Sinc};
\gpsetdashtype{gp dt solid}
\draw[gp path] (9.934,7.026)--(11.174,7.026);
\draw[gp path] (2.721,7.693)--(2.953,7.444)--(3.184,7.210)--(3.416,6.969)--(3.648,6.723)%
  --(3.880,6.484)--(4.112,6.252)--(4.344,6.020)--(4.575,5.792)--(4.807,5.563)--(5.039,5.338)%
  --(5.271,5.114)--(5.503,4.892)--(5.735,4.669)--(5.966,4.448)--(6.198,4.224)--(6.430,4.008)%
  --(6.662,3.786)--(6.894,3.572)--(7.125,3.354)--(7.357,3.136)--(7.589,2.922)--(7.821,2.704)%
  --(8.053,2.497)--(8.285,2.285)--(8.516,2.129)--(8.748,1.926)--(8.980,1.951)--(9.212,1.951)%
  --(9.444,1.973)--(9.676,1.926)--(9.907,1.926)--(10.139,1.898)--(10.371,1.973)--(10.603,1.898)%
  --(10.835,1.926)--(11.066,1.973)--(11.298,1.898)--(11.530,1.951)--(11.704,1.911);
\gppoint{gp mark 7}{(2.721,7.693)}
\gppoint{gp mark 7}{(2.953,7.444)}
\gppoint{gp mark 7}{(3.184,7.210)}
\gppoint{gp mark 7}{(3.416,6.969)}
\gppoint{gp mark 7}{(3.648,6.723)}
\gppoint{gp mark 7}{(3.880,6.484)}
\gppoint{gp mark 7}{(4.112,6.252)}
\gppoint{gp mark 7}{(4.344,6.020)}
\gppoint{gp mark 7}{(4.575,5.792)}
\gppoint{gp mark 7}{(4.807,5.563)}
\gppoint{gp mark 7}{(5.039,5.338)}
\gppoint{gp mark 7}{(5.271,5.114)}
\gppoint{gp mark 7}{(5.503,4.892)}
\gppoint{gp mark 7}{(5.735,4.669)}
\gppoint{gp mark 7}{(5.966,4.448)}
\gppoint{gp mark 7}{(6.198,4.224)}
\gppoint{gp mark 7}{(6.430,4.008)}
\gppoint{gp mark 7}{(6.662,3.786)}
\gppoint{gp mark 7}{(6.894,3.572)}
\gppoint{gp mark 7}{(7.125,3.354)}
\gppoint{gp mark 7}{(7.357,3.136)}
\gppoint{gp mark 7}{(7.589,2.922)}
\gppoint{gp mark 7}{(7.821,2.704)}
\gppoint{gp mark 7}{(8.053,2.497)}
\gppoint{gp mark 7}{(8.285,2.285)}
\gppoint{gp mark 7}{(8.516,2.129)}
\gppoint{gp mark 7}{(8.748,1.926)}
\gppoint{gp mark 7}{(8.980,1.951)}
\gppoint{gp mark 7}{(9.212,1.951)}
\gppoint{gp mark 7}{(9.444,1.973)}
\gppoint{gp mark 7}{(9.676,1.926)}
\gppoint{gp mark 7}{(9.907,1.926)}
\gppoint{gp mark 7}{(10.139,1.898)}
\gppoint{gp mark 7}{(10.371,1.973)}
\gppoint{gp mark 7}{(10.603,1.898)}
\gppoint{gp mark 7}{(10.835,1.926)}
\gppoint{gp mark 7}{(11.066,1.973)}
\gppoint{gp mark 7}{(11.298,1.898)}
\gppoint{gp mark 7}{(11.530,1.951)}
\gppoint{gp mark 7}{(10.554,7.026)}
\draw[gp path] (2.431,8.219)--(2.431,1.420)--(11.704,1.420)--(11.704,8.219)--cycle;
\gpdefrectangularnode{gp plot 1}{\pgfpoint{2.431cm}{1.420cm}}{\pgfpoint{11.704cm}{8.219cm}}
\end{tikzpicture}
}
\caption{Numerical results for Example~\ref{ex:5}.}
\label{fig:ex5}
\end{figure}
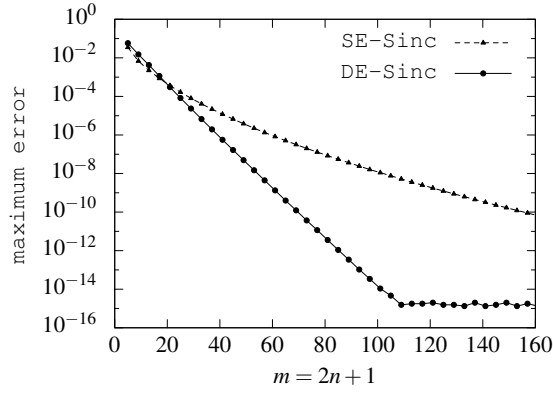
\begin{figure}
\centering
\scalebox{0.6}{
\begin{tikzpicture}[gnuplot]
\tikzset{every node/.append style={font={\ttfamily\fontsize{14.4pt}{17.28pt}\selectfont}}}
\gpmonochromelines
\path (0.000,0.000) rectangle (12.500,8.750);
\gpcolor{color=gp lt color border}
\gpsetlinetype{gp lt border}
\gpsetdashtype{gp dt solid}
\gpsetlinewidth{1.00}
\draw[gp path] (2.431,1.420)--(2.611,1.420);
\draw[gp path] (11.704,1.420)--(11.524,1.420);
\node[gp node right] at (2.166,1.420) {$10^{-16}$};
\draw[gp path] (2.431,1.845)--(2.521,1.845);
\draw[gp path] (11.704,1.845)--(11.614,1.845);
\draw[gp path] (2.431,2.270)--(2.611,2.270);
\draw[gp path] (11.704,2.270)--(11.524,2.270);
\node[gp node right] at (2.166,2.270) {$10^{-14}$};
\draw[gp path] (2.431,2.695)--(2.521,2.695);
\draw[gp path] (11.704,2.695)--(11.614,2.695);
\draw[gp path] (2.431,3.120)--(2.611,3.120);
\draw[gp path] (11.704,3.120)--(11.524,3.120);
\node[gp node right] at (2.166,3.120) {$10^{-12}$};
\draw[gp path] (2.431,3.545)--(2.521,3.545);
\draw[gp path] (11.704,3.545)--(11.614,3.545);
\draw[gp path] (2.431,3.970)--(2.611,3.970);
\draw[gp path] (11.704,3.970)--(11.524,3.970);
\node[gp node right] at (2.166,3.970) {$10^{-10}$};
\draw[gp path] (2.431,4.395)--(2.521,4.395);
\draw[gp path] (11.704,4.395)--(11.614,4.395);
\draw[gp path] (2.431,4.820)--(2.611,4.820);
\draw[gp path] (11.704,4.820)--(11.524,4.820);
\node[gp node right] at (2.166,4.820) {$10^{-8}$};
\draw[gp path] (2.431,5.244)--(2.521,5.244);
\draw[gp path] (11.704,5.244)--(11.614,5.244);
\draw[gp path] (2.431,5.669)--(2.611,5.669);
\draw[gp path] (11.704,5.669)--(11.524,5.669);
\node[gp node right] at (2.166,5.669) {$10^{-6}$};
\draw[gp path] (2.431,6.094)--(2.521,6.094);
\draw[gp path] (11.704,6.094)--(11.614,6.094);
\draw[gp path] (2.431,6.519)--(2.611,6.519);
\draw[gp path] (11.704,6.519)--(11.524,6.519);
\node[gp node right] at (2.166,6.519) {$10^{-4}$};
\draw[gp path] (2.431,6.944)--(2.521,6.944);
\draw[gp path] (11.704,6.944)--(11.614,6.944);
\draw[gp path] (2.431,7.369)--(2.611,7.369);
\draw[gp path] (11.704,7.369)--(11.524,7.369);
\node[gp node right] at (2.166,7.369) {$10^{-2}$};
\draw[gp path] (2.431,7.794)--(2.521,7.794);
\draw[gp path] (11.704,7.794)--(11.614,7.794);
\draw[gp path] (2.431,8.219)--(2.611,8.219);
\draw[gp path] (11.704,8.219)--(11.524,8.219);
\node[gp node right] at (2.166,8.219) {$10^{0}$};
\draw[gp path] (2.431,1.420)--(2.431,1.600);
\draw[gp path] (2.431,8.219)--(2.431,8.039);
\node[gp node center] at (2.431,0.976) {$0$};
\draw[gp path] (3.590,1.420)--(3.590,1.600);
\draw[gp path] (3.590,8.219)--(3.590,8.039);
\node[gp node center] at (3.590,0.976) {$20$};
\draw[gp path] (4.749,1.420)--(4.749,1.600);
\draw[gp path] (4.749,8.219)--(4.749,8.039);
\node[gp node center] at (4.749,0.976) {$40$};
\draw[gp path] (5.908,1.420)--(5.908,1.600);
\draw[gp path] (5.908,8.219)--(5.908,8.039);
\node[gp node center] at (5.908,0.976) {$60$};
\draw[gp path] (7.068,1.420)--(7.068,1.600);
\draw[gp path] (7.068,8.219)--(7.068,8.039);
\node[gp node center] at (7.068,0.976) {$80$};
\draw[gp path] (8.227,1.420)--(8.227,1.600);
\draw[gp path] (8.227,8.219)--(8.227,8.039);
\node[gp node center] at (8.227,0.976) {$100$};
\draw[gp path] (9.386,1.420)--(9.386,1.600);
\draw[gp path] (9.386,8.219)--(9.386,8.039);
\node[gp node center] at (9.386,0.976) {$120$};
\draw[gp path] (10.545,1.420)--(10.545,1.600);
\draw[gp path] (10.545,8.219)--(10.545,8.039);
\node[gp node center] at (10.545,0.976) {$140$};
\draw[gp path] (11.704,1.420)--(11.704,1.600);
\draw[gp path] (11.704,8.219)--(11.704,8.039);
\node[gp node center] at (11.704,0.976) {$160$};
\draw[gp path] (2.431,8.219)--(2.431,1.420)--(11.704,1.420)--(11.704,8.219)--cycle;
\node[gp node center,rotate=-270] at (0.354,4.819) {maximum error};
\node[gp node center] at (7.067,0.310) {$m=2n+1$};
\node[gp node right] at (9.669,7.701) {SE-Sinc};
\gpsetdashtype{gp dt 2}
\draw[gp path] (9.934,7.701)--(11.174,7.701);
\draw[gp path] (2.721,7.583)--(2.953,7.280)--(3.184,7.035)--(3.416,6.841)--(3.648,6.666)%
  --(3.880,6.506)--(4.112,6.361)--(4.344,6.226)--(4.575,6.098)--(4.807,5.978)--(5.039,5.863)%
  --(5.271,5.756)--(5.503,5.650)--(5.735,5.551)--(5.966,5.454)--(6.198,5.358)--(6.430,5.267)%
  --(6.662,5.178)--(6.894,5.093)--(7.125,5.010)--(7.357,4.930)--(7.589,4.850)--(7.821,4.771)%
  --(8.053,4.694)--(8.285,4.621)--(8.516,4.546)--(8.748,4.476)--(8.980,4.405)--(9.212,4.333)%
  --(9.444,4.268)--(9.676,4.197)--(9.907,4.135)--(10.139,4.068)--(10.371,4.006)--(10.603,3.941)%
  --(10.835,3.880)--(11.066,3.818)--(11.298,3.757)--(11.530,3.699)--(11.704,3.653);
\gpsetpointsize{4.00}
\gppoint{gp mark 9}{(2.721,7.583)}
\gppoint{gp mark 9}{(2.953,7.280)}
\gppoint{gp mark 9}{(3.184,7.035)}
\gppoint{gp mark 9}{(3.416,6.841)}
\gppoint{gp mark 9}{(3.648,6.666)}
\gppoint{gp mark 9}{(3.880,6.506)}
\gppoint{gp mark 9}{(4.112,6.361)}
\gppoint{gp mark 9}{(4.344,6.226)}
\gppoint{gp mark 9}{(4.575,6.098)}
\gppoint{gp mark 9}{(4.807,5.978)}
\gppoint{gp mark 9}{(5.039,5.863)}
\gppoint{gp mark 9}{(5.271,5.756)}
\gppoint{gp mark 9}{(5.503,5.650)}
\gppoint{gp mark 9}{(5.735,5.551)}
\gppoint{gp mark 9}{(5.966,5.454)}
\gppoint{gp mark 9}{(6.198,5.358)}
\gppoint{gp mark 9}{(6.430,5.267)}
\gppoint{gp mark 9}{(6.662,5.178)}
\gppoint{gp mark 9}{(6.894,5.093)}
\gppoint{gp mark 9}{(7.125,5.010)}
\gppoint{gp mark 9}{(7.357,4.930)}
\gppoint{gp mark 9}{(7.589,4.850)}
\gppoint{gp mark 9}{(7.821,4.771)}
\gppoint{gp mark 9}{(8.053,4.694)}
\gppoint{gp mark 9}{(8.285,4.621)}
\gppoint{gp mark 9}{(8.516,4.546)}
\gppoint{gp mark 9}{(8.748,4.476)}
\gppoint{gp mark 9}{(8.980,4.405)}
\gppoint{gp mark 9}{(9.212,4.333)}
\gppoint{gp mark 9}{(9.444,4.268)}
\gppoint{gp mark 9}{(9.676,4.197)}
\gppoint{gp mark 9}{(9.907,4.135)}
\gppoint{gp mark 9}{(10.139,4.068)}
\gppoint{gp mark 9}{(10.371,4.006)}
\gppoint{gp mark 9}{(10.603,3.941)}
\gppoint{gp mark 9}{(10.835,3.880)}
\gppoint{gp mark 9}{(11.066,3.818)}
\gppoint{gp mark 9}{(11.298,3.757)}
\gppoint{gp mark 9}{(11.530,3.699)}
\gppoint{gp mark 9}{(10.554,7.701)}
\node[gp node right] at (9.669,7.026) {DE-Sinc};
\gpsetdashtype{gp dt solid}
\draw[gp path] (9.934,7.026)--(11.174,7.026);
\draw[gp path] (2.721,7.726)--(2.953,7.417)--(3.184,7.107)--(3.416,6.817)--(3.648,6.544)%
  --(3.880,6.283)--(4.112,6.032)--(4.344,5.783)--(4.575,5.537)--(4.807,5.294)--(5.039,5.061)%
  --(5.271,4.825)--(5.503,4.590)--(5.735,4.363)--(5.966,4.132)--(6.198,3.908)--(6.430,3.678)%
  --(6.662,3.451)--(6.894,3.234)--(7.125,3.017)--(7.357,2.798)--(7.589,2.563)--(7.821,2.362)%
  --(8.053,2.182)--(8.285,2.054)--(8.516,1.973)--(8.748,2.010)--(8.980,2.026)--(9.212,2.010)%
  --(9.444,1.973)--(9.676,1.992)--(9.907,2.010)--(10.139,1.962)--(10.371,1.973)--(10.603,1.983)%
  --(10.835,2.010)--(11.066,1.951)--(11.298,2.026)--(11.530,1.973)--(11.704,2.000);
\gppoint{gp mark 7}{(2.721,7.726)}
\gppoint{gp mark 7}{(2.953,7.417)}
\gppoint{gp mark 7}{(3.184,7.107)}
\gppoint{gp mark 7}{(3.416,6.817)}
\gppoint{gp mark 7}{(3.648,6.544)}
\gppoint{gp mark 7}{(3.880,6.283)}
\gppoint{gp mark 7}{(4.112,6.032)}
\gppoint{gp mark 7}{(4.344,5.783)}
\gppoint{gp mark 7}{(4.575,5.537)}
\gppoint{gp mark 7}{(4.807,5.294)}
\gppoint{gp mark 7}{(5.039,5.061)}
\gppoint{gp mark 7}{(5.271,4.825)}
\gppoint{gp mark 7}{(5.503,4.590)}
\gppoint{gp mark 7}{(5.735,4.363)}
\gppoint{gp mark 7}{(5.966,4.132)}
\gppoint{gp mark 7}{(6.198,3.908)}
\gppoint{gp mark 7}{(6.430,3.678)}
\gppoint{gp mark 7}{(6.662,3.451)}
\gppoint{gp mark 7}{(6.894,3.234)}
\gppoint{gp mark 7}{(7.125,3.017)}
\gppoint{gp mark 7}{(7.357,2.798)}
\gppoint{gp mark 7}{(7.589,2.563)}
\gppoint{gp mark 7}{(7.821,2.362)}
\gppoint{gp mark 7}{(8.053,2.182)}
\gppoint{gp mark 7}{(8.285,2.054)}
\gppoint{gp mark 7}{(8.516,1.973)}
\gppoint{gp mark 7}{(8.748,2.010)}
\gppoint{gp mark 7}{(8.980,2.026)}
\gppoint{gp mark 7}{(9.212,2.010)}
\gppoint{gp mark 7}{(9.444,1.973)}
\gppoint{gp mark 7}{(9.676,1.992)}
\gppoint{gp mark 7}{(9.907,2.010)}
\gppoint{gp mark 7}{(10.139,1.962)}
\gppoint{gp mark 7}{(10.371,1.973)}
\gppoint{gp mark 7}{(10.603,1.983)}
\gppoint{gp mark 7}{(10.835,2.010)}
\gppoint{gp mark 7}{(11.066,1.951)}
\gppoint{gp mark 7}{(11.298,2.026)}
\gppoint{gp mark 7}{(11.530,1.973)}
\gppoint{gp mark 7}{(10.554,7.026)}
\draw[gp path] (2.431,8.219)--(2.431,1.420)--(11.704,1.420)--(11.704,8.219)--cycle;
\gpdefrectangularnode{gp plot 1}{\pgfpoint{2.431cm}{1.420cm}}{\pgfpoint{11.704cm}{8.219cm}}
\end{tikzpicture}
}
\caption{Numerical results for Example~\ref{ex:6}.}
\label{fig:ex6}
\end{figure}
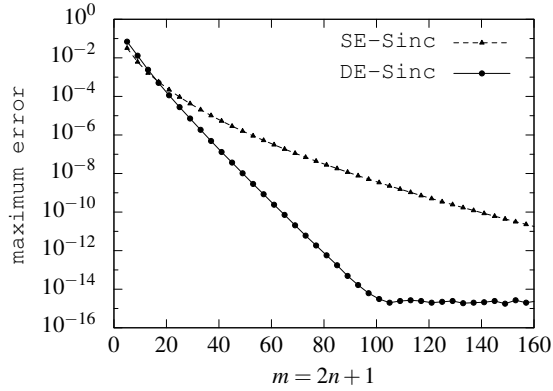

In the case of Example~\ref{ex:3},
$F(s)$ is an entire function.
In the case of Example~\ref{ex:4},
$F(s)$ has a pole on $\Omega^{+}=\{z\in\mathbb{R}:\Re z > 0\}$,
and Theorem~\ref{thm:SE-old} does not support in such a case
(whereas Theorems~\ref{thm:SE-new} and~\ref{thm:DE-new} do).
In the case of Example~\ref{ex:5},
$F(s)$ has poles at $s=\pm \I$.
In the case of Example~\ref{ex:6},
$F(s)$ has branch points at $s=\pm \I$.
The numerical results are shown in Figs.~\ref{fig:ex3}--\ref{fig:ex6}.
We observe that the SE-Sinc and DE-Sinc convolution
exhibit nearly identical convergence behavior in all figures,
regardless of whether $F(s)$ has singular points
or not. This is because $F(s)$ satisfies
Assumption~\ref{assump:F} in all the four examples.



\subsection{Cases where \texorpdfstring{$F$}{F} is not analytic at the origin}

Finally, let us consider the cases
where $F(s)$ does not satisfy Assumption~\ref{assump:F}.
In the subsequent three examples,
$d$ is heuristically set as $d=3.14$ in the SE-Sinc convolution,
and $d=1.57$ in the DE-Sinc convolution,
without theoretical justification.

\begin{example}
\label{ex:7}
Consider the following indefinite convolution
\[
\int_0^x \log(x - t)\sqrt{t}\D{t}
 = \frac{2}{9}x^{3/2}\left\{-8 + 3\log(4x)\right\},
\quad 0\leq x\leq 2.
\]
In this case,
$\hat{f}(s) = -(\gamma + \log s) / s$ and $F(s)=s(-\gamma + \log s)$,
where $\gamma$ is Euler's constant.
\end{example}

\begin{example}
\label{ex:8}
Consider the following indefinite convolution
\[
\int_0^x \frac{(x - t)^{1/3}}{\Gamma(4/3)}\sqrt{t}\D{t}
 = \frac{\sqrt{\pi}}{2\Gamma(17/6)}x^{11/6},
\quad 0\leq x\leq 2.
\]
In this case,
$\hat{f}(s) = s^{-4/3}$ and $F(s)=s^{4/3}$.
\end{example}

\begin{example}
\label{ex:9}
Consider the following indefinite convolution
\[
\int_0^x H(x - t -1)\sqrt{t}\D{t}
 = \dfrac{2}{3}(x-1)^{3/2} H(x - 1), \quad 0\leq x\leq 2,
\]
where $H(x)$ is the Heaviside step function.
In this case,
$\hat{f}(s) = \E^{-s} / s$ and $F(s)=s \E^{-1/s}$.
\end{example}

In the case of Example~\ref{ex:7},
$F(s)$ has a logarithmic branch point at the origin.
In the case of Example~\ref{ex:8},
$F(s)$ has an algebraic branch point at the origin.
In the case of Example~\ref{ex:9},
$F(s)$ has an essential singular point at the origin.
Therefore,
$F(s)$ violates Assumption~\ref{assump:F}
in all the three examples.
Nevertheless, the numerical results (Figs.~\ref{fig:ex7}--\ref{fig:ex9})
indicate that the SE-Sinc convolution in Examples~\ref{ex:7}
and~\ref{ex:8} performs robustly exhibiting its characteristic
convergence rate: $\OO(\exp(-c\sqrt{m}))$.
This is not the case in Example~\ref{ex:9};
the SE-Sinc convolution no longer attains
root-exponential convergence,
although convergence of some order is still observed.

The DE-Sinc convolution shows a complex behavior.
In Example~\ref{ex:7},
it seems to converge with its usual rate
at the first stage ($m\leq 25$),
but the rate becomes worse thereafter.
Similarly, in Example~\ref{ex:8},
it seems to converge with its usual rate
at the first stage ($m\leq 41$),
but the rate becomes worse thereafter.
In Example~\ref{ex:9},
the DE-Sinc convolution fails to converge for $m\geq 17$.
This is because $\exp(-(\aDEm)^{-1})$ becomes a zero matrix
for $m\geq 17$ in the computation of $F(\aDEm)$.

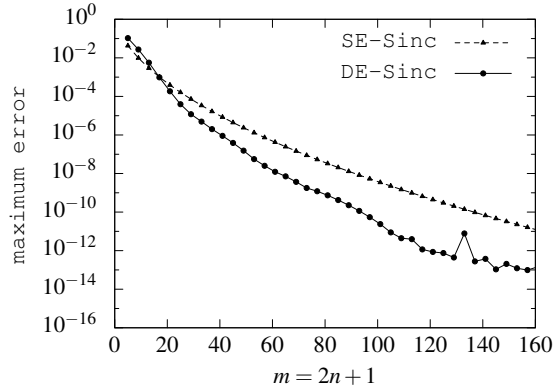
\begin{figure}
\centering
\scalebox{0.6}{
\begin{tikzpicture}[gnuplot]
\tikzset{every node/.append style={font={\ttfamily\fontsize{14.4pt}{17.28pt}\selectfont}}}
\gpmonochromelines
\path (0.000,0.000) rectangle (12.500,8.750);
\gpcolor{color=gp lt color border}
\gpsetlinetype{gp lt border}
\gpsetdashtype{gp dt solid}
\gpsetlinewidth{1.00}
\draw[gp path] (2.431,1.420)--(2.611,1.420);
\draw[gp path] (11.704,1.420)--(11.524,1.420);
\node[gp node right] at (2.166,1.420) {$10^{-16}$};
\draw[gp path] (2.431,1.845)--(2.521,1.845);
\draw[gp path] (11.704,1.845)--(11.614,1.845);
\draw[gp path] (2.431,2.270)--(2.611,2.270);
\draw[gp path] (11.704,2.270)--(11.524,2.270);
\node[gp node right] at (2.166,2.270) {$10^{-14}$};
\draw[gp path] (2.431,2.695)--(2.521,2.695);
\draw[gp path] (11.704,2.695)--(11.614,2.695);
\draw[gp path] (2.431,3.120)--(2.611,3.120);
\draw[gp path] (11.704,3.120)--(11.524,3.120);
\node[gp node right] at (2.166,3.120) {$10^{-12}$};
\draw[gp path] (2.431,3.545)--(2.521,3.545);
\draw[gp path] (11.704,3.545)--(11.614,3.545);
\draw[gp path] (2.431,3.970)--(2.611,3.970);
\draw[gp path] (11.704,3.970)--(11.524,3.970);
\node[gp node right] at (2.166,3.970) {$10^{-10}$};
\draw[gp path] (2.431,4.395)--(2.521,4.395);
\draw[gp path] (11.704,4.395)--(11.614,4.395);
\draw[gp path] (2.431,4.820)--(2.611,4.820);
\draw[gp path] (11.704,4.820)--(11.524,4.820);
\node[gp node right] at (2.166,4.820) {$10^{-8}$};
\draw[gp path] (2.431,5.244)--(2.521,5.244);
\draw[gp path] (11.704,5.244)--(11.614,5.244);
\draw[gp path] (2.431,5.669)--(2.611,5.669);
\draw[gp path] (11.704,5.669)--(11.524,5.669);
\node[gp node right] at (2.166,5.669) {$10^{-6}$};
\draw[gp path] (2.431,6.094)--(2.521,6.094);
\draw[gp path] (11.704,6.094)--(11.614,6.094);
\draw[gp path] (2.431,6.519)--(2.611,6.519);
\draw[gp path] (11.704,6.519)--(11.524,6.519);
\node[gp node right] at (2.166,6.519) {$10^{-4}$};
\draw[gp path] (2.431,6.944)--(2.521,6.944);
\draw[gp path] (11.704,6.944)--(11.614,6.944);
\draw[gp path] (2.431,7.369)--(2.611,7.369);
\draw[gp path] (11.704,7.369)--(11.524,7.369);
\node[gp node right] at (2.166,7.369) {$10^{-2}$};
\draw[gp path] (2.431,7.794)--(2.521,7.794);
\draw[gp path] (11.704,7.794)--(11.614,7.794);
\draw[gp path] (2.431,8.219)--(2.611,8.219);
\draw[gp path] (11.704,8.219)--(11.524,8.219);
\node[gp node right] at (2.166,8.219) {$10^{0}$};
\draw[gp path] (2.431,1.420)--(2.431,1.600);
\draw[gp path] (2.431,8.219)--(2.431,8.039);
\node[gp node center] at (2.431,0.976) {$0$};
\draw[gp path] (3.590,1.420)--(3.590,1.600);
\draw[gp path] (3.590,8.219)--(3.590,8.039);
\node[gp node center] at (3.590,0.976) {$20$};
\draw[gp path] (4.749,1.420)--(4.749,1.600);
\draw[gp path] (4.749,8.219)--(4.749,8.039);
\node[gp node center] at (4.749,0.976) {$40$};
\draw[gp path] (5.908,1.420)--(5.908,1.600);
\draw[gp path] (5.908,8.219)--(5.908,8.039);
\node[gp node center] at (5.908,0.976) {$60$};
\draw[gp path] (7.068,1.420)--(7.068,1.600);
\draw[gp path] (7.068,8.219)--(7.068,8.039);
\node[gp node center] at (7.068,0.976) {$80$};
\draw[gp path] (8.227,1.420)--(8.227,1.600);
\draw[gp path] (8.227,8.219)--(8.227,8.039);
\node[gp node center] at (8.227,0.976) {$100$};
\draw[gp path] (9.386,1.420)--(9.386,1.600);
\draw[gp path] (9.386,8.219)--(9.386,8.039);
\node[gp node center] at (9.386,0.976) {$120$};
\draw[gp path] (10.545,1.420)--(10.545,1.600);
\draw[gp path] (10.545,8.219)--(10.545,8.039);
\node[gp node center] at (10.545,0.976) {$140$};
\draw[gp path] (11.704,1.420)--(11.704,1.600);
\draw[gp path] (11.704,8.219)--(11.704,8.039);
\node[gp node center] at (11.704,0.976) {$160$};
\draw[gp path] (2.431,8.219)--(2.431,1.420)--(11.704,1.420)--(11.704,8.219)--cycle;
\node[gp node center,rotate=-270] at (0.354,4.819) {maximum error};
\node[gp node center] at (7.067,0.310) {$m=2n+1$};
\node[gp node right] at (9.669,7.701) {SE-Sinc};
\gpsetdashtype{gp dt 2}
\draw[gp path] (9.934,7.701)--(11.174,7.701);
\draw[gp path] (2.721,7.636)--(2.953,7.365)--(3.184,7.142)--(3.416,6.946)--(3.648,6.770)%
  --(3.880,6.609)--(4.112,6.460)--(4.344,6.320)--(4.575,6.188)--(4.807,6.063)--(5.039,5.943)%
  --(5.271,5.829)--(5.503,5.719)--(5.735,5.613)--(5.966,5.510)--(6.198,5.411)--(6.430,5.315)%
  --(6.662,5.221)--(6.894,5.129)--(7.125,5.040)--(7.357,4.953)--(7.589,4.868)--(7.821,4.785)%
  --(8.053,4.703)--(8.285,4.624)--(8.516,4.545)--(8.748,4.468)--(8.980,4.393)--(9.212,4.318)%
  --(9.444,4.245)--(9.676,4.173)--(9.907,4.103)--(10.139,4.033)--(10.371,3.964)--(10.603,3.896)%
  --(10.835,3.829)--(11.066,3.763)--(11.298,3.698)--(11.530,3.635)--(11.704,3.587);
\gpsetpointsize{4.00}
\gppoint{gp mark 9}{(2.721,7.636)}
\gppoint{gp mark 9}{(2.953,7.365)}
\gppoint{gp mark 9}{(3.184,7.142)}
\gppoint{gp mark 9}{(3.416,6.946)}
\gppoint{gp mark 9}{(3.648,6.770)}
\gppoint{gp mark 9}{(3.880,6.609)}
\gppoint{gp mark 9}{(4.112,6.460)}
\gppoint{gp mark 9}{(4.344,6.320)}
\gppoint{gp mark 9}{(4.575,6.188)}
\gppoint{gp mark 9}{(4.807,6.063)}
\gppoint{gp mark 9}{(5.039,5.943)}
\gppoint{gp mark 9}{(5.271,5.829)}
\gppoint{gp mark 9}{(5.503,5.719)}
\gppoint{gp mark 9}{(5.735,5.613)}
\gppoint{gp mark 9}{(5.966,5.510)}
\gppoint{gp mark 9}{(6.198,5.411)}
\gppoint{gp mark 9}{(6.430,5.315)}
\gppoint{gp mark 9}{(6.662,5.221)}
\gppoint{gp mark 9}{(6.894,5.129)}
\gppoint{gp mark 9}{(7.125,5.040)}
\gppoint{gp mark 9}{(7.357,4.953)}
\gppoint{gp mark 9}{(7.589,4.868)}
\gppoint{gp mark 9}{(7.821,4.785)}
\gppoint{gp mark 9}{(8.053,4.703)}
\gppoint{gp mark 9}{(8.285,4.624)}
\gppoint{gp mark 9}{(8.516,4.545)}
\gppoint{gp mark 9}{(8.748,4.468)}
\gppoint{gp mark 9}{(8.980,4.393)}
\gppoint{gp mark 9}{(9.212,4.318)}
\gppoint{gp mark 9}{(9.444,4.245)}
\gppoint{gp mark 9}{(9.676,4.173)}
\gppoint{gp mark 9}{(9.907,4.103)}
\gppoint{gp mark 9}{(10.139,4.033)}
\gppoint{gp mark 9}{(10.371,3.964)}
\gppoint{gp mark 9}{(10.603,3.896)}
\gppoint{gp mark 9}{(10.835,3.829)}
\gppoint{gp mark 9}{(11.066,3.763)}
\gppoint{gp mark 9}{(11.298,3.698)}
\gppoint{gp mark 9}{(11.530,3.635)}
\gppoint{gp mark 9}{(10.554,7.701)}
\node[gp node right] at (9.669,7.026) {DE-Sinc};
\gpsetdashtype{gp dt solid}
\draw[gp path] (9.934,7.026)--(11.174,7.026);
\draw[gp path] (2.721,7.804)--(2.953,7.553)--(3.184,7.263)--(3.416,6.939)--(3.648,6.632)%
  --(3.880,6.347)--(4.112,6.127)--(4.344,5.964)--(4.575,5.798)--(4.807,5.650)--(5.039,5.494)%
  --(5.271,5.324)--(5.503,5.137)--(5.735,4.990)--(5.966,4.858)--(6.198,4.757)--(6.430,4.636)%
  --(6.662,4.501)--(6.894,4.429)--(7.125,4.339)--(7.357,4.235)--(7.589,4.121)--(7.821,3.996)%
  --(8.053,3.858)--(8.285,3.705)--(8.516,3.522)--(8.748,3.394)--(8.980,3.373)--(9.212,3.147)%
  --(9.444,3.092)--(9.676,3.067)--(9.907,2.969)--(10.139,3.499)--(10.371,2.883)--(10.603,2.939)%
  --(10.835,2.708)--(11.066,2.828)--(11.298,2.732)--(11.530,2.692)--(11.704,2.747);
\gppoint{gp mark 7}{(2.721,7.804)}
\gppoint{gp mark 7}{(2.953,7.553)}
\gppoint{gp mark 7}{(3.184,7.263)}
\gppoint{gp mark 7}{(3.416,6.939)}
\gppoint{gp mark 7}{(3.648,6.632)}
\gppoint{gp mark 7}{(3.880,6.347)}
\gppoint{gp mark 7}{(4.112,6.127)}
\gppoint{gp mark 7}{(4.344,5.964)}
\gppoint{gp mark 7}{(4.575,5.798)}
\gppoint{gp mark 7}{(4.807,5.650)}
\gppoint{gp mark 7}{(5.039,5.494)}
\gppoint{gp mark 7}{(5.271,5.324)}
\gppoint{gp mark 7}{(5.503,5.137)}
\gppoint{gp mark 7}{(5.735,4.990)}
\gppoint{gp mark 7}{(5.966,4.858)}
\gppoint{gp mark 7}{(6.198,4.757)}
\gppoint{gp mark 7}{(6.430,4.636)}
\gppoint{gp mark 7}{(6.662,4.501)}
\gppoint{gp mark 7}{(6.894,4.429)}
\gppoint{gp mark 7}{(7.125,4.339)}
\gppoint{gp mark 7}{(7.357,4.235)}
\gppoint{gp mark 7}{(7.589,4.121)}
\gppoint{gp mark 7}{(7.821,3.996)}
\gppoint{gp mark 7}{(8.053,3.858)}
\gppoint{gp mark 7}{(8.285,3.705)}
\gppoint{gp mark 7}{(8.516,3.522)}
\gppoint{gp mark 7}{(8.748,3.394)}
\gppoint{gp mark 7}{(8.980,3.373)}
\gppoint{gp mark 7}{(9.212,3.147)}
\gppoint{gp mark 7}{(9.444,3.092)}
\gppoint{gp mark 7}{(9.676,3.067)}
\gppoint{gp mark 7}{(9.907,2.969)}
\gppoint{gp mark 7}{(10.139,3.499)}
\gppoint{gp mark 7}{(10.371,2.883)}
\gppoint{gp mark 7}{(10.603,2.939)}
\gppoint{gp mark 7}{(10.835,2.708)}
\gppoint{gp mark 7}{(11.066,2.828)}
\gppoint{gp mark 7}{(11.298,2.732)}
\gppoint{gp mark 7}{(11.530,2.692)}
\gppoint{gp mark 7}{(10.554,7.026)}
\draw[gp path] (2.431,8.219)--(2.431,1.420)--(11.704,1.420)--(11.704,8.219)--cycle;
\gpdefrectangularnode{gp plot 1}{\pgfpoint{2.431cm}{1.420cm}}{\pgfpoint{11.704cm}{8.219cm}}
\end{tikzpicture}
}
\caption{Numerical results for Example~\ref{ex:7}.}
\label{fig:ex7}
\end{figure}
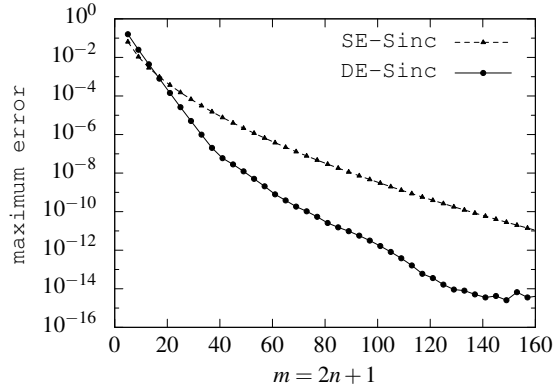
\begin{figure}
\centering
\scalebox{0.6}{
\begin{tikzpicture}[gnuplot]
\tikzset{every node/.append style={font={\ttfamily\fontsize{14.4pt}{17.28pt}\selectfont}}}
\gpmonochromelines
\path (0.000,0.000) rectangle (12.500,8.750);
\gpcolor{color=gp lt color border}
\gpsetlinetype{gp lt border}
\gpsetdashtype{gp dt solid}
\gpsetlinewidth{1.00}
\draw[gp path] (2.431,1.420)--(2.611,1.420);
\draw[gp path] (11.704,1.420)--(11.524,1.420);
\node[gp node right] at (2.166,1.420) {$10^{-16}$};
\draw[gp path] (2.431,1.845)--(2.521,1.845);
\draw[gp path] (11.704,1.845)--(11.614,1.845);
\draw[gp path] (2.431,2.270)--(2.611,2.270);
\draw[gp path] (11.704,2.270)--(11.524,2.270);
\node[gp node right] at (2.166,2.270) {$10^{-14}$};
\draw[gp path] (2.431,2.695)--(2.521,2.695);
\draw[gp path] (11.704,2.695)--(11.614,2.695);
\draw[gp path] (2.431,3.120)--(2.611,3.120);
\draw[gp path] (11.704,3.120)--(11.524,3.120);
\node[gp node right] at (2.166,3.120) {$10^{-12}$};
\draw[gp path] (2.431,3.545)--(2.521,3.545);
\draw[gp path] (11.704,3.545)--(11.614,3.545);
\draw[gp path] (2.431,3.970)--(2.611,3.970);
\draw[gp path] (11.704,3.970)--(11.524,3.970);
\node[gp node right] at (2.166,3.970) {$10^{-10}$};
\draw[gp path] (2.431,4.395)--(2.521,4.395);
\draw[gp path] (11.704,4.395)--(11.614,4.395);
\draw[gp path] (2.431,4.820)--(2.611,4.820);
\draw[gp path] (11.704,4.820)--(11.524,4.820);
\node[gp node right] at (2.166,4.820) {$10^{-8}$};
\draw[gp path] (2.431,5.244)--(2.521,5.244);
\draw[gp path] (11.704,5.244)--(11.614,5.244);
\draw[gp path] (2.431,5.669)--(2.611,5.669);
\draw[gp path] (11.704,5.669)--(11.524,5.669);
\node[gp node right] at (2.166,5.669) {$10^{-6}$};
\draw[gp path] (2.431,6.094)--(2.521,6.094);
\draw[gp path] (11.704,6.094)--(11.614,6.094);
\draw[gp path] (2.431,6.519)--(2.611,6.519);
\draw[gp path] (11.704,6.519)--(11.524,6.519);
\node[gp node right] at (2.166,6.519) {$10^{-4}$};
\draw[gp path] (2.431,6.944)--(2.521,6.944);
\draw[gp path] (11.704,6.944)--(11.614,6.944);
\draw[gp path] (2.431,7.369)--(2.611,7.369);
\draw[gp path] (11.704,7.369)--(11.524,7.369);
\node[gp node right] at (2.166,7.369) {$10^{-2}$};
\draw[gp path] (2.431,7.794)--(2.521,7.794);
\draw[gp path] (11.704,7.794)--(11.614,7.794);
\draw[gp path] (2.431,8.219)--(2.611,8.219);
\draw[gp path] (11.704,8.219)--(11.524,8.219);
\node[gp node right] at (2.166,8.219) {$10^{0}$};
\draw[gp path] (2.431,1.420)--(2.431,1.600);
\draw[gp path] (2.431,8.219)--(2.431,8.039);
\node[gp node center] at (2.431,0.976) {$0$};
\draw[gp path] (3.590,1.420)--(3.590,1.600);
\draw[gp path] (3.590,8.219)--(3.590,8.039);
\node[gp node center] at (3.590,0.976) {$20$};
\draw[gp path] (4.749,1.420)--(4.749,1.600);
\draw[gp path] (4.749,8.219)--(4.749,8.039);
\node[gp node center] at (4.749,0.976) {$40$};
\draw[gp path] (5.908,1.420)--(5.908,1.600);
\draw[gp path] (5.908,8.219)--(5.908,8.039);
\node[gp node center] at (5.908,0.976) {$60$};
\draw[gp path] (7.068,1.420)--(7.068,1.600);
\draw[gp path] (7.068,8.219)--(7.068,8.039);
\node[gp node center] at (7.068,0.976) {$80$};
\draw[gp path] (8.227,1.420)--(8.227,1.600);
\draw[gp path] (8.227,8.219)--(8.227,8.039);
\node[gp node center] at (8.227,0.976) {$100$};
\draw[gp path] (9.386,1.420)--(9.386,1.600);
\draw[gp path] (9.386,8.219)--(9.386,8.039);
\node[gp node center] at (9.386,0.976) {$120$};
\draw[gp path] (10.545,1.420)--(10.545,1.600);
\draw[gp path] (10.545,8.219)--(10.545,8.039);
\node[gp node center] at (10.545,0.976) {$140$};
\draw[gp path] (11.704,1.420)--(11.704,1.600);
\draw[gp path] (11.704,8.219)--(11.704,8.039);
\node[gp node center] at (11.704,0.976) {$160$};
\draw[gp path] (2.431,8.219)--(2.431,1.420)--(11.704,1.420)--(11.704,8.219)--cycle;
\node[gp node center,rotate=-270] at (0.354,4.819) {maximum error};
\node[gp node center] at (7.067,0.310) {$m=2n+1$};
\node[gp node right] at (9.669,7.701) {SE-Sinc};
\gpsetdashtype{gp dt 2}
\draw[gp path] (9.934,7.701)--(11.174,7.701);
\draw[gp path] (2.721,7.712)--(2.953,7.383)--(3.184,7.140)--(3.416,6.939)--(3.648,6.760)%
  --(3.880,6.597)--(4.112,6.446)--(4.344,6.305)--(4.575,6.171)--(4.807,6.045)--(5.039,5.925)%
  --(5.271,5.809)--(5.503,5.699)--(5.735,5.592)--(5.966,5.489)--(6.198,5.389)--(6.430,5.292)%
  --(6.662,5.198)--(6.894,5.106)--(7.125,5.016)--(7.357,4.929)--(7.589,4.844)--(7.821,4.760)%
  --(8.053,4.679)--(8.285,4.599)--(8.516,4.520)--(8.748,4.443)--(8.980,4.367)--(9.212,4.292)%
  --(9.444,4.220)--(9.676,4.148)--(9.907,4.076)--(10.139,4.007)--(10.371,3.938)--(10.603,3.870)%
  --(10.835,3.803)--(11.066,3.737)--(11.298,3.672)--(11.530,3.608)--(11.704,3.560);
\gpsetpointsize{4.00}
\gppoint{gp mark 9}{(2.721,7.712)}
\gppoint{gp mark 9}{(2.953,7.383)}
\gppoint{gp mark 9}{(3.184,7.140)}
\gppoint{gp mark 9}{(3.416,6.939)}
\gppoint{gp mark 9}{(3.648,6.760)}
\gppoint{gp mark 9}{(3.880,6.597)}
\gppoint{gp mark 9}{(4.112,6.446)}
\gppoint{gp mark 9}{(4.344,6.305)}
\gppoint{gp mark 9}{(4.575,6.171)}
\gppoint{gp mark 9}{(4.807,6.045)}
\gppoint{gp mark 9}{(5.039,5.925)}
\gppoint{gp mark 9}{(5.271,5.809)}
\gppoint{gp mark 9}{(5.503,5.699)}
\gppoint{gp mark 9}{(5.735,5.592)}
\gppoint{gp mark 9}{(5.966,5.489)}
\gppoint{gp mark 9}{(6.198,5.389)}
\gppoint{gp mark 9}{(6.430,5.292)}
\gppoint{gp mark 9}{(6.662,5.198)}
\gppoint{gp mark 9}{(6.894,5.106)}
\gppoint{gp mark 9}{(7.125,5.016)}
\gppoint{gp mark 9}{(7.357,4.929)}
\gppoint{gp mark 9}{(7.589,4.844)}
\gppoint{gp mark 9}{(7.821,4.760)}
\gppoint{gp mark 9}{(8.053,4.679)}
\gppoint{gp mark 9}{(8.285,4.599)}
\gppoint{gp mark 9}{(8.516,4.520)}
\gppoint{gp mark 9}{(8.748,4.443)}
\gppoint{gp mark 9}{(8.980,4.367)}
\gppoint{gp mark 9}{(9.212,4.292)}
\gppoint{gp mark 9}{(9.444,4.220)}
\gppoint{gp mark 9}{(9.676,4.148)}
\gppoint{gp mark 9}{(9.907,4.076)}
\gppoint{gp mark 9}{(10.139,4.007)}
\gppoint{gp mark 9}{(10.371,3.938)}
\gppoint{gp mark 9}{(10.603,3.870)}
\gppoint{gp mark 9}{(10.835,3.803)}
\gppoint{gp mark 9}{(11.066,3.737)}
\gppoint{gp mark 9}{(11.298,3.672)}
\gppoint{gp mark 9}{(11.530,3.608)}
\gppoint{gp mark 9}{(10.554,7.701)}
\node[gp node right] at (9.669,7.026) {DE-Sinc};
\gpsetdashtype{gp dt solid}
\draw[gp path] (9.934,7.026)--(11.174,7.026);
\draw[gp path] (2.721,7.880)--(2.953,7.538)--(3.184,7.217)--(3.416,6.901)--(3.648,6.585)%
  --(3.880,6.272)--(4.112,5.968)--(4.344,5.667)--(4.575,5.375)--(4.807,5.148)--(5.039,5.011)%
  --(5.271,4.858)--(5.503,4.693)--(5.735,4.528)--(5.966,4.351)--(6.198,4.217)--(6.430,4.081)%
  --(6.662,3.977)--(6.894,3.856)--(7.125,3.720)--(7.357,3.624)--(7.589,3.539)--(7.821,3.441)%
  --(8.053,3.332)--(8.285,3.212)--(8.516,3.082)--(8.748,2.942)--(8.980,2.784)--(9.212,2.600)%
  --(9.444,2.506)--(9.676,2.362)--(9.907,2.253)--(10.139,2.229)--(10.371,2.146)--(10.603,2.079)%
  --(10.835,2.111)--(11.066,2.018)--(11.298,2.195)--(11.530,2.079)--(11.704,2.103);
\gppoint{gp mark 7}{(2.721,7.880)}
\gppoint{gp mark 7}{(2.953,7.538)}
\gppoint{gp mark 7}{(3.184,7.217)}
\gppoint{gp mark 7}{(3.416,6.901)}
\gppoint{gp mark 7}{(3.648,6.585)}
\gppoint{gp mark 7}{(3.880,6.272)}
\gppoint{gp mark 7}{(4.112,5.968)}
\gppoint{gp mark 7}{(4.344,5.667)}
\gppoint{gp mark 7}{(4.575,5.375)}
\gppoint{gp mark 7}{(4.807,5.148)}
\gppoint{gp mark 7}{(5.039,5.011)}
\gppoint{gp mark 7}{(5.271,4.858)}
\gppoint{gp mark 7}{(5.503,4.693)}
\gppoint{gp mark 7}{(5.735,4.528)}
\gppoint{gp mark 7}{(5.966,4.351)}
\gppoint{gp mark 7}{(6.198,4.217)}
\gppoint{gp mark 7}{(6.430,4.081)}
\gppoint{gp mark 7}{(6.662,3.977)}
\gppoint{gp mark 7}{(6.894,3.856)}
\gppoint{gp mark 7}{(7.125,3.720)}
\gppoint{gp mark 7}{(7.357,3.624)}
\gppoint{gp mark 7}{(7.589,3.539)}
\gppoint{gp mark 7}{(7.821,3.441)}
\gppoint{gp mark 7}{(8.053,3.332)}
\gppoint{gp mark 7}{(8.285,3.212)}
\gppoint{gp mark 7}{(8.516,3.082)}
\gppoint{gp mark 7}{(8.748,2.942)}
\gppoint{gp mark 7}{(8.980,2.784)}
\gppoint{gp mark 7}{(9.212,2.600)}
\gppoint{gp mark 7}{(9.444,2.506)}
\gppoint{gp mark 7}{(9.676,2.362)}
\gppoint{gp mark 7}{(9.907,2.253)}
\gppoint{gp mark 7}{(10.139,2.229)}
\gppoint{gp mark 7}{(10.371,2.146)}
\gppoint{gp mark 7}{(10.603,2.079)}
\gppoint{gp mark 7}{(10.835,2.111)}
\gppoint{gp mark 7}{(11.066,2.018)}
\gppoint{gp mark 7}{(11.298,2.195)}
\gppoint{gp mark 7}{(11.530,2.079)}
\gppoint{gp mark 7}{(10.554,7.026)}
\draw[gp path] (2.431,8.219)--(2.431,1.420)--(11.704,1.420)--(11.704,8.219)--cycle;
\gpdefrectangularnode{gp plot 1}{\pgfpoint{2.431cm}{1.420cm}}{\pgfpoint{11.704cm}{8.219cm}}
\end{tikzpicture}
}
\caption{Numerical results for Example~\ref{ex:8}.}
\label{fig:ex8}
\end{figure}
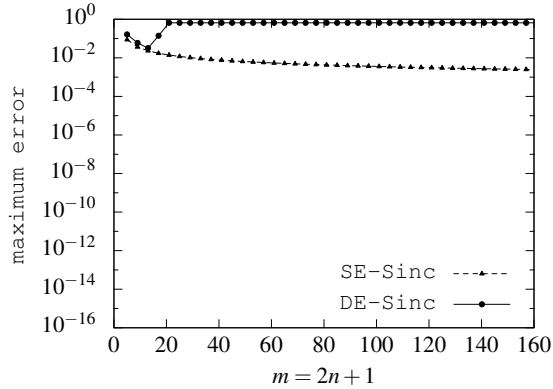
\begin{figure}
\centering
\scalebox{0.6}{
\begin{tikzpicture}[gnuplot]
\tikzset{every node/.append style={font={\ttfamily\fontsize{14.4pt}{17.28pt}\selectfont}}}
\gpmonochromelines
\path (0.000,0.000) rectangle (12.500,8.750);
\gpcolor{color=gp lt color border}
\gpsetlinetype{gp lt border}
\gpsetdashtype{gp dt solid}
\gpsetlinewidth{1.00}
\draw[gp path] (2.431,1.420)--(2.611,1.420);
\draw[gp path] (11.704,1.420)--(11.524,1.420);
\node[gp node right] at (2.166,1.420) {$10^{-16}$};
\draw[gp path] (2.431,1.845)--(2.521,1.845);
\draw[gp path] (11.704,1.845)--(11.614,1.845);
\draw[gp path] (2.431,2.270)--(2.611,2.270);
\draw[gp path] (11.704,2.270)--(11.524,2.270);
\node[gp node right] at (2.166,2.270) {$10^{-14}$};
\draw[gp path] (2.431,2.695)--(2.521,2.695);
\draw[gp path] (11.704,2.695)--(11.614,2.695);
\draw[gp path] (2.431,3.120)--(2.611,3.120);
\draw[gp path] (11.704,3.120)--(11.524,3.120);
\node[gp node right] at (2.166,3.120) {$10^{-12}$};
\draw[gp path] (2.431,3.545)--(2.521,3.545);
\draw[gp path] (11.704,3.545)--(11.614,3.545);
\draw[gp path] (2.431,3.970)--(2.611,3.970);
\draw[gp path] (11.704,3.970)--(11.524,3.970);
\node[gp node right] at (2.166,3.970) {$10^{-10}$};
\draw[gp path] (2.431,4.395)--(2.521,4.395);
\draw[gp path] (11.704,4.395)--(11.614,4.395);
\draw[gp path] (2.431,4.820)--(2.611,4.820);
\draw[gp path] (11.704,4.820)--(11.524,4.820);
\node[gp node right] at (2.166,4.820) {$10^{-8}$};
\draw[gp path] (2.431,5.244)--(2.521,5.244);
\draw[gp path] (11.704,5.244)--(11.614,5.244);
\draw[gp path] (2.431,5.669)--(2.611,5.669);
\draw[gp path] (11.704,5.669)--(11.524,5.669);
\node[gp node right] at (2.166,5.669) {$10^{-6}$};
\draw[gp path] (2.431,6.094)--(2.521,6.094);
\draw[gp path] (11.704,6.094)--(11.614,6.094);
\draw[gp path] (2.431,6.519)--(2.611,6.519);
\draw[gp path] (11.704,6.519)--(11.524,6.519);
\node[gp node right] at (2.166,6.519) {$10^{-4}$};
\draw[gp path] (2.431,6.944)--(2.521,6.944);
\draw[gp path] (11.704,6.944)--(11.614,6.944);
\draw[gp path] (2.431,7.369)--(2.611,7.369);
\draw[gp path] (11.704,7.369)--(11.524,7.369);
\node[gp node right] at (2.166,7.369) {$10^{-2}$};
\draw[gp path] (2.431,7.794)--(2.521,7.794);
\draw[gp path] (11.704,7.794)--(11.614,7.794);
\draw[gp path] (2.431,8.219)--(2.611,8.219);
\draw[gp path] (11.704,8.219)--(11.524,8.219);
\node[gp node right] at (2.166,8.219) {$10^{0}$};
\draw[gp path] (2.431,1.420)--(2.431,1.600);
\draw[gp path] (2.431,8.219)--(2.431,8.039);
\node[gp node center] at (2.431,0.976) {$0$};
\draw[gp path] (3.590,1.420)--(3.590,1.600);
\draw[gp path] (3.590,8.219)--(3.590,8.039);
\node[gp node center] at (3.590,0.976) {$20$};
\draw[gp path] (4.749,1.420)--(4.749,1.600);
\draw[gp path] (4.749,8.219)--(4.749,8.039);
\node[gp node center] at (4.749,0.976) {$40$};
\draw[gp path] (5.908,1.420)--(5.908,1.600);
\draw[gp path] (5.908,8.219)--(5.908,8.039);
\node[gp node center] at (5.908,0.976) {$60$};
\draw[gp path] (7.068,1.420)--(7.068,1.600);
\draw[gp path] (7.068,8.219)--(7.068,8.039);
\node[gp node center] at (7.068,0.976) {$80$};
\draw[gp path] (8.227,1.420)--(8.227,1.600);
\draw[gp path] (8.227,8.219)--(8.227,8.039);
\node[gp node center] at (8.227,0.976) {$100$};
\draw[gp path] (9.386,1.420)--(9.386,1.600);
\draw[gp path] (9.386,8.219)--(9.386,8.039);
\node[gp node center] at (9.386,0.976) {$120$};
\draw[gp path] (10.545,1.420)--(10.545,1.600);
\draw[gp path] (10.545,8.219)--(10.545,8.039);
\node[gp node center] at (10.545,0.976) {$140$};
\draw[gp path] (11.704,1.420)--(11.704,1.600);
\draw[gp path] (11.704,8.219)--(11.704,8.039);
\node[gp node center] at (11.704,0.976) {$160$};
\draw[gp path] (2.431,8.219)--(2.431,1.420)--(11.704,1.420)--(11.704,8.219)--cycle;
\node[gp node center,rotate=-270] at (0.354,4.819) {maximum error};
\node[gp node center] at (7.067,0.310) {$m=2n+1$};
\node[gp node right] at (9.669,2.612) {SE-Sinc};
\gpsetdashtype{gp dt 2}
\draw[gp path] (9.934,2.612)--(11.174,2.612);
\draw[gp path] (2.721,7.770)--(2.953,7.610)--(3.184,7.525)--(3.416,7.470)--(3.648,7.431)%
  --(3.880,7.400)--(4.112,7.375)--(4.344,7.353)--(4.575,7.334)--(4.807,7.318)--(5.039,7.303)%
  --(5.271,7.290)--(5.503,7.278)--(5.735,7.267)--(5.966,7.256)--(6.198,7.247)--(6.430,7.238)%
  --(6.662,7.229)--(6.894,7.221)--(7.125,7.213)--(7.357,7.206)--(7.589,7.198)--(7.821,7.192)%
  --(8.053,7.185)--(8.285,7.179)--(8.516,7.173)--(8.748,7.167)--(8.980,7.162)--(9.212,7.157)%
  --(9.444,7.152)--(9.676,7.147)--(9.907,7.142)--(10.139,7.138)--(10.371,7.134)--(10.603,7.129)%
  --(10.835,7.125)--(11.066,7.122)--(11.298,7.118)--(11.530,7.114)--(11.704,7.111);
\gpsetpointsize{4.00}
\gppoint{gp mark 9}{(2.721,7.770)}
\gppoint{gp mark 9}{(2.953,7.610)}
\gppoint{gp mark 9}{(3.184,7.525)}
\gppoint{gp mark 9}{(3.416,7.470)}
\gppoint{gp mark 9}{(3.648,7.431)}
\gppoint{gp mark 9}{(3.880,7.400)}
\gppoint{gp mark 9}{(4.112,7.375)}
\gppoint{gp mark 9}{(4.344,7.353)}
\gppoint{gp mark 9}{(4.575,7.334)}
\gppoint{gp mark 9}{(4.807,7.318)}
\gppoint{gp mark 9}{(5.039,7.303)}
\gppoint{gp mark 9}{(5.271,7.290)}
\gppoint{gp mark 9}{(5.503,7.278)}
\gppoint{gp mark 9}{(5.735,7.267)}
\gppoint{gp mark 9}{(5.966,7.256)}
\gppoint{gp mark 9}{(6.198,7.247)}
\gppoint{gp mark 9}{(6.430,7.238)}
\gppoint{gp mark 9}{(6.662,7.229)}
\gppoint{gp mark 9}{(6.894,7.221)}
\gppoint{gp mark 9}{(7.125,7.213)}
\gppoint{gp mark 9}{(7.357,7.206)}
\gppoint{gp mark 9}{(7.589,7.198)}
\gppoint{gp mark 9}{(7.821,7.192)}
\gppoint{gp mark 9}{(8.053,7.185)}
\gppoint{gp mark 9}{(8.285,7.179)}
\gppoint{gp mark 9}{(8.516,7.173)}
\gppoint{gp mark 9}{(8.748,7.167)}
\gppoint{gp mark 9}{(8.980,7.162)}
\gppoint{gp mark 9}{(9.212,7.157)}
\gppoint{gp mark 9}{(9.444,7.152)}
\gppoint{gp mark 9}{(9.676,7.147)}
\gppoint{gp mark 9}{(9.907,7.142)}
\gppoint{gp mark 9}{(10.139,7.138)}
\gppoint{gp mark 9}{(10.371,7.134)}
\gppoint{gp mark 9}{(10.603,7.129)}
\gppoint{gp mark 9}{(10.835,7.125)}
\gppoint{gp mark 9}{(11.066,7.122)}
\gppoint{gp mark 9}{(11.298,7.118)}
\gppoint{gp mark 9}{(11.530,7.114)}
\gppoint{gp mark 9}{(10.554,2.612)}
\node[gp node right] at (9.669,1.937) {DE-Sinc};
\gpsetdashtype{gp dt solid}
\draw[gp path] (9.934,1.937)--(11.174,1.937);
\draw[gp path] (2.721,7.886)--(2.953,7.698)--(3.184,7.584)--(3.416,7.856)--(3.648,8.141)%
  --(3.880,8.141)--(4.112,8.141)--(4.344,8.141)--(4.575,8.141)--(4.807,8.141)--(5.039,8.141)%
  --(5.271,8.141)--(5.503,8.141)--(5.735,8.141)--(5.966,8.141)--(6.198,8.141)--(6.430,8.141)%
  --(6.662,8.141)--(6.894,8.141)--(7.125,8.141)--(7.357,8.141)--(7.589,8.141)--(7.821,8.141)%
  --(8.053,8.141)--(8.285,8.141)--(8.516,8.141)--(8.748,8.141)--(8.980,8.141)--(9.212,8.141)%
  --(9.444,8.141)--(9.676,8.141)--(9.907,8.141)--(10.139,8.141)--(10.371,8.141)--(10.603,8.141)%
  --(10.835,8.141)--(11.066,8.141)--(11.298,8.141)--(11.530,8.141)--(11.704,8.141);
\gppoint{gp mark 7}{(2.721,7.886)}
\gppoint{gp mark 7}{(2.953,7.698)}
\gppoint{gp mark 7}{(3.184,7.584)}
\gppoint{gp mark 7}{(3.416,7.856)}
\gppoint{gp mark 7}{(3.648,8.141)}
\gppoint{gp mark 7}{(3.880,8.141)}
\gppoint{gp mark 7}{(4.112,8.141)}
\gppoint{gp mark 7}{(4.344,8.141)}
\gppoint{gp mark 7}{(4.575,8.141)}
\gppoint{gp mark 7}{(4.807,8.141)}
\gppoint{gp mark 7}{(5.039,8.141)}
\gppoint{gp mark 7}{(5.271,8.141)}
\gppoint{gp mark 7}{(5.503,8.141)}
\gppoint{gp mark 7}{(5.735,8.141)}
\gppoint{gp mark 7}{(5.966,8.141)}
\gppoint{gp mark 7}{(6.198,8.141)}
\gppoint{gp mark 7}{(6.430,8.141)}
\gppoint{gp mark 7}{(6.662,8.141)}
\gppoint{gp mark 7}{(6.894,8.141)}
\gppoint{gp mark 7}{(7.125,8.141)}
\gppoint{gp mark 7}{(7.357,8.141)}
\gppoint{gp mark 7}{(7.589,8.141)}
\gppoint{gp mark 7}{(7.821,8.141)}
\gppoint{gp mark 7}{(8.053,8.141)}
\gppoint{gp mark 7}{(8.285,8.141)}
\gppoint{gp mark 7}{(8.516,8.141)}
\gppoint{gp mark 7}{(8.748,8.141)}
\gppoint{gp mark 7}{(8.980,8.141)}
\gppoint{gp mark 7}{(9.212,8.141)}
\gppoint{gp mark 7}{(9.444,8.141)}
\gppoint{gp mark 7}{(9.676,8.141)}
\gppoint{gp mark 7}{(9.907,8.141)}
\gppoint{gp mark 7}{(10.139,8.141)}
\gppoint{gp mark 7}{(10.371,8.141)}
\gppoint{gp mark 7}{(10.603,8.141)}
\gppoint{gp mark 7}{(10.835,8.141)}
\gppoint{gp mark 7}{(11.066,8.141)}
\gppoint{gp mark 7}{(11.298,8.141)}
\gppoint{gp mark 7}{(11.530,8.141)}
\gppoint{gp mark 7}{(10.554,1.937)}
\draw[gp path] (2.431,8.219)--(2.431,1.420)--(11.704,1.420)--(11.704,8.219)--cycle;
\gpdefrectangularnode{gp plot 1}{\pgfpoint{2.431cm}{1.420cm}}{\pgfpoint{11.704cm}{8.219cm}}
\end{tikzpicture}
}
\caption{Numerical results for Example~\ref{ex:9}.}
\label{fig:ex9}
\end{figure}

\section{Concluding remarks}
\label{sec:conclusion}

For the indefinite convolution~\eqref{eq:p-x},
Stenger~\cite{stenger95:_colloc}
derived an excellent approximation formula~\eqref{eq:approx-p-SE}.
He also gave the convergence analysis (Theorem~\ref{thm:SE-old}),
which claims that his formula may attain root-exponential convergence:
$\OO(\exp(-c\sqrt{m}))$.
However, the convergence theorem
relies on two nontrivial assumptions.
The first one is Assumption~\ref{assump:SE1}.
In this assumption, $\hat{f}$ is assumed to be analytic on $\Omega^{+}$,
which is not satisfied in several standard examples,
such as $f(x)=\E^{x}$ $(\hat{f}(s)=1/(s-1))$.
Furthermore, it is assumed that $\aSEm$ is diagonalizable,
and $\sigma(\aSEm)\subseteq\Omega^{+}$.
The latter assumption ($\sigma(\aSEm)\subseteq\Omega^{+}$)
is known as Stenger's conjecture,
which has not strictly been proved thus far.
This assumption is made to ensure that $F(\aSEm)$ is well-defined.
The second one is Assumption~\ref{assump:SE2}.
In this assumption, $P(v,x)$ is assumed to belong to
$\MC_{\alpha,\beta}(\SEt(\domD_{d+\epsilon}))$,
and the parameters $\alpha$, $\beta$ and $d$ are used
for implementation to select $h$ by~\eqref{eq:h-SE}
and $M$ and $N$ by~\eqref{eq:MN-SE}.
However, those parameters are not easy to obtain
because $P(v,x)$ is not explicitly known.

The first contribution of this paper is to eliminate
the two assumptions. Instead of Assumption~\ref{assump:SE1},
this study makes Assumption~\ref{assump:F},
which is more amenable to practical applications.
Under Assumption~\ref{assump:F}, the well-definedness of $F(\aSEm)$
is established (Lemma~\ref{lem:SE-F-welldef}).
This result reveals that Stenger's conjecture is not required
to prove the well-definedness of $F(\aSEm)$.
On a relevant note, this study reveals that
$\sigma(\aSEm)\to \{0\}$ as $m\to\infty$ (Corollary~\ref{cor:SE-new}).
Furthermore, eliminating Assumption~\ref{assump:SE2},
this study establishes Theorem~\ref{thm:SE-new}.
This elimination is fairly useful for a practical use.

It should be noted that the results of this paper
rely on the finiteness of the interval $[a, b]$,
whereas Stenger's framework accommodates potentially infinite intervals.
Extension to infinite intervals is one of future subjects.
Furthermore, as for the convergence theorem,
this paper (Theorem~\ref{thm:SE-new}) assumes that $g$ is bounded,
whereas $g$ may have integrable singularity
in Stenger's error analysis (Theorem~\ref{thm:SE-old}).
Extension to the unbounded case is also one of future subjects.

The second contribution of this paper is to improve
the convergence rate of Stenger's formula. The idea
is to replace the SE transformation with the DE
transformation, which results in the formula~\eqref{eq:approx-p-DE}.
Furthemore, this study provides similar theoretical
results to the case of the SE transformation;
Lemma~\ref{lem:DE-F-welldef}
and Theorem~\ref{thm:DE-new} are established.
As a result, it is shown that
the improved formula may attain nearly exponential
convergence: $\OO(\exp(-\tilde{c}m/\log m))$,
which is significantly higher than root-exponential
convergence.
In fact, such an improvement can be observed
in the numerical results of Examples~\ref{ex:1}--\ref{ex:6}.

Numerical results of Examples~\ref{ex:7}--\ref{ex:9}
offer preliminary insights into the behavior of
the Sinc convolution in contexts beyond the scope of
Theorems~\ref{thm:SE-new} and~\ref{thm:DE-new}.
Notably, effective convergence is still observed in
Examples~\ref{ex:7} and~\ref{ex:8},
suggesting the potential robustness of the method.
A rigorous theoretical framework to explain these
phenomena remains a subject for future work.

As a final remark,
Stenger's beautiful expression of
the indefinite convolution~\eqref{eq:p-F-J-g}
can be utilized not only for the Sinc convolution,
but also for any other numerical methods.
As described in the introduction,
Stenger's conjecture
was formulated for
some approximation formulas~\citep{stenger15:_comput_schr,stenger21:_indef}.
The idea of refinement of theory of
the Sinc convolution presented in this study
may also be applied to those formulas.




\section*{Funding}
This work was supported by JSPS Grant-in-Aid for Scientific Research (C) JP23K03218.

\bibliographystyle{IMANUM-BIB}
\bibliography{sinc-conv-refine}


\end{document}